\documentclass[11pt]{article}
\usepackage[a4paper,margin=28mm]{geometry}
\usepackage[T1]{fontenc}
\usepackage{lmodern}
\usepackage{microtype}
\usepackage{amsmath,amssymb,amsthm,mathtools}
\usepackage{bm}
\usepackage{enumitem} 
\usepackage{booktabs}
\usepackage{xcolor}
\usepackage{aliascnt}
\usepackage{hyperref}
\usepackage[nameinlink,noabbrev]{cleveref}
\hypersetup{hidelinks,
  pdftitle={Global Strong Solutions for the Affine Darken Class of Maxwell--Stefan Diffusion Systems},
  pdfauthor={Dieter Bothe}}
\allowdisplaybreaks

\newcommand{\R}{\mathbb R}

\newcommand{\one}{\mathbf 1}
\newcommand{\diag}{\operatorname{diag}}
\newcommand{\Div}{\operatorname{div}}

\newcommand{\sym}{\operatorname{sym}}
\newcommand{\cS}{\mathcal S}
\newcommand{\cZ}{\mathcal Z}
\newcommand{\cK}{\mathcal K}
\newcommand{\cT}{\mathcal T}
\newcommand{\dd}{\,\mathrm d}
\newcommand{\eps}{\varepsilon}

\DeclareRobustCommand{\MSD}{\mathord{\text{\normalfont\DH}}}

\newtheorem{theorem}{Theorem}[section]
\newaliascnt{lemma}{theorem}
\newtheorem{lemma}[lemma]{Lemma}
\aliascntresetthe{lemma}
\crefname{lemma}{lemma}{lemmas}
\Crefname{lemma}{Lemma}{Lemmas}
\newaliascnt{proposition}{theorem}
\newtheorem{proposition}[proposition]{Proposition}
\aliascntresetthe{proposition}
\crefname{proposition}{proposition}{propositions}
\Crefname{proposition}{Proposition}{Propositions}
\newaliascnt{corollary}{theorem}
\newtheorem{corollary}[corollary]{Corollary}
\aliascntresetthe{corollary}
\crefname{corollary}{corollary}{corollaries}
\Crefname{corollary}{Corollary}{Corollaries}
\newaliascnt{remark}{theorem}
\newtheorem{remark}[remark]{Remark}
\aliascntresetthe{remark}
\crefname{remark}{remark}{remarks}
\Crefname{remark}{Remark}{Remarks}
\newaliascnt{definition}{theorem}

\aliascntresetthe{definition}
\crefname{definition}{definition}{definitions}
\Crefname{definition}{Definition}{Definitions}
\title{\bf Global Strong Solutions for the Affine Darken Class of Maxwell--Stefan Diffusion Systems}
\author{Dieter Bothe\\[2pt]
\normalsize Department of Mathematics, Technische Universit\"at Darmstadt\\
\normalsize Peter-Gr\"unberg-Stra\ss e 10, 64287 Darmstadt, Germany\\
\normalsize \texttt{bothe@mma.tu-darmstadt.de}}
\date{}

\begin{document}
\maketitle

\begin{abstract}
We study the nonlinear cross-diffusion system arising from the multicomponent Darken closure of Maxwell--Stefan diffusion. For general smooth and uniformly positive tracer diffusivities, we derive the constrained flux representation, the mixing-entropy symmetrization and Fisher dissipation, composition-uniform normal ellipticity, and local maximal-$L^p$ strong well-posedness. Our main result concerns the affine Darken class, in which the tracer diffusivities have the form $d_i(x)=\sigma(x)+\tau(x)\delta_i$ for a fixed strictly ordered parameter list. For every $N\ge3$, every compatible strictly positive initial composition in the natural trace space, and without a smallness assumption, we prove global maximal-$L^p$ strong well-posedness on bounded connected $C^{2+\alpha}$ domains in space dimension $d\ge2$, with $p>d+2$. The solution remains uniformly separated from the boundary of the composition simplex, is classical for positive times, and converges exponentially to the homogeneous equilibrium. The proof combines Stieltjes spectral coordinates, entropy-stabilized De Giorgi estimates, quantitative conormal Schauder estimates, and maximal regularity. We also characterize the affine class as the maximal Darken class preserving the Stieltjes eigenframe associated with the prescribed pole list. Constant multiplicative frictions, including repeated factors, are covered by an exact grouping argument.
\end{abstract}



\section{Introduction and main result}
\label{sec:introduction}
Maxwell--Stefan diffusion describes molecular transport in multicomponent gases and liquids through the relative motion of their constituents. Its constitutive equations balance thermodynamic driving forces against the friction generated by relative species motion. In the standard molar formulation, these relations read \cite{Bothe2011,BotheDruet2023}
\begin{equation}\label{eq:MS-standard-intro}
 \boldsymbol{\mathcal F}_i
 =-\sum_{k\ne i}\frac{x_i x_k}{\MSD_{ik}}(v_i-v_k)
 =-\sum_{k\ne i}\frac{x_kJ_i-x_iJ_k}{c\MSD_{ik}},
 \qquad i=1,\ldots,N.
\end{equation}
Here $\boldsymbol{\mathcal F}_i$ denotes the normalized thermodynamic driving force of species $i$, with $\sum_i\boldsymbol{\mathcal F}_i=0$. If $c_i$ is its molar concentration, then $c=\sum_i c_i>0$ is the total molar concentration and $x_i=c_i/c$ its mole fraction. The species velocities $v_i$ define the molar-average velocity $v_{\rm mol}=\sum_k x_kv_k$ and the physical diffusive molar fluxes $J_i=cx_i(v_i-v_{\rm mol})$, which satisfy $\sum_iJ_i=0$. The symmetric pair diffusivities $\MSD_{ik}=\MSD_{ki}>0$ measure inverse frictional resistance; they generally differ from the entries of the Fickian matrix relating fluxes to concentration gradients. This mutual coupling can produce diffusion even when a species' own concentration gradient vanishes, or a flux directed up that gradient.

We now specialize to the setting of this paper: a nonreacting ideal mixture at fixed temperature and pressure, without external forces, with constant total molar concentration $c=c_*>0$ and vanishing molar-average velocity. The molar chemical potentials then satisfy $\mu_i=\mu_i^\circ+RT\log x_i$, where $R$ is the gas constant, $T$ the absolute temperature, and $\mu_i^\circ$ is spatially constant. Consequently the normalized driving forces become $\boldsymbol{\mathcal F}_i=x_i\nabla\mu_i/(RT)=\nabla x_i$. With $j_i=J_i/c_*$, \eqref{eq:MS-standard-intro} gives the normalized Maxwell--Stefan flux relations used below, and the species balances read $\partial_t x_i+\Div j_i=0$ with $\sum_i j_i=0$.

Substituting the constrained fluxes into the species balances gives a nonlinear parabolic cross-diffusion system on the composition simplex, the set of nonnegative mole fractions with unit sum. This ideal diffusion subsystem is the PDE problem studied in the present paper. It retains the multicomponent coupling and the flux constraint and serves as a core problem for global strong well-posedness analysis. The diffusion operator is singular on the full composition space and acts parabolically only on the tangent space of the simplex. Our global continuation argument therefore combines quantitative control of the constrained parabolic operator with exclusion of contact with the simplex boundary. This paper establishes such a large-data global result for an affine class of multicomponent Darken diffusivities. Nonideal thermodynamic driving forces are reserved for subsequent work.

The physical motivation for the Darken relation is to connect collective diffusion to species mobilities. A self-diffusivity describes the long-time spreading of a tagged molecule, whereas a Maxwell--Stefan diffusivity characterizes resistance to the relative motion of species populations. Darken's binary relation \cite{Darken1948}, written as $\MSD_{12}=x_2d_1+x_1d_2$, connects these distinct observables through a constitutive approximation. Liu, Vlugt, and Bardow \cite{LiuVlugtBardow2011} formulated the multicomponent extension used here. It parameterizes pair transport through species coefficients and composition, providing a practical connection to self-diffusion data.

In the continuum-thermodynamic closure of Bothe and Druet \cite[Section~8]{BotheDruet2023}, this multicomponent relation follows exactly from a positive \emph{diagonal core}, with the flux constraint enforced by projection. Core diagonality therefore does not imply independent Fick diffusion: the projected operator is generally coupled. The exact core/Maxwell--Stefan equivalence concerns the constitutive equations; identifying the core entries with independently measured self- or tracer-diffusivities is an additional physical modeling assumption. We use the term \emph{tracer diffusivity} in this model sense; \Cref{rem:darken-identification} makes the interpretation explicit.

The global proof adapts the Stieltjes spectral coordinates and entropy-stabilized Euler--Poisson--Darboux localization developed for Maxwell--Stefan diffusion with additive friction coefficients in \cite{BotheAdditive2026}. Scalar energy and oscillation estimates yield H\"older continuity of the roots. Conormal Schauder estimates for the affine moment system and linear maximal-$L^p$ regularity then provide the uniform terminal control required for global continuation and positive-time smoothing.

For this formulation, let $d_i(x)>0$ denote the core coefficient (model tracer diffusivity) of species $i$ and set
\begin{equation}\label{eq:darken-Dmix-intro}
 S_D(x):=\sum_{k=1}^N\frac{x_k}{d_k(x)},
 \qquad D_{\rm mix}(x):=S_D(x)^{-1}.
\end{equation}
Thus $D_{\rm mix}$ is the mole-fraction-weighted harmonic mean of the tracer diffusivities.  The crossed letter $\MSD$ distinguishes the pair diffusivities from the diagonal matrix $D$ below.  The multicomponent Darken prescription is
\begin{equation}\label{eq:darken-MS-intro}
 \MSD_{ij}(x)
 =\frac{d_i(x)d_j(x)}{D_{\rm mix}(x)}
 =d_i(x)d_j(x)S_D(x),
 \qquad i\ne j,
\end{equation}
so the positive binary frictions in our normalization are
\begin{equation}\label{eq:darken-friction-intro}
 f_{ij}(x)=\frac1{\MSD_{ij}(x)}
 =\frac{D_{\rm mix}(x)}{d_i(x)d_j(x)}.
\end{equation}
A direct inversion proved in \Cref{prop:inverse} gives
\begin{equation}\label{eq:darken-pde-intro}
 \partial_t x=\Div\!\big((I-x\one^{\sf T})D(x)\nabla x\big),
 \qquad D(x)=\diag(d_1(x),\ldots,d_N(x)),
\end{equation}
subject to $\sum_i x_i=1$ and homogeneous Neumann, i.e.\ no-flux, boundary conditions.  Thus the Darken operator is the compression of a diagonal tracer-diffusion operator to the tangent space of the composition simplex.

For general smooth, uniformly positive $d_i(x)$, this structure yields an exact mixing-entropy symmetrization, Fisher dissipation, composition-uniform normal ellipticity, and local maximal-$L^p$ strong well-posedness.  The eigenvalues of the constrained diffusion operator lie between $\min_i d_i(x)$ and $\max_i d_i(x)$.  The global argument uses the following affine restriction.  We call
\begin{equation}\label{eq:affine-darken}
 d_i(x)=\sigma(x)+\tau(x)\delta_i,
 \qquad 0<\delta_1<\delta_2<\cdots<\delta_N,
\end{equation}
the \emph{affine Darken class} associated with the fixed species parameters $\delta_1,\ldots,\delta_N$.  Here $\sigma(x)$ is a common diffusivity offset, while $\tau(x)$ scales the fixed pattern of species differences:
\[
 d_i(x)-d_j(x)=\tau(x)(\delta_i-\delta_j).
\]
The transport coefficients in this model are the tracer diffusivities $d_i$, which are required to remain uniformly positive.  Neither $\sigma$ nor $\tau$ is required to be positive separately, and both may depend on the full composition.  The parameters $\delta_i$ specify the relative species dependence; they are not additional independently prescribed tracer diffusivities.  Section~\ref{sec:affine-darken} discusses the normalization and interpretation of this representation.

On the tangent space $E=\{q:\one^{\sf T}q=0\}$,
\begin{equation}\label{eq:affine-compression-intro}
 (I-x\one^{\sf T})D(x)|_E
 =\sigma(x)I_E+\tau(x)(I-x\one^{\sf T})D_0|_E,
 \qquad D_0=\diag(\delta_1,\ldots,\delta_N).
\end{equation}
At each interior composition, the operator therefore preserves the Stieltjes frame generated by the fixed pole list $\delta=(\delta_1,\ldots,\delta_N)$.  Throughout this paper, a \emph{fixed Stieltjes frame} means that the pole list is fixed; the frame vectors $q_j(x)$ depend on the composition.  \Cref{prop:fixed-frame-characterization} shows that this affine family is the maximal Darken class preserving that prescribed fixed-pole frame.

The constitutive classes considered here satisfy the inclusions
\[
 \begin{aligned}
  \text{general Darken class}
  &\ \supset\ \text{affine Darken class}\\
  &\ \supset\ \left\{
    \begin{gathered}
     \text{constant multiplicative frictions}\\[-2pt]
     \text{with pairwise distinct factors}
    \end{gathered}\right\}.
 \end{aligned}
\]
For general Darken diffusivities we establish entropy identities, uniform normal ellipticity, and local strong well-posedness. For the affine class we prove global strong solvability without a smallness assumption. The final inclusion identifies the constant multiplicative-friction subclass with distinct factors and follows, after relabeling the species, by taking
\begin{equation}\label{eq:Bdelta-intro}
 B_\delta(x):=\sum_{k=1}^N\frac{x_k}{\delta_k},
 \qquad \sigma(x)=0,
 \qquad \tau(x)=B_\delta(x)^{-1}.
\end{equation}
Then $d_i=\delta_i/B_\delta$, and \eqref{eq:darken-MS-intro} gives $\MSD_{ij}=\delta_i\delta_j$ and $f_{ij}=(\delta_i\delta_j)^{-1}$.  Thus constant multiplicatively separated frictions with distinct factors form a constant-Maxwell--Stefan-diffusivity subclass of the affine Darken class.  The all-equal friction array $f_{ij}=f_0>0$ is also included, by taking $\sigma=f_0^{-1}$ and $\tau=0$ with any distinct positive pole list.  Repeated factors that are not all equal do not admit this strict-pole representation; see \Cref{cor:darken-constant-corner}.  Nevertheless, \Cref{app:repeated-factors} proves the same global conclusions for them by an exact grouping argument.  

\subsection{Main theorem}
Let
\[
 \cS=\left\{x\in\R^N:\ x_i\ge0,\ \sum_{i=1}^N x_i=1\right\},
 \qquad \cS^\circ=\{x\in\cS:\ x_i>0\}.
\]
Let $\Omega\subset\R^d$, $d\ge2$, be bounded and connected of class $C^{2+\alpha}$, $0<\alpha<1$, and fix $p>d+2$.  Let $\sigma,\tau$ be $C^3$ on a Euclidean neighborhood of $\cS$, define $d_i$ by \eqref{eq:affine-darken}, and assume
\begin{equation}\label{eq:darken-uniform-bounds}
 0<d_*\le d_i(x)\le d^*<\infty
 \qquad (x\in\cS,\ i=1,\ldots,N).
\end{equation}
With $m=N-1$ put
\[
 X_0=L^p(\Omega;\R^m),\qquad
 X_1=\{v\in W^{2,p}(\Omega;\R^m):\partial_\nu v=0\},
\]
\[
 X_{\gamma,p}=(X_0,X_1)_{1-1/p,p}
 =B^{2-2/p}_{p,p,N}(\Omega;\R^m)
 \hookrightarrow C^{1+\eta}(\overline\Omega;\R^m),
 \qquad 0<\eta<1-\frac{d+2}{p},
\]
where the subscript $N$ denotes the homogeneous Neumann trace condition.  This trace is well defined because $2-2/p>1+1/p$; the stronger hypothesis $p>d+2$ also gives the displayed $C^{1+\eta}$ embedding.  All maximal-regularity spaces below are equipped with the usual intersection norm unless a trace term is written explicitly.  No compatibility condition of order higher than $\partial_\nu x^0=0$ is imposed: the classical regularity conclusion starts at positive times.

\begin{theorem}[Global strong solvability for the affine Darken class]\label{thm:main}
Let $N\ge3$, let $0<\delta_1<\cdots<\delta_N$, and assume \eqref{eq:affine-darken} and \eqref{eq:darken-uniform-bounds}.  Let $x^0\in B^{2-2/p}_{p,p}(\Omega;\R^N)$ satisfy
\begin{equation}\label{eq:initial-data}
 \sum_{i=1}^N x_i^0=1,\qquad
 \inf_{\overline\Omega}x_i^0>0\quad(i=1,\ldots,N),
 \qquad \partial_\nu x_i^0=0.
\end{equation}
Then the Darken diffusion system \eqref{eq:darken-pde-intro}, equivalently the ideal Maxwell--Stefan system \eqref{eq:MS-law} with pair diffusivities \eqref{eq:darken-MS-intro}, has a unique global maximal-$L^p$ strong solution.  In reduced concentration variables $u=(x_1,\ldots,x_{N-1})$,
\[
 u\in W^{1,p}_{\rm loc}([0,\infty);X_0)
 \cap L^p_{\rm loc}([0,\infty);X_1)
 \cap C([0,\infty);X_{\gamma,p}).
\]
The solution remains in a compact subset of $\cS^\circ$.  There is a $\beta\in(0,\alpha)$ such that, for every $\tau_0>0$, the solution belongs uniformly on positive-time unit slabs to $C_{t,x}^{1+\beta/2,2+\beta}$; in particular it is classical on $[\tau_0,\infty)\times\overline\Omega$.  If
\[
 \overline x_i=|\Omega|^{-1}\int_\Omega x_i^0\,\dd x,
\]
then the relative mixing entropy decays exponentially to the homogeneous equilibrium $\overline x$; consequently $x(t)\to\overline x$ exponentially in $L^2(\Omega)$ and, after positive-time smoothing, in $C^1(\overline\Omega)$.
\end{theorem}

The nonlinear continuation issue is the derivation of a priori bounds that remain uniform up to a possible terminal time. The proof first obtains separation and H\"older control in Stieltjes root coordinates, then transfers these estimates to the moment system to obtain uniform conormal regularity and a terminal maximal-regularity trace bound. These estimates exclude both contact with the simplex boundary and divergence of the strong-solution continuation norm.

\begin{corollary}[Constant multiplicative frictions]\label{cor:constant-friction}
Let $N\ge3$, assume the domain, exponent, and initial-data hypotheses of
\Cref{thm:main}, and let $\widehat\delta_1,\ldots,\widehat\delta_N>0$,
not necessarily distinct.  Set
\[
 B_{\widehat\delta}(x)=\sum_{k=1}^N\frac{x_k}{\widehat\delta_k},
 \qquad d_i(x)=\frac{\widehat\delta_i}{B_{\widehat\delta}(x)}.
\]
Then the Darken Maxwell--Stefan diffusivities are constant,
\begin{equation}\label{eq:constant-MS-diffusivities}
 \MSD_{ij}=\widehat\delta_i\widehat\delta_j,
 \qquad f_{ij}=\frac1{\widehat\delta_i\widehat\delta_j},
\end{equation}
and all conclusions of \Cref{thm:main} hold.
\end{corollary}
\begin{proof}
Here $S_D=B_{\widehat\delta}^2$, hence
$\MSD_{ij}=d_i d_jS_D=\widehat\delta_i\widehat\delta_j$.
The function $B_{\widehat\delta}$ is positive on the closed simplex, so
its reciprocal is smooth on a neighborhood and the tracer diffusivities
are uniformly positive and bounded.  If the factors are distinct,
\Cref{thm:main} applies after relabeling, with $\sigma=0$ and
$\tau=B_{\widehat\delta}^{-1}$.  If all factors equal $\eta$, the system
is $x_t=\eta^2\Delta x$ and is also represented by $\sigma=\eta^2$,
$\tau=0$ with any distinct positive pole list.  The remaining cases follow
from \Cref{prop:repeated-factor-extension}, whose proof groups equal
factors and does not pass to a limit in the root gaps.
\end{proof}

\paragraph{Arbitrary positive constant ternary diffusivities.}
For $N=3$, let $a=\MSD_{12}$, $b=\MSD_{13}$, and $c=\MSD_{23}$ be arbitrary positive constants. The positive factors
\[
 \widehat\delta_1=\sqrt{ab/c},\qquad
 \widehat\delta_2=\sqrt{ac/b},\qquad
 \widehat\delta_3=\sqrt{bc/a}
\]
satisfy $\MSD_{ij}=\widehat\delta_i\widehat\delta_j$. Hence \Cref{cor:constant-friction}, including its repeated-factor cases, applies. Under the domain, exponent, and initial-data hypotheses of \Cref{thm:main}, all conclusions of that theorem therefore hold for the ideal ternary Maxwell--Stefan system with arbitrary positive constant pair diffusivities.

\subsection{Relation to existing well-posedness theory}
\label{sec:literature}
The additive-friction analysis \cite{BotheAdditive2026} establishes global strong solutions without a smallness assumption by means of Stieltjes coordinates and entropy-stabilized De Giorgi estimates.  The present paper applies that method to the affine Darken class, in which the compressed diagonal structure occurs in the mobility operator.

The local strong theory for the isothermal Maxwell--Stefan diffusion subsystem is based on normal ellipticity and quasilinear parabolic theory \cite{Bothe2011}.  Herberg, Meyries, Pr\"uss, and Wilke developed a bounded-domain maximal-$L^p$ theory for a reactive mass-based system, including local semiflow, regularization, positivity, and stability of positive equilibria \cite{HerbergMeyriesPrussWilke2017}.  Related strong theories for complete or coupled mixture models include the entropic symmetrization and near-equilibrium results of Giovangigli and collaborators \cite{GiovangigliMassot1998Local,GiovangigliMassot1998Stability,Giovangigli1999,GiovangigliMatuszewski2013}, perturbative global Cauchy theory \cite{BondesanBriant2022}, compressible-flow results \cite{PiaseckiShibataZatorska2019Local,PiaseckiShibataZatorska2019Strong}, and the class-one mixture theories in \cite{BotheDruet2021,BotheDruetIncompressible2021,Druet2022}.  These results are local, perturbative, structurally restricted, or concern larger flow systems; they do not yield the bounded-domain large-data conclusion of \Cref{thm:main} for arbitrary strictly positive initial compositions.

For ternary mixtures with two equal binary diffusivities, Boudin, Grec, and Salvarani \cite{BoudinGrecSalvarani2012} obtained global smooth solutions by a triangular reduction.

For the normalized ideal ternary system with equal molar masses, unit total density, and no reactions, \cite[Theorem~1.1]{BotheTernary2026v1} establishes global strong solvability on bounded connected $C^{2+\alpha_0}$ domains in $d\ge1$ for arbitrary positive constant pair frictions and compatible nonnegative initial compositions. In this common normalized setting, \Cref{cor:constant-friction} independently recovers global strong solvability for strictly positive compatible data in $d\ge2$, including repeated pair diffusivities, by expressing the three pair diffusivities as products of positive species factors.

The general global existence theory for Maxwell--Stefan systems is predominantly at the weak-solution level. Marion and Temam established global weak existence and positivity for a multicomponent reactive-flow model \cite{MarionTemam2015}. For the classical isothermal/isobaric system with positive symmetric constant binary diffusivities, J\"ungel and Stelzer constructed bounded global weak solutions on bounded $C^{1,1}$ domains in $d\le3$ with no-flux boundary conditions \cite[Theorem~1]{JungelStelzer2013}. In the absence of reactions and with positive conserved mass for each species, the weak solutions they construct converge exponentially in $L^1$ \cite[Theorem~2]{JungelStelzer2013}. Reactive and coupled extensions include \cite{DausJungelTang2020,ChenJungel2015,MuchaPokornyZatorska2015,DruetJungel2020,Druet2021,HelmerJungel2021,GeorgiadisJungel2024,HuoJungelTzavaras2024}.  The qualitative theory also includes weak--strong uniqueness \cite{HuoJungelTzavaras2022,GeltnerHeitzingerJungel2026}, renormalized solutions and uniqueness criteria \cite{GeorgiadisKimTzavaras2024}, entropy equality \cite{BerselliGeorgiadisTzavaras2025}, and partial regularity for entropy-structured cross-diffusion systems \cite{BraukhoffRaithelZamponi2022}.  Together, these works establish global weak solvability for broader coefficient classes and describe uniqueness, entropy and regularity properties of weak solutions.

Liu, Vlugt, and Bardow \cite{LiuVlugtBardow2011} introduced the multicomponent prescription as a predictive model for composition-dependent Maxwell--Stefan diffusivities; Bothe and Druet \cite[Section~8]{BotheDruet2023} derived the same structure from continuum thermodynamics and interpreted the core entries as model tracer diffusivities.  Related transport interpretations occur, for example, in polymer mixtures \cite{RonsinHarting2019}.  \Cref{thm:main} establishes bounded-domain, large-data strong well-posedness for the family $d_i(x)=\sigma(x)+\tau(x)\delta_i$, for every $N\ge3$, with full composition dependence of both scalar functions.  Its scope is the ideal, unforced diffusion subsystem with uniformly positive compatible initial data.  It does not assert global strong solvability for arbitrary composition-dependent tracer diffusivities in the Darken closure, for nonideal driving forces, or for coupled flow systems.

\subsection{Structure of the proof and relation to the additive-friction class}
The proof follows the Stieltjes-coordinate and entropy-stabilized De Giorgi method of \cite{BotheAdditive2026}.  For additive frictions $f_{ij}=g_i+g_j$, the compressed diagonal structure occurs in the resistance operator.  In the affine Darken class it occurs in the mobility operator: \eqref{eq:affine-compression-intro} is an affine combination of the identity and the compression generated by the fixed pole matrix $D_0$.  The Stieltjes inverse map and EPD connection identities therefore apply, with modal diffusivities
\begin{equation}\label{eq:alpha-intro}
 \alpha_j(z)=\sigma(x(z))+\tau(x(z))z_j.
\end{equation}
These coefficients may depend on their selected root coordinates.  The EPD production identity remains valid because its divergence-form chain rule introduces no selected-root derivative of the modal coefficient; the argument does not require $\partial_{z_j}\alpha_j=0$.

The proof proceeds in four stages.
\begin{enumerate}[label=(\roman*)]
\item Local strong solvability and the root equations imply an invariant
root rectangle and uniform species separation.
\item Stabilized scalar energy estimates imply H\"older regularity of the
roots and first-order conormal regularity of the affine moment system.
\item Time H\"older and spatial Lipschitz estimates yield the terminal
trace bound and global continuation.
\item Second-order conormal regularity and entropy dissipation imply
classical regularity and exponential convergence to equilibrium.
\end{enumerate}
The $C^1$ regularity on compact strong-solution intervals justifies the tests;
the scalar regularity constants are independent of its norm.  The second conormal estimate is used only after continuation. The local conormal argument uses the unweighted H\"older specialization of the estimates of Dong and Jeon \cite{DongJeon2026}. Appendix~\ref{app:conormal-quantitative} gives the quantitative bounds used in the proof, with explicit dependence on the coefficient modulus. A smooth scalar example there shows that an ellipticity bound alone cannot control the uniform gradient constant.

The Stieltjes geometry, multi-EPD cancellation, scalar iteration, affine-moment identities, and continuation method originate in \cite{BotheAdditive2026}.  Here we derive the Darken inversion, fixed-frame characterization, entropy symmetrization, and root equations, and verify the coefficient bounds, transformed boundary tensor, and frozen operators required for this constitutive law.  Section~\ref{sec:structure-local} establishes the Darken structure and local strong theory.  Section~\ref{sec:stieltjes-block} develops the root coordinates, their invariant rectangle, and the affine moments.  Section~\ref{sec:epd-dg-block} proves the entropy-stabilized localization and scalar regularity estimates.  Section~\ref{sec:continuation-block} establishes continuation, positive-time smoothing, and relaxation.

\section{Multicomponent Darken structure and local strong theory}
\label{sec:structure-local}
\subsection{Darken closure and direct constrained inversion}
\label{sec:inverse}
In the ideal setting specified in the introduction, the Maxwell--Stefan relations \eqref{eq:MS-standard-intro}, with $j_i=J_i/c_*$, and the species balances take the form
\begin{equation}\label{eq:MS-law}
 \partial_t x_i+\Div j_i=0,\qquad
 \sum_{i=1}^N j_i=0,
 \qquad
 \nabla x_i=\sum_{k\ne i}f_{ik}(x)(x_i j_k-x_kj_i),
\end{equation}
with $j_i\cdot\nu=0$ on $\partial\Omega$.  This is the ideal
isothermal/isobaric molar formulation at constant total molar concentration
and with vanishing molar-average velocity, in the normalization of
\cite{BotheAdditive2026}.  For the physical interpretation, let
$c_*>0$ be the dimensional total molar concentration and let the physical
molar masses $M_i^{\rm phys}>0$ be arbitrary constants.  Then
\[
 c_i=c_*x_i,\qquad
 \rho_i=c_*M_i^{\rm phys}x_i,\qquad
 \rho=c_*\sum_i M_i^{\rm phys}x_i.
\]
The corresponding barycentric velocity in this molar-average frame is
\[
 u_{\rm bar}=\frac{\sum_i M_i^{\rm phys}j_i}
 {\sum_i M_i^{\rm phys}x_i},
\]
and for unequal molar masses it need not vanish; likewise $\rho$ need not be
constant.  Thus the results below concern the molar-frame system
\eqref{eq:MS-law}, not the constant-total-mass-density,
zero-barycentric-velocity unequal-mass formulation.  Finally,
$c_*M_i^{\rm phys}\int_\Omega x_i$ is the conserved physical mass of species
$i$, so zero or positive conserved molar amount is equivalent to zero or
positive physical species mass.

Let $d_i(x)>0$ be arbitrary tracer diffusivities and define $S_D,D_{\rm mix}$ as in \eqref{eq:darken-Dmix-intro}.  The multicomponent Darken closure is \eqref{eq:darken-MS-intro}, equivalently
\begin{equation}\label{eq:darken-friction}
 f_{ij}(x)=\frac1{S_D(x)d_i(x)d_j(x)}
 =\frac{D_{\rm mix}(x)}{d_i(x)d_j(x)}.
\end{equation}
This is the multicomponent Darken relation derived from the core-diagonal closure in \cite[Section~8]{BotheDruet2023}, written in the present molar normalization.

\begin{proposition}[Direct Darken flux]\label{prop:inverse}
Under \eqref{eq:darken-friction}, the Maxwell--Stefan law \eqref{eq:MS-law} is equivalent to
\begin{equation}\label{eq:flux-core}
 j_i=x_i\sum_{k=1}^N d_k(x)\nabla x_k-d_i(x)\nabla x_i,
\end{equation}
and therefore to
\begin{equation}\label{eq:full-system}
 \partial_t x=\Div(A(x)\nabla x),\qquad
 A(x)=(I-x\one^{\sf T})D(x),
 \qquad D(x)=\diag(d_1(x),\ldots,d_N(x)).
\end{equation}
Moreover $j_i\cdot\nu=0$ for all $i$ is equivalent to $\partial_\nu x_i=0$ for all $i$.
\end{proposition}
\begin{proof}
Put
\[
 S_D=\sum_k\frac{x_k}{d_k},
 \qquad q=\sum_k\frac{j_k}{d_k}.
\]
Using \eqref{eq:darken-friction}, the Maxwell--Stefan relation gives
\[
 \nabla x_i
 =\frac1{S_Dd_i}\left(x_iq-S_Dj_i\right).
\]
Hence
\[
 j_i=\frac{x_i}{S_D}q-d_i\nabla x_i.
\]
Summing over $i$ and using $\sum_i j_i=0$ yields
$q=S_D\sum_i d_i\nabla x_i$, which proves \eqref{eq:flux-core} and \eqref{eq:full-system}.

Conversely, set $G:=\sum_k d_k\nabla x_k$ and define
$j_i:=x_iG-d_i\nabla x_i$.  Since $\sum_i x_i=1$ and
$\sum_i\nabla x_i=0$, one has $\sum_i j_i=0$ and
\[
 q:=\sum_i\frac{j_i}{d_i}=S_DG-\sum_i\nabla x_i=S_DG.
\]
Consequently,
\[
 \sum_{k\ne i}f_{ik}(x_i j_k-x_kj_i)
 =\frac{x_iq-S_Dj_i}{S_Dd_i}
 =\nabla x_i,
\]
which recovers the Maxwell--Stefan law in \eqref{eq:MS-law}.

For the boundary condition, if $j_i\cdot\nu=0$ then
$d_i\partial_\nu x_i=x_iQ_\nu$, where
$Q_\nu=\sum_kd_k\partial_\nu x_k$.  Since $\sum_i\partial_\nu x_i=0$,
\[
 0=Q_\nu\sum_i\frac{x_i}{d_i}=S_DQ_\nu.
\]
Thus $Q_\nu=0$ and every $\partial_\nu x_i$ vanishes.  The converse follows directly from \eqref{eq:flux-core}.
\end{proof}

\begin{remark}[Core-diagonal closure and the multicomponent Darken equation]\label{rem:darken-identification}
Bothe and Druet show in \cite[Section~8]{BotheDruet2023} that a diagonal core $\diag(d_1,\ldots,d_N)$ yields the Maxwell--Stefan diffusivities
$\MSD_{ij}=d_id_j\sum_kx_k/d_k$. Thus \eqref{eq:full-system} is an exact PDE representation of this constitutive closure.

The tracer interpretation can be seen by splitting species $i$ into dynamically identical labeled and unlabeled populations. Extend the closure by assigning both copies the same coefficient $d_i(x)$, depending on the unsplit composition, and leave the other core entries unchanged. If $\theta$ is the labeled fraction, \eqref{eq:flux-core} then gives, in the same normalization,
\[
 j_{i^*}=\theta j_i-x_i d_i(x)\nabla\theta.
\]
The label balance $\partial_t(x_i\theta)+\Div j_{i^*}=0$ and the species balance $\partial_t x_i+\Div j_i=0$ give
\[
 x_i\,\partial_t\theta+j_i\cdot\nabla\theta
 =\Div\bigl(x_i d_i(x)\nabla\theta\bigr).
\]
Where $x_i>0$, the no-flux conditions for the species and its label give $\partial_\nu\theta=0$. In a homogeneous stationary host with $x_i>0$ and $j_i=0$, the label satisfies $\partial_t\theta=d_i\Delta\theta$. This calculation assigns a tracer role to $d_i$ under the specified extension of the closure. Its agreement with measured self- or tracer-diffusivities remains a constitutive hypothesis.
\end{remark}

\subsection{The affine Darken class and fixed-frame characterization}
\label{sec:affine-darken}
For the global theory, assume \eqref{eq:affine-darken}.  Put
\[
 D_0=\diag(\delta_1,\ldots,\delta_N),
 \qquad P_x:=I-x\one^{\sf T},
 \qquad E=\{q\in\R^N:\one^{\sf T}q=0\}.
\]
Then on $E$,
\begin{equation}\label{eq:affine-compression}
 A(x)|_E=\sigma(x)I_E+\tau(x)A_0(x)|_E,
 \qquad A_0(x):=P_xD_0.
\end{equation}
\paragraph{Interpretation, positivity, and normalization.}
The affine dependence in \eqref{eq:affine-darken} is a constitutive restriction on the tracer-diffusivity vector, not an additional consequence of the Darken relation.  For fixed poles, the coefficients are recovered from the endpoint tracer diffusivities by
\[
 \tau(x)=\frac{d_N(x)-d_1(x)}{\delta_N-\delta_1},
 \qquad
 \sigma(x)=\frac{\delta_Nd_1(x)-\delta_1d_N(x)}{\delta_N-\delta_1}.
\]
Equivalently, every intermediate tracer diffusivity is a fixed convex combination of the endpoint values:
\[
 d_i(x)=\frac{\delta_N-\delta_i}{\delta_N-\delta_1}\,d_1(x)
       +\frac{\delta_i-\delta_1}{\delta_N-\delta_1}\,d_N(x).
\]
Thus \eqref{eq:darken-uniform-bounds} holds for all species if and only if it holds for $d_1$ and $d_N$. The theorem allows $\tau$ to vanish or change sign. At an interior composition with $\tau(x)=0$, all tracer diffusivities coincide, while the fixed distinct poles still determine the Stieltjes coordinates.

For example, consider four species and express the tracer diffusivities in fixed units. Take $\delta=(1,2,3,4)$, $\sigma=2$, and $\tau=(x_1-x_2)/4$. Then $1\le d_i\le3$ on the closed simplex. The ordering of the tracer diffusivities reverses as the relative abundance shifts from $x_1>x_2$ to $x_1<x_2$, with a common diffusivity on $x_1=x_2$. The pair diffusivities also vary with composition. For example, $\MSD_{12}=2$ on $x_1=x_2$, while its continuous extension to the first pure-species vertex equals $d_2=5/2$.

The origin and scale of the pole parameters are not intrinsic.  For constants $a\in\R$ and $b>0$ with $a+b\delta_1>0$, the change
\[
 \widetilde\delta_i=a+b\delta_i,
 \qquad \widetilde\tau=\tau/b,
 \qquad \widetilde\sigma=\sigma-a\tau/b
\]
leaves every $d_i$ unchanged.  Consequently, the offset $\sigma$ and the scale $\tau$ do not, without a chosen normalization and further constitutive information, have separate microscopic interpretations.  In dimensional variables, $[\sigma]=[d_i]=\mathsf L^2\mathsf T^{-1}$ and $[\tau][\delta_i]=\mathsf L^2\mathsf T^{-1}$, where $\mathsf L$ and $\mathsf T$ denote length and time.  In the constant-factor representation $\MSD_{ij}=\delta_i\delta_j$, the factors $\delta_i$ have dimensions $\mathsf L\mathsf T^{-1/2}$, and $\tau=B_\delta^{-1}$ has the same dimensions.

The following proposition characterizes the tracer-diffusivity vectors for which the prescribed Stieltjes frame diagonalizes the mobility.

\begin{proposition}[Characterization of the fixed Stieltjes frame]\label{prop:fixed-frame-characterization}
Fix $x\in\cS^\circ$ and distinct positive numbers $\delta_1,\ldots,\delta_N$.  Let $z_j$ be a Stieltjes root associated with $D_0$ and let
\[
 (q_j)_i=-\frac{x_i}{\delta_i-z_j}.
\]
For a positive diagonal matrix $D=\diag(d_1,\ldots,d_N)$ set $A_D=P_xD$.  Then the following are equivalent:
\begin{enumerate}[label=(\roman*)]
 \item $q_j$ is an eigenvector of $A_D$ for one (and hence every) Stieltjes index $j$;
 \item there exist scalars $\sigma,\tau$ such that
       $d_i=\sigma+\tau\delta_i$ for every $i$.
\end{enumerate}
In this case all $q_j$ are eigenvectors and
\begin{equation}\label{eq:affine-modal-pointwise}
 A_Dq_j=(\sigma+\tau z_j)q_j.
\end{equation}
Consequently, among diagonal Darken cores, the state-dependent family \eqref{eq:affine-darken} is exactly the maximal one retaining, at every interior composition, the Stieltjes frame generated by the prescribed fixed species parameters $\delta_i$. The frame depends on $x$ through the associated root vectors.
\end{proposition}
\begin{proof}
If $d_i=\sigma+\tau\delta_i$, then on $E$ one has
$A_D=\sigma I_E+\tau A_0$, and \eqref{eq:affine-modal-pointwise} follows from $A_0q_j=z_jq_j$.

Conversely, suppose $A_Dq_j=\alpha q_j$ for one $j$ and set
$s:=\one^{\sf T}Dq_j$.  Componentwise,
\[
 (d_i-\alpha)(q_j)_i=x_is.
\]
Since $x_i>0$ and $(q_j)_i=-x_i/(\delta_i-z_j)$,
\[
 d_i=\alpha+s z_j-s\delta_i.
\]
Thus $d_i=\sigma+\tau\delta_i$ with
$\sigma=\alpha+s z_j$ and $\tau=-s$, independently of $i$.
\end{proof}

\begin{corollary}[The constant-friction subclass of Darken]\label{cor:darken-constant-corner}
Let $\widehat\delta_i>0$, not necessarily distinct, put
$B_{\widehat\delta}(x)=\sum_kx_k/\widehat\delta_k$, and set
$d_i=\widehat\delta_i/B_{\widehat\delta}$.  Then the pair diffusivities are
$\MSD_{ij}=\widehat\delta_i\widehat\delta_j$, and the frictions are constant.
Conversely, for $N\ge3$, every positive core-diagonal/Darken closure with composition-independent pair frictions has this form, with uniquely determined positive multiplicative factors $\widehat\delta_i$.
The global theorem applies directly when these factors are pairwise distinct, after relabeling, and also in the all-equal case.  Repeated factors that are not all equal are not covered by the strict-pole representation \eqref{eq:affine-darken}, but they satisfy the global conclusions by the grouping proof in \Cref{app:repeated-factors}.
\end{corollary}
\begin{proof}
For the first assertion, $S_D=B_{\widehat\delta}^2$, so
$\MSD_{ij}=d_i d_jS_D=\widehat\delta_i\widehat\delta_j$.
For the converse, \eqref{eq:darken-friction} shows that at every interior composition the friction array is multiplicatively separated.  If it is independent of the composition, then for distinct $i,j,k$ the triple-ratio identity
\[
 \beta_i^2=\frac{f_{ij}f_{ik}}{f_{jk}}
\]
defines unique positive constants $\beta_i$ with $f_{ij}=\beta_i\beta_j$.
Writing $\widehat\delta_i=\beta_i^{-1}$ and comparing \eqref{eq:darken-friction} gives
$d_i=s\widehat\delta_i$ with a common scalar $s(x)>0$.  The normalization
$S_D=\sum_kx_k/d_k$ and $f_{ij}=1/(\widehat\delta_i\widehat\delta_j)$ yields
$s=B_{\widehat\delta}^{-1}$.
If some factors coincide but not all do, then some tracer values coincide while others differ.  A representation with distinct poles would give
$d_i-d_j=\tau(\delta_i-\delta_j)$; equality of two tracer values would force $\tau=0$ and hence equality of all tracer values, a contradiction.  The all-equal case is instead represented with $\tau=0$, as in \Cref{cor:constant-friction}.
\end{proof}

\subsection{Reduced concentration system and entropy symmetrization}
\label{sec:concentration-geometry}
In this subsection and in \Cref{sec:local}, the affine assumption is not needed: the $d_i$ may be arbitrary positive Darken tracer diffusivities.
Put
\[
 u=(x_1,\ldots,x_m)^{\sf T},\qquad
 x_N=1-\one^{\sf T}u,
\]
where now $\one\in\R^m$.  Define
\[
 D'=\diag(d_1,\ldots,d_m),\qquad
 r=(d_1-d_N,\ldots,d_m-d_N)^{\sf T}.
\]
Since
\[
 \sum_{k=1}^Nd_k\nabla x_k
 =\sum_{k=1}^m(d_k-d_N)\nabla x_k=r^{\sf T}\nabla u,
\]
the reduced system is
\begin{equation}
 u_t-\Div(C(u)\nabla u)=0,
 \qquad
 C(u)=D'(u)-u r(u)^{\sf T}.
 \label{eq:reduced-system}
\end{equation}

\begin{lemma}[Determinant and inverse]
\label{lem:det-inverse}
For every $x\in\cS$ for which $d_i>0$,
\begin{equation}
 \det C(u)
 =\left(\prod_{i=1}^Nd_i\right)
 \left(\sum_{i=1}^N\frac{x_i}{d_i}\right)>0.
 \label{eq:detC}
\end{equation}
Moreover
\begin{equation}
 C^{-1}
 =(D')^{-1}
 +\frac{(D')^{-1}u\,r^{\sf T}(D')^{-1}}
 {d_N\sum_{i=1}^Nx_i/d_i}.
 \label{eq:Cinverse}
\end{equation}
\end{lemma}
\begin{proof}
The matrix determinant lemma gives
\[
 \det C=\det D'\bigl(1-r^{\sf T}(D')^{-1}u\bigr).
\]
Now
\[
 1-r^{\sf T}(D')^{-1}u
 =x_N+d_N\sum_{i=1}^m\frac{x_i}{d_i}
 =d_N\sum_{i=1}^N\frac{x_i}{d_i},
\]
which proves \eqref{eq:detC}.  Formula \eqref{eq:Cinverse} is the Sherman--Morrison formula with the same denominator.
\end{proof}

Let
\[
 h(x):=\sum_{i=1}^N x_i\log x_i,\qquad x\in\cS^\circ,
\]
be the mixing entropy on the composition simplex.  With
$\iota(u):=(u^{\sf T},1-\one^{\sf T}u)^{\sf T}$, its expression in reduced
concentration coordinates is
\begin{equation}
 h_u(u):=h(\iota(u))
 =\sum_{i=1}^m x_i\log x_i+x_N\log x_N.
 \label{eq:mixing-entropy}
\end{equation}
The subscripts in $h_u$, $h_M$, and $h_z$ indicate the reduced concentration,
moment, and root representations of the same entropy, respectively; the
latter two are introduced below.  Derivatives of these densities are
denoted by $D$ or $\partial$.
The reduced entropy Hessian is
\begin{equation}
 H(u):=D_u^2h_u(u)
 =\diag(1/x_1,\ldots,1/x_m)+\frac1{x_N}\one\one^{\sf T}.
 \label{eq:H}
\end{equation}

\begin{proposition}[Exact core-diagonal entropy symmetrizer]
\label{prop:HC}
For every $u$ in the open reduced simplex,
\begin{equation}
 H(u)C(u)
 =\diag\left(\frac{d_1}{x_1},\ldots,\frac{d_m}{x_m}\right)
 +\frac{d_N}{x_N}\one\one^{\sf T}.
 \label{eq:HC}
\end{equation}
In particular
\[
 HC=C^{\sf T}H
\]
and for every $\xi\in\R^m\setminus\{0\}$,
\begin{equation}
 \xi^{\sf T}HC\xi
 =\sum_{i=1}^m\frac{d_i}{x_i}\xi_i^2
 +\frac{d_N}{x_N}\left(\sum_{i=1}^m\xi_i\right)^2>0.
 \label{eq:HCcoercive}
\end{equation}
\end{proposition}
\begin{proof}
Multiply
\[
 H=\diag(1/x_i)+x_N^{-1}\one\one^{\sf T},
 \qquad C=D'-ur^{\sf T}.
\]
The first product gives
\[
 \diag(d_i/x_i)-\one r^{\sf T}.
\]
The second gives
\[
 x_N^{-1}\one\bigl((d_1,\ldots,d_m)-(1-x_N)r\bigr)^{\sf T}.
\]
After adding $-\one r^{\sf T}$, the row vector multiplying $\one$ reduces to
\[
 x_N^{-1}\bigl((d_1,\ldots,d_m)-r\bigr)^{\sf T}
 =\frac{d_N}{x_N}\one^{\sf T}.
\]
This proves \eqref{eq:HC}; symmetry and positivity are immediate.
\end{proof}

\begin{proposition}[Fisher mobility factorization]
\label{prop:mobility-factorization}
On the open reduced simplex,
\begin{equation}
 H^{-1}=\diag(x_1,\ldots,x_m)-uu^{\sf T}.
 \label{eq:H-inverse}
\end{equation}
Consequently
\begin{equation}
 \mathcal M(u):=C(u)H(u)^{-1}
 \label{eq:phenomenological-mobility}
\end{equation}
is symmetric and positive definite, and the ideal Fick matrix factors as
\begin{equation}
 C=\mathcal M H.
 \label{eq:ideal-mobility-factorization}
\end{equation}
More explicitly, if $\Pi:\R^N\to\R^m$ extracts the first $m$ components,
\[
 \mathcal M(u)=\Pi P_x\diag(x_i d_i)P_x^{\sf T}\Pi^{\sf T}.
\]
The projections encode the flux constraint and generally couple the species. The Fisher Hessian $H$ is the ideal thermodynamic factor. Under \eqref{eq:darken-uniform-bounds},
\[
 d_*H\le HC\le d^*H,\qquad
 d_*H^{-1}\le\mathcal M\le d^*H^{-1}
\]
in the order of symmetric quadratic forms.
\end{proposition}
\begin{proof}
Formula \eqref{eq:H-inverse} follows directly from the Sherman--Morrison
formula, using $x_N=1-\one^{\sf T}u$. Let $R:\R^m\to E$ denote the tangent embedding
$R\xi=(\xi,-\one^{\sf T}\xi)$. The matrices satisfy
$C=\Pi P_x D R$ and
$R H^{-1}=\diag(x_i)P_x^{\sf T}\Pi^{\sf T}$. The latter identity follows from \eqref{eq:H-inverse} in the first $m$ rows and from the zero column sums on both sides. Multiplication gives the projected formula for $\mathcal M$. The symmetry identity in
Proposition~\ref{prop:HC} gives
\[
 \mathcal M^{\sf T}=H^{-1}C^{\sf T}=CH^{-1}=\mathcal M.
\]
For $\xi\ne0$, put $\eta=H^{-1}\xi$.  Then $\xi=H\eta$ and
\[
 \xi^{\sf T}\mathcal M\xi
 =(H\eta)^{\sf T}C\eta
 =\eta^{\sf T}HC\eta>0
\]
by \eqref{eq:HCcoercive}, hence $\mathcal M$ is positive definite. The first comparison follows from \eqref{eq:H} and \eqref{eq:HC}. Congruence by $H^{-1}$ gives the second.
\end{proof}

\begin{corollary}[Entropy/Fisher identity]
\label{cor:entropy}
Every sufficiently regular solution taking values in $\cS^\circ$ satisfies
\begin{equation}
 \frac{\dd}{\dd t}\int_\Omega\sum_{i=1}^Nx_i\log x_i\,\dd x
 +\int_\Omega\sum_{i=1}^Nd_i(x)\frac{|\nabla x_i|^2}{x_i}\,\dd x=0.
 \label{eq:entropy-id}
\end{equation}
Equivalently,
\begin{equation}
 \frac{\dd}{\dd t}\int_\Omega h(x)\,\dd x
 +4\sum_i\int_\Omega d_i(x)|\nabla\sqrt{x_i}|^2\,\dd x=0.
 \label{eq:fisher-id}
\end{equation}
\end{corollary}
\begin{proof}
Test \eqref{eq:reduced-system} with $D_u h_u(u)$.  The production is $\nabla u^{\sf T}HC\nabla u$.  Formula \eqref{eq:HCcoercive}, with $\nabla x_N=-\sum_{i=1}^{m}\nabla x_i$, gives the displayed species sum.
\end{proof}

\begin{lemma}[Uniform physical spectrum]
\label{lem:spectrum-bounds}
For $x\in\cS^\circ$, every eigenvalue $\lambda$ of $C(u)$ is real and satisfies
\begin{equation}
 \min_i d_i(x)\le\lambda\le\max_i d_i(x).
 \label{eq:spectrum-bounds}
\end{equation}
The same inclusion extends continuously to the closed simplex.
\end{lemma}
\begin{proof}
In full tangent variables $E=\{q\in\R^N:\one^{\sf T}q=0\}$, put $X=\diag(x_1,\ldots,x_N)$.  The operator is
\[
 A=(I-x\one^{\sf T})D.
\]
For $q,w\in E$,
\[
 q^{\sf T}X^{-1}Aw=\sum_i d_i\frac{q_iw_i}{x_i}.
\]
Thus $A|_E$ is self-adjoint in the positive metric $X^{-1}$, hence has real spectrum, and an eigenvector $q$ satisfies the Rayleigh formula
\[
 \lambda=\frac{\sum_i d_iq_i^2/x_i}{\sum_iq_i^2/x_i}.
\]
The reduced matrix $C$ is a coordinate representation of $A|_E$, proving \eqref{eq:spectrum-bounds}.  At boundary compositions, $C$ remains continuous and invertible by Lemma~\ref{lem:det-inverse}; the spectral inclusion follows by approximation from the interior.
\end{proof}

These are spectral bounds for the constrained operator.  They do not assert uniform Euclidean coercivity of $\sym C$ on the closed simplex: the Fisher metric becomes singular at a vanishing concentration.  On a compact interior range, by contrast, $H$ and $H^{-1}$ are bounded, and \eqref{eq:HC} supplies the bounded symmetrization used later.

\subsection{Local maximal regularity and continuation alternative}
\label{sec:local}
For the local theory, let $d_1,\ldots,d_N$ be arbitrary $C^3$ functions on a Euclidean neighborhood of $\cS$, satisfying \eqref{eq:darken-uniform-bounds}; no affine dependence on a pole list is assumed.  We first formulate the theory on an open neighborhood of the closed reduced simplex.  The positive-data continuation argument will later take place on compact subsets of $\cS^\circ$.  On every compact subset of $\cS^\circ$, $H$ and $H^{-1}$ are uniformly bounded and Proposition~\ref{prop:HC} shows directly that $H^{1/2}CH^{-1/2}$ is symmetric positive definite.  Lemma~\ref{lem:spectrum-bounds} and \eqref{eq:darken-uniform-bounds} give $\sigma(C)\subset[d_*,d^*]$ on the closed simplex.  Normal ellipticity is an open condition in finite dimensions.  Since the tracer diffusivities $d_i$ are $C^3$ on a Euclidean neighborhood of the simplex, the reduced coefficient map $C(u)$ has the same regularity there.  After shrinking that neighborhood there is an open set $\mathcal U$ in reduced concentration space containing the closed reduced simplex and constants $0<c<C<\infty$, $\vartheta_0<\pi/2$ such that
\[
 \sigma(C(u))\subset\{\zeta:\ c\le|\zeta|\le C,\ |\arg\zeta|\le\vartheta_0\}
 \qquad(u\in\mathcal U).
\]
Thus all matrices on $\mathcal U$ remain uniformly normally elliptic and uniformly invertible.  Their spectra need not remain real outside the simplex; only this uniform sectorial inclusion is used by the local theory.

Expanding the divergence gives
\begin{equation}
 u_t-C(u)\Delta u=F(u,\nabla u),
 \qquad
 F(u,\nabla u)=\sum_{\alpha=1}^dDC(u)[\partial_\alpha u]\,\partial_\alpha u.
 \label{eq:nondiv-u}
\end{equation}
Since $X_{\gamma,p}\hookrightarrow C^{1+\eta}$, on bounded trace-space sets with range in a fixed $\cK\Subset\mathcal U$ the maps
\[
 u\mapsto[-C(u)\Delta]\in\mathcal L(X_1,X_0),\qquad
 u\mapsto F(u,\nabla u)\in X_0
\]
are locally Lipschitz.  The term $F$ arises from the spatial product rule
and is estimated using the $C^1$ embedding of the trace space.

\begin{lemma}[Frozen Neumann complementing condition]
\label{lem:LS}
Let $\mathfrak C$ be a compact family of real matrices with
\[
 \sigma(C_*)\subset\{\zeta:\ c\le|\zeta|\le C,\ |\arg\zeta|\le\vartheta_0\},
 \qquad \vartheta_0<\pi/2.
\]
Then, after choosing a compatible parameter sector, the constant-coefficient Neumann problems
\[
 \lambda v-C_*\Delta v=f,\qquad \partial_\nu v=g,
 \qquad C_*\in\mathfrak C,
\]
are parameter elliptic and satisfy the full parameter-dependent Lopatinskii--Shapiro condition uniformly for $\lambda$ in a fixed sector $\Sigma_{\pi-\vartheta}$ with $\vartheta_0<\vartheta<\pi/2$.
\end{lemma}
\begin{proof}
For the half-space normal ODE at tangential frequency $\xi'$, set
\[
 M(\xi',\lambda)=|\xi'|^2I+\lambda C_*^{-1}.
\]
Choose $\vartheta_0<\vartheta<\pi/2$.  On the compact normalized parameter set
\[
 \xi'\in\R^{d-1},\qquad
 \lambda\in\overline{\Sigma_{\pi-\vartheta}}\cup\{0\},\qquad
 |\xi'|^2+|\lambda|=1,
\]
every eigenvalue $|\xi'|^2+\lambda/\mu$, with $\mu\in\sigma(C_*)$,
stays a positive distance from $(-\infty,0]$, uniformly in
$C_*\in\mathfrak C$.  Indeed the angular bound
$\pi-\vartheta+\vartheta_0<\pi$ excludes cancellation with a nonnegative
real number.  The normalization excludes the zero parameter pair and
includes $\lambda=0$ whenever $\xi'\ne0$.  Hence the principal matrix square root
\[
 R=M^{1/2}
\]
is well defined, holomorphic in the parameters in a uniform complex neighborhood, invertible, and satisfies $\sigma(R)\subset\{\Re z>0\}$.  Every decaying solution is
\[
 v(x_n)=e^{-x_nR}c,
\]
and, since the outward normal of the upper half-space is $-e_n$, the
Neumann datum satisfies
\[
 \partial_\nu v(0)=Rc=g,\qquad c=R^{-1}g.
\]
Uniform estimates follow from compactness of the normalized parameter set
and of $\mathfrak C$.  This matrix-square-root argument verifies the
complementing condition used in \cite{Augner2024}.
\end{proof}

\begin{lemma}[Uniform frozen maximal regularity]
\label{lem:uniform-MR}
Let $\cK\Subset\mathcal U$ and $M<\infty$.  For every
$v\in X_{\gamma,p}$ with
\[
 v(\overline\Omega)\subset\cK,
 \qquad \|v\|_{X_{\gamma,p}}\le M,
\]
the realization
\[
 \mathcal A(v)w:=-C(v(\cdot))\Delta w,
 \qquad D(\mathcal A(v))=X_1,
\]
has maximal $L^p$ regularity on $X_0$.  On every fixed finite horizon the
zero-initial-trace maximal-regularity constant can be chosen uniformly over this family, and the
same zero-trace constant works on all shorter intervals.
\end{lemma}
\begin{proof}
We verify the hypotheses of the bounded-domain theorem of Augner \cite[Assumption~5.1 and
Theorem~5.2]{Augner2024}.  The component space is the finite-dimensional Banach space $E=\mathbb C^m$, hence belongs to Augner's class $\mathcal{HT}$.  In Augner's notation the operator order is $2m_A=2$, hence $m_A=1$.  There is one boundary operator $B_1=\partial_\nu I_E$ and the only nonzero boundary projection is $P_{1,1}=I_E$.  Since
$X_{\gamma,p}\hookrightarrow C^{1+\eta}$, the top-order coefficients
$C(v(x))\delta_{\alpha\beta}$ are uniformly continuous (indeed
$C^{1+\eta}$), while the frozen realization has no lower-order terms.  The
single boundary operator is the fixed first-order Neumann operator
$\partial_\nu I_E$; its coefficients are $C^1$ on a $C^2$ boundary.
Lemma~\ref{lem:LS} verifies parameter ellipticity and the full
parameter-dependent Lopatinskii--Shapiro condition with one angle strictly
below $\pi/2$.  Augner's theorem yields maximal regularity first for a shifted
realization $\mu+\mathcal A(v)$.  On every finite interval this is equivalent
to maximal regularity of the unshifted realization. The following conjugation
shows that no uniform choice of the shift is needed below.  If
\[
 u_t+\mathcal A(v)u=f,\qquad u(0)=0,
\]
and $q(t)=e^{-\mu t}u(t)$, then
\[
 q_t+(\mu+\mathcal A(v))q=e^{-\mu t}f,\qquad q(0)=0.
\]
Conversely $u=e^{\mu t}q$.  Multiplication by $e^{\pm\mu t}$ is an
isomorphism of both $L^p(0,T;X_0)$ and
$W^{1,p}(0,T;X_0)\cap L^p(0,T;X_1)$ on each fixed finite horizon.  Hence every
frozen unshifted realization has maximal $L^p$ regularity.

For uniformity, choose $0<\eta'<\eta$.  The displayed family of $v$ is
relatively compact in $C^{1+\eta'}(\overline\Omega)$, hence the coefficient
fields $C(v)$ form a relatively compact family in the same topology.  Every
coefficient field in its closure still has the same sectorial and boundary
symbol bounds.  Moreover
\[
 \|\mathcal A(v)-\mathcal A(\widetilde v)\|_{\mathcal L(X_1,X_0)}
 \le C\|C(v)-C(\widetilde v)\|_{L^\infty}.
\]
For a fixed horizon $T_0$, maximal regularity means that
$\partial_t+\mathcal A(v)$ is an isomorphism from the zero-trace
maximal-regularity space onto $L^p(0,T_0;X_0)$.  Invertibility is stable in
operator norm; a finite Neumann-series covering of the compact coefficient
closure therefore gives one inverse bound for the whole family.  Extending a
forcing by zero from $(0,T)$ to $(0,T_0)$ and restricting the corresponding
solution shows that the same zero-trace constant works for all $0<T\le T_0$.
\end{proof}

\begin{theorem}[Local strong solution and restart alternative]
\label{thm:local}
Let $N$, $\Omega$, $p$, and the spaces $X_0,X_1,X_{\gamma,p}$ be as in \Cref{sec:introduction}.  Let the $d_i$ be arbitrary $C^3$ tracer diffusivities on a neighborhood of $\cS$ satisfying \eqref{eq:darken-uniform-bounds}, and let $\mathcal U$ be the normally elliptic neighborhood constructed above.  Every datum $u^0\in X_{\gamma,p}$ with range in $\mathcal U$ generates a unique maximal strong solution of \eqref{eq:reduced-system} with range in $\mathcal U$ on $[0,T_{\max})$, in the class
\[
 W^{1,p}_{\rm loc}([0,T_{\max});X_0)
 \cap L^p_{\rm loc}([0,T_{\max});X_1)
 \cap C([0,T_{\max});X_{\gamma,p}).
\]
In particular this applies to the strictly positive compatible simplex data \eqref{eq:initial-data}, without the affine assumption \eqref{eq:affine-darken}.  The solution depends locally Lipschitz continuously on its initial datum in the maximal-regularity and trace norms on every compact subinterval of its existence interval.
For every compact $\cK\Subset\mathcal U$ and $M<\infty$ there is $\tau_*(\cK,M)>0$ such that every compatible datum with range in $\cK$ and $X_{\gamma,p}$ norm at most $M$ generates a solution of length at least $\tau_*$.  Here the domain, exponent, and coefficient law are fixed.  Consequently, if $T_{\max}<\infty$ and the solution range remains in one compact subset of $\mathcal U$, then
\[
 \|u(t)\|_{X_{\gamma,p}}\longrightarrow\infty
 \qquad(t\uparrow T_{\max}).
\]
\end{theorem}
\begin{proof}
We give the quantitative freezing argument because the uniform restart time is
used at the terminal continuation step.  For $J=(0,T)$ set
\[
 \mathbb E_T:=W^{1,p}(J;X_0)\cap L^p(J;X_1),\qquad
 \mathbb E_T^0:=\{h\in\mathbb E_T:h(0)=0\}.
\]
Fix $\cK\Subset\mathcal U$ and $M<\infty$, and choose an open
$\mathcal O$ with $\cK\Subset\mathcal O\Subset\mathcal U$.  Put
$d_0=\operatorname{dist}(\cK,\partial\mathcal O)>0$.  We prove one common existence time
for all compatible data $v_0$ satisfying
\[
 v_0(\overline\Omega)\subset\cK,\qquad
 \|v_0\|_{X_{\gamma,p}}\le M.
\]

Fix $T_0>0$.  Let $u_*^{v_0}$ solve the frozen homogeneous problem
\begin{equation}\label{eq:frozen-reference-local}
 \partial_tu_*^{v_0}+\mathcal A(v_0)u_*^{v_0}=0,
 \qquad u_*^{v_0}(0)=v_0,
\end{equation}
on $(0,T_0)$.  The uniform zero-trace estimate of
Lemma~\ref{lem:uniform-MR}, combined with a fixed bounded coretraction of the
trace $\mathbb E_{T_0}\to X_{\gamma,p}$, gives
\begin{equation}\label{eq:frozen-reference-bound}
 \sup_{v_0}\|u_*^{v_0}\|_{\mathbb E_{T_0}}\le M_{\rm ref}<\infty,
\end{equation}
where the supremum is over the displayed family of initial data, and we take
$M_{\rm ref}\ge1+M$. To obtain the bound, subtract a coretraction of $v_0$ and apply the uniform zero-trace estimate.
The family $\mathcal A(v_0)$ is uniformly bounded in
$\mathcal L(X_1,X_0)$.

The condition $p>d+2$ yields the uniform short-time $C^1$ modulus used below.  Choose
\[
 \frac{1+d/p}{2}<\xi<1-\frac1p,
 \qquad 0<\theta_0<1-\frac1p-\xi.
\]
Then, on a fixed interval,
\[
 \mathbb E_T\hookrightarrow
 C^{\theta_0}\bigl([0,T];(X_0,X_1)_{\xi,1}\bigr)
 \hookrightarrow C^{\theta_0}([0,T];C^1(\overline\Omega)),
\]
because
$(X_0,X_1)_{\xi,1}\hookrightarrow B^{2\xi}_{p,1}(\Omega)
\hookrightarrow C^1(\overline\Omega)$ and $2\xi>1+d/p$.
The standard short-interval extension construction (compare
\cite[Lemma~7.2]{Amann2005}) controls the extension norm by the sum of the
maximal-regularity norm and the initial trace norm, uniformly for
$0<T\le T_0$.  Applying the fixed-horizon embedding to that extension therefore gives
\begin{equation}\label{eq:uniform-short-time-C1}
 \|w(t)-w(s)\|_{C^1}
 \le C|t-s|^{\theta_0}
 \bigl(\|w\|_{\mathbb E_T}+\|w(0)\|_{X_{\gamma,p}}\bigr),
\end{equation}
with $C$ independent of $T\le T_0$; for $w\in\mathbb E_T^0$ the initial-trace
term is absent.  In particular, by \eqref{eq:frozen-reference-bound},
\begin{equation}\label{eq:frozen-reference-close}
 \sup_{v_0}\sup_{0\le t\le T}
 \|u_*^{v_0}(t)-v_0\|_{C^1}
 \le C M_{\rm ref}T^{\theta_0}.
\end{equation}

Write $u=u_*^{v_0}+h$ with $h\in\mathbb E_T^0$.  The equation for $h$ is
\[
 \partial_th+\mathcal A(v_0)h
 =\bigl[\mathcal A(v_0)-\mathcal A(u_*^{v_0}+h)\bigr]
   (u_*^{v_0}+h)
   +F(u_*^{v_0}+h,\nabla(u_*^{v_0}+h)).
\]
Let $\mathcal S_{v_0,T}$ denote the zero-trace solution operator for
$\partial_t+\mathcal A(v_0)$ on $(0,T)$.  Thus the preceding equation defines
\[
 \mathfrak T_{v_0,T}(h)
 :=\mathcal S_{v_0,T}\!\left(
 [\mathcal A(v_0)-\mathcal A(u_*^{v_0}+h)](u_*^{v_0}+h)
 +F(u_*^{v_0}+h,\nabla(u_*^{v_0}+h))\right).
\]
Choose a fixed radius $r>0$.  For
$\|h\|_{\mathbb E_T}\le r$, \eqref{eq:uniform-short-time-C1} and
\eqref{eq:frozen-reference-close} give
\[
 \sup_{0\le t\le T}
 \|u_*^{v_0}(t)+h(t)-v_0\|_{C^1}
 \le C(M_{\rm ref}+r)T^{\theta_0}.
\]
After reducing $T$, this is less than $d_0/2$, so all corresponding functions
take values in $\mathcal O$, uniformly in $v_0$.  On that state set,
$\mathcal A$ and the quadratic map $F$ are uniformly locally Lipschitz.  The
uniform frozen maximal-regularity estimate, the preceding $C^1$ modulus, and
the time integration of the quadratic term give
\begin{align*}
 \|\mathfrak T_{v_0,T}(h)\|_{\mathbb E_T}
 &\le C\bigl(T^{\theta_0}+T^{1/p}\bigr)(M_{\rm ref}+r)^2,\\
 \|\mathfrak T_{v_0,T}(h_1)-\mathfrak T_{v_0,T}(h_2)\|_{\mathbb E_T}
 &\le C\bigl(T^{\theta_0}+T^{1/p}\bigr)(M_{\rm ref}+r)
       \|h_1-h_2\|_{\mathbb E_T},
\end{align*}
where all constants are independent of the admissible datum.  The principal perturbation contains the factor $T^{\theta_0}$ from the coefficient difference
$\|C(u_*^{v_0}+h)-C(v_0)\|_{L^\infty}$ and is multiplied by the
$L^p(0,T;X_1)$ norm of $u_*^{v_0}+h$.  The absolute quadratic forcing contains the factor
$T^{1/p}$ from the common $C^1$ bound, whereas in the difference estimate the
zero initial trace of $h_1-h_2$ gives the additional $T^{\theta_0}$ factor in
its $C^1$ norm.  The contraction constants therefore do not depend on any later a-priori
regularity estimate.  Taking $T>0$ smaller once more makes
$\mathfrak T_{v_0,T}$ a strict contraction of the closed
$r$-ball into itself.  This proves local existence with a common lower bound
$\tau_*(\cK,M)>0$.  The same estimates for the difference of two solutions
with the same trace give uniqueness on a short interval, and iteration gives
uniqueness on every common existence interval.

The same construction gives the asserted continuous dependence on the datum.  If $v_0$ and $\widetilde v_0$ are nearby, the difference of their frozen reference solutions satisfies
\[
 \partial_t w+\mathcal A(v_0)w
 =\bigl(\mathcal A(\widetilde v_0)-\mathcal A(v_0)\bigr)
       u_*^{\widetilde v_0},\qquad
 w(0)=v_0-\widetilde v_0.
\]
Uniform nonzero-trace maximal regularity on the fixed horizon and
\eqref{eq:frozen-reference-bound} give
$\|w\|_{\mathbb E_{T_0}}\le C\|v_0-\widetilde v_0\|_{X_{\gamma,p}}$.
The zero-trace inverses depend locally Lipschitz continuously on the frozen
operator by the resolvent identity for bounded inverses.  The uniform
contraction estimate therefore gives the same dependence for the fixed
points.  Applying the fixed-horizon trace embedding to the same extension gives
\[
 \|w\|_{C([0,T];X_{\gamma,p})}
 \le C\bigl(\|w\|_{\mathbb E_T}+\|w(0)\|_{X_{\gamma,p}}\bigr),
 \qquad 0<T\le T_0,
\]
with a constant independent of $T$.  Applying this estimate to differences
gives the asserted dependence in the continuous trace norm.
Finitely many restarts along a fixed compact segment of the
reference trajectory establish this conclusion on every smaller common
existence interval.

Successive continuation defines the maximal existence interval.  Finally suppose
$T_{\max}<\infty$, the range stays in one compact $\cK\Subset\mathcal U$, but
$\|u(t)\|_{X_{\gamma,p}}\not\to\infty$.  Then there are $M<\infty$ and
$t_n\uparrow T_{\max}$ with $\|u(t_n)\|_{X_{\gamma,p}}\le M$.  The common
lifespan $\tau_*(\cK,M)$ applies to every restarted datum $u(t_n)$; for large
$n$ it extends the solution beyond $T_{\max}$, contradicting maximality.
Thus the stated continuation alternative holds.
\end{proof}

\section{Stieltjes coordinates, root equations, and affine moments}
\label{sec:stieltjes-block}
\subsection{Short-time positivity and validity of root coordinates}
\label{sec:root-entry}
From this section onward we impose the affine Darken assumptions of \Cref{thm:main}, including the distinct fixed poles.
The present section develops the fixed spectral representation used for both the scalar regularity argument and the later system estimates.  We first establish the validity of the root coordinates, then derive the global Stieltjes representation and root equations, and finally introduce the affine moment coordinates used after H\"older regularity has been obtained.

Uniform positivity of the initial concentrations and continuity in the
trace space imply that the local maximal-$L^p$ solution remains in the
open simplex on an initial time interval.  The Stieltjes chart is therefore
well defined on that interval.

\begin{lemma}[Short-time interior persistence]\label{lem:short-positive}
Let $x^0$ satisfy \eqref{eq:initial-data}, and put
\[
 \kappa_0:=\min_{1\le i\le N}\min_{\overline\Omega}x_i^0>0.
\]
For the local solution of \Cref{thm:local} there is $T_0>0$ such that
\[
 x_i(t,y)\ge \frac{\kappa_0}{2}
 \qquad (0\le t\le T_0,\ y\in\overline\Omega,\ i=1,\ldots,N).
\]
Consequently the Stieltjes roots introduced in \Cref{sec:roots} are well defined and belong to a compact root rectangle on $[0,T_0]$.
\end{lemma}
\begin{proof}
Because $p>d+2$, the trace space $X_{\gamma,p}$ embeds into $C^1(\overline\Omega)$ and in particular into $C(\overline\Omega)$.  The local solution is continuous from $[0,T_0]$ into $X_{\gamma,p}$ after $T_0$ is reduced if necessary.  Hence
\[
 \max_i\|x_i(t)-x_i^0\|_{L^\infty}\le \kappa_0/2
\]
for all sufficiently small $t$.  The lower bound follows.  The root map is continuous on $\cS^\circ$ and the initial image of the compact set $x^0(\overline\Omega)$ lies a positive distance from every pole $\delta_i$; continuity then gives a compact root rectangle for a short time.
\end{proof}

This short-time persistence argument establishes the interval on which
the root coordinates can first be introduced.
The invariant root rectangle in \Cref{thm:root-rectangle} subsequently
gives a time-uniform quantitative separation from the simplex boundary.

\subsection{Stieltjes root geometry}
\label{sec:roots}
Let
\[
 D_0=\diag(\delta_1,\ldots,\delta_N),
 \qquad
 A_0(x)=(I-x\one^{\sf T})D_0.
\]
For $x\in\cS^\circ$ define
\begin{equation}
 m_x(z)=\sum_{i=1}^N\frac{x_i}{\delta_i-z}.
 \label{eq:stieltjes}
\end{equation}
Since
\[
 m_x'(z)=\sum_i\frac{x_i}{(\delta_i-z)^2}>0,
\]
there is a unique zero $z_j$ in every gap $(\delta_j,\delta_{j+1})$:
\begin{equation}
 \delta_1<z_1<\delta_2<z_2<\cdots<z_{N-1}<\delta_N.
 \label{eq:interlace}
\end{equation}

\begin{proposition}[Characteristic polynomial and global root chart]
\label{prop:root-chart}
The nonzero spectrum of $A_0|_E$ is $\{z_1,\ldots,z_{N-1}\}$, and
\begin{equation}
 \det(A_0-zI)
 =-z\prod_{i=1}^N(\delta_i-z)m_x(z).
 \label{eq:charpoly}
\end{equation}
Moreover the map
\[
 x\longmapsto(z_1,\ldots,z_{N-1})
\]
is a real-analytic diffeomorphism from $\cS^\circ$ onto
\[
 \cZ:=\prod_{j=1}^{N-1}(\delta_j,\delta_{j+1}),
\]
with inverse
\begin{equation}
 x_i(z)
 =\frac{\prod_{j=1}^{N-1}(\delta_i-z_j)}
 {\prod_{k\ne i}(\delta_i-\delta_k)}.
 \label{eq:inverse-root}
\end{equation}
\end{proposition}
\begin{proof}
The determinant lemma applied to $A_0=D_0-x\delta^{\sf T}$ gives
\[
 \det(A_0-zI)
 =\det(D_0-zI)\left(1-\delta^{\sf T}(D_0-zI)^{-1}x\right).
\]
Since
\[
 \frac{\delta_i}{\delta_i-z}=1+\frac{z}{\delta_i-z},
\]
the bracket equals $-z m_x(z)$, proving \eqref{eq:charpoly}.  The inverse formula follows by writing the numerator polynomial
\[
 Q_x(z)=m_x(z)\prod_i(\delta_i-z)
\]
and evaluating at $z=\delta_i$.  Interlacing gives positivity of every $x_i(z)$ and the Lagrange interpolation identity gives $\sum_i x_i=1$.  Analyticity and mutual inversion are immediate.
\end{proof}

\begin{lemma}[Stieltjes product identity and reciprocal pole moment]\label{lem:stieltjes-product}
For $x\in\cS^\circ$ with root coordinates $z=(z_1,\ldots,z_{N-1})$,
\begin{equation}\label{eq:stieltjes-product}
 m_x(\zeta)
 =\frac{\prod_{j=1}^{N-1}(z_j-\zeta)}
 {\prod_{i=1}^{N}(\delta_i-\zeta)}.
\end{equation}
In particular the reciprocal pole moment
\begin{equation}\label{eq:Bdelta-product}
 B_\delta(x):=\sum_{i=1}^N\frac{x_i}{\delta_i}=m_x(0)
 =\frac{\prod_{j=1}^{N-1}z_j}{\prod_{i=1}^{N}\delta_i}.
\end{equation}
\end{lemma}
\begin{proof}
The two sides of \eqref{eq:stieltjes-product} have the same simple poles $\delta_i$, the same zeros $z_1,\ldots,z_{N-1}$, and the same asymptotic behaviour $-\zeta^{-1}+O(\zeta^{-2})$ as $|\zeta|\to\infty$.  Equivalently, the numerator polynomial
\[
 Q_x(\zeta)=m_x(\zeta)\prod_i(\delta_i-\zeta)
\]
has degree $N-1$, zeros $z_j$, and leading coefficient fixed by $\sum_i x_i=1$.  Evaluating at $\zeta=0$ gives \eqref{eq:Bdelta-product}.
\end{proof}

\begin{remark}[Modal diffusivities for constant multiplicative frictions]\label{rem:trace-determinant}
The product identity is not needed to determine the modal diffusivities in the general affine Darken class; those are given directly by $\sigma+\tau z_j$.  It identifies the constant-friction subclass.  For $\sigma=0$ and $\tau=B_\delta^{-1}$, \eqref{eq:Bdelta-product} yields
\[
 \alpha_j=\frac{z_j}{B_\delta(x(z))}
 =\frac{\prod_i\delta_i}{\prod_{k\ne j}z_k},
\]
which is independent of the selected root for constant multiplicative frictions.  The global regularity argument does not require this additional identity.
\end{remark}

Write
\[
 q_j=\partial_{z_j}x.
\]
Then from \eqref{eq:inverse-root},
\begin{equation}
 (q_j)_i=-\frac{x_i}{\delta_i-z_j}.
 \label{eq:qj}
\end{equation}

\begin{lemma}[Root-frame identities]
\label{lem:root-frame}
For every $j$,
\begin{equation}
 A_0q_j=z_jq_j,
 \qquad \partial_{z_jz_j}x=0.
 \label{eq:eig-affine}
\end{equation}
For $j\ne k$,
\begin{equation}
 \partial_{z_jz_k}x
 =\frac{q_j-q_k}{z_k-z_j}.
 \label{eq:mixed-connection}
\end{equation}
\end{lemma}
\begin{proof}
Equation \eqref{eq:qj} and $m_x(z_j)=0$ give $\one^{\sf T}q_j=0$.  Further,
\[
 (D_0-z_jI)q_j=-x,
 \qquad \one^{\sf T}D_0q_j=-1,
\]
so
\[
 A_0q_j=D_0q_j-x\one^{\sf T}D_0q_j=z_jq_j.
\]
The product formula \eqref{eq:inverse-root} is affine separately in each $z_j$, proving the second identity in \eqref{eq:eig-affine}.  Differentiating \eqref{eq:qj} in a transverse root gives \eqref{eq:mixed-connection}.
\end{proof}

\begin{proposition}[Dual spectral frame and Fisher metric in root variables]\label{prop:dual-root-frame}
Let
\[
 g_x(v,w):=\sum_{i=1}^N\frac{v_iw_i}{x_i},
 \qquad v,w\in E=\{q:\one^{\sf T}q=0\},
\]
be the Fisher metric on the tangent space, and write
\[
 m_j':=m_x'(z_j)=\sum_i\frac{x_i}{(\delta_i-z_j)^2}.
\]
Then the root vectors are Fisher orthogonal,
\begin{equation}\label{eq:fisher-root-orthogonal}
 g_x(q_j,q_k)=m_j'\,\delta_{jk},
\end{equation}
and the differentials of the roots are the normalized Fisher-dual covectors,
\begin{equation}\label{eq:root-dual}
 dz_j[v]=\frac{g_x(q_j,v)}{m_j'}
 =-\frac{1}{m_j'}\sum_i\frac{v_i}{\delta_i-z_j}.
\end{equation}
Consequently, in root coordinates,
\begin{equation}\label{eq:fisher-root-metric}
 g_x(v,w)=\sum_{j=1}^{N-1}m_j'\,dz_j[v]dz_j[w].
\end{equation}
For the affine Darken mobility $A=\sigma I_E+\tau A_0$ on $E$, with
$\alpha_j=\sigma+\tau z_j$, one has
\begin{equation}\label{eq:fisher-mobility-root}
 g_x(v,Aw)=\sum_{j=1}^{N-1}\alpha_jm_j'\,dz_j[v]dz_j[w].
\end{equation}
\end{proposition}
\begin{proof}
By \eqref{eq:qj},
\[
 g_x(q_j,q_k)=\sum_i\frac{x_i}{(\delta_i-z_j)(\delta_i-z_k)}.
\]
For $j\ne k$ this is the divided difference
$[m_x(z_j)-m_x(z_k)]/(z_j-z_k)$ and therefore vanishes; for $j=k$ it equals $m_j'$.  Differentiating the identity $m_x(z_j)=0$ in a tangent direction $v$ gives
\[
 0=\sum_i\frac{v_i}{\delta_i-z_j}+m_j'\,dz_j[v],
\]
which is \eqref{eq:root-dual}.  Since $q_j=\partial_{z_j}x$, every tangent vector has the expansion $v=\sum_jdz_j[v]q_j$, and \eqref{eq:fisher-root-metric} follows from orthogonality.  Finally $Aq_j=\alpha_jq_j$ gives \eqref{eq:fisher-mobility-root}.
\end{proof}

The last identity is the spectral form of the mixing-entropy dissipation.  The same weights $\alpha_jm_j'$ occur in the EPD representation of the mixing entropy in root coordinates.

\subsection{Stieltjes eigenvectors and the root equations}
\label{sec:root-system}
For the affine Darken class, \eqref{eq:affine-compression} and \Cref{lem:root-frame} give
\begin{equation}\label{eq:alpha}
 A(x)q_j=\alpha_j(z)q_j,
 \qquad
 \alpha_j(z)=\sigma(x(z))+\tau(x(z))z_j.
\end{equation}
The modal coefficients inherit the physical ellipticity bounds
\begin{equation}\label{eq:alpha-bounds}
 d_*\le\alpha_j(z)\le d^*.
\end{equation}
Indeed, since $z_j\in(\delta_j,\delta_{j+1})$ and the map $\delta\mapsto\sigma+\tau\delta$ is affine,
\begin{equation}\label{eq:alpha-convex-combination}
 \alpha_j
 =\frac{\delta_{j+1}-z_j}{\delta_{j+1}-\delta_j}\,d_j
 +\frac{z_j-\delta_j}{\delta_{j+1}-\delta_j}\,d_{j+1}.
\end{equation}
Thus the Stieltjes vectors remain eigenvectors even when the modal diffusivities depend on all roots, including the selected coordinate $z_j$.  The convex-combination formula also shows that no sign condition on $\tau$ is needed for modal ellipticity.  On every compact root rectangle the derivatives of the $\alpha_j$ up to the order needed below are uniformly bounded by the $C^3$ norms of $\sigma$ and $\tau$ on the corresponding compact concentration range.

\begin{theorem}[Exact root system]
\label{thm:root-system}
On every positive strong-solution strip the roots satisfy
\begin{equation}
 \partial_tz_j
 -\Div(\alpha_j(z)\nabla z_j)
 +\sum_{k\ne j}
 \frac{\alpha_j(z)+\alpha_k(z)}{z_j-z_k}
 \nabla z_j\cdot\nabla z_k=0,
 \label{eq:root-system}
\end{equation}
with homogeneous Neumann data $\partial_\nu z_j=0$.
\end{theorem}
\begin{proof}
Since $x_t=\sum_jq_jz_{j,t}$ and $A\nabla x=\sum_j\alpha_jq_j\nabla z_j$,
\[
 \Div(A\nabla x)
 =\sum_jq_j\Div(\alpha_j\nabla z_j)
 +\sum_{j,k}\alpha_j\partial_{z_k}q_j\,
 \nabla z_k\cdot\nabla z_j.
\]
The diagonal connection terms vanish by $\partial_{z_j}q_j=0$.  For $j<k$, Lemma~\ref{lem:root-frame} gives
\[
 \alpha_j\partial_kq_j+\alpha_k\partial_jq_k
 =(\alpha_j+\alpha_k)\frac{q_j-q_k}{z_k-z_j}.
\]
Resolving the resulting identity in the basis $q_1,\ldots,q_{N-1}$ yields \eqref{eq:root-system}.  Since the physical no-flux condition is $A\partial_\nu x=0$, and $A|_E$ is invertible, $\partial_\nu x=0$; smooth invertibility of the root map then gives $\partial_\nu z_j=0$.
\end{proof}

\begin{remark}[Comparison with the additive-friction root equation]
\label{rem:root-system-additive-parallel}
Equation \eqref{eq:root-system} is the same fixed-Stieltjes divergence system that underlies the additive-friction argument in \cite[Section~5]{BotheAdditive2026}.  In both cases the inverse root map is multi-affine and
\[
 \partial_{z_jz_j}x=0,
 \qquad
 \partial_{z_jz_k}x
 =\frac{\partial_{z_j}x-\partial_{z_k}x}{z_k-z_j},
\]
so the mixed Cauchy coefficient is always
$(\alpha_j+\alpha_k)/(z_j-z_k)$.  The constitutive formulas for the
modal diffusivities are different.  For additive frictions $f_{ij}=g_i+g_j$,
write $z_k^{\mathrm{add}}$ for the Stieltjes roots associated with the distinct
pole list $g_1,\ldots,g_N$ and set
$e_1^{\mathrm{add}}:=\sum_{\ell=1}^N g_\ell$.  Then
\[
 \alpha_j^{\mathrm{add}}
 =\left(e_1^{\mathrm{add}}
   -\sum_{\substack{k=1\\k\ne j}}^{N-1}z_k^{\mathrm{add}}\right)^{-1},
\]
whereas the present Darken class has $\alpha_j=\sigma(x(z))+\tau(x(z))z_j$.  In particular $\partial_{z_j}\alpha_j$ need not vanish.  The EPD production identity and the subsequent one-sided estimates remain valid because they are written in divergence form and require no differentiation of $\alpha_j$ in the selected direction.  If the principal divergence is expanded instead, it generally contains $\partial_{z_j}\alpha_j\,|\nabla z_j|^2$.  The identity $\partial_{z_jz_j}x=0$ removes the diagonal quadratic gradient term arising from the coordinate transformation, not this contribution from a variable diffusivity.
\end{remark}

\begin{lemma}[Strong-Sobolev transfer to the root chart]
\label{lem:root-strong-transfer}
Let $I\Subset(0,T_{\max})$ and assume that the concentration range on
$\overline I\times\overline\Omega$ is contained in a compact subset of
$\cS^\circ$.  Then the Stieltjes root transform $z=z(u)$ satisfies
\begin{equation}
 z\in W^{1,p}(I;L^p(\Omega;\R^{N-1}))
 \cap L^p(I;W^{2,p}(\Omega;\R^{N-1}))
 \cap C(\overline I;X_{\gamma,p}),
 \label{eq:strong-root-chain-class}
\end{equation}
and $\partial_\nu z=0$.  In particular, since $p>d+2$,
$z,\nabla z$ are bounded and continuous on every compact positive-time strip.
\end{lemma}
\begin{proof}
On the stated compact state range the maps $u\mapsto z$ and $z\mapsto u$
are smooth with all derivatives needed below uniformly bounded.  The Bochner
chain rule gives $z_t=D z(u)u_t\in L^p(I;L^p)$.  For spatial derivatives,
\[
 D^2_xz=Dz(u)D_x^2u+D^2z(u)[D_xu,D_xu].
\]
The first term belongs to $L^p$ because $D_x^2u\in L^p$, while the second does
because $X_{\gamma,p}\hookrightarrow C^{1+\eta}$ makes $D_xu$ uniformly
bounded on $I\times\Omega$.  Smooth composition in the trace space gives the
continuity into $X_{\gamma,p}$.  Finally $\partial_\nu u=0$ and the ordinary
chain rule imply $\partial_\nu z=Dz(u)\partial_\nu u=\nobreak0$.
\end{proof}

\begin{theorem}[Invariant root rectangle]
\label{thm:root-rectangle}
Let $t_0\ge0$ be a time at which the concentrations are strictly positive and let
\[
 a_j\le z_j(t_0,x)\le b_j,
 \qquad
 \delta_j<a_j\le b_j<\delta_{j+1}.
\]
Then for every later time in the maximal strong-existence interval,
\begin{equation}
 a_j\le z_j(t,x)\le b_j.
 \label{eq:rectangle}
\end{equation}
Moreover $\min_xz_j(t,x)$ is nondecreasing and $\max_xz_j(t,x)$ is nonincreasing.  Consequently the concentration range remains in a compact subset of $\cS^\circ$.
\end{theorem}
\begin{proof}
Fix $t_0<T<T_{\max}$ and enclose $R=\prod_j[a_j,b_j]$ in the interior of a compact root rectangle $\widehat R\Subset\cZ$. The root chart is valid at $t_0$ and stays valid up to any first exit from $\widehat R$, since $x(\widehat R)$ is compact in $\cS^\circ$. On this interval, fix $j$ and test \eqref{eq:root-system} by $r=(z_j-b_j)_+$. Since $\nabla z_j=\nabla r$ on the active set,
\[
 \frac12\frac{\dd}{\dd t}\|r\|_2^2
 +\int\alpha_j|\nabla r|^2
 =-\sum_{k\ne j}\int
 \frac{\alpha_j+\alpha_k}{z_j-z_k}
 r\,\nabla r\cdot\nabla z_k.
\]
The gradients are bounded by Lemma~\ref{lem:root-strong-transfer} and the trace embedding on each compact strong-solution interval with root range in $\widehat R$. If $t_0=0$, Lemma~\ref{lem:short-positive} gives a compact interior concentration range on a short initial interval. The same Sobolev-composition estimates apply there up to the initial trace, and $z\in C([0,T_1];C^1)$ for sufficiently small $T_1>0$. On $\widehat R$ the root gaps are bounded away from zero. Hence
\[
 \frac12\frac{\dd}{\dd t}\|r\|_2^2
 +\frac{d_*}{2}\|\nabla r\|_2^2
 \le C\|r\|_2^2.
\]
Gronwall gives $r\equiv0$ from the zero initial truncation. Testing by $-(a_j-z_j)_+$ gives the lower bound. The constants may depend on the compact interval ending at $T$. By continuity, at a first exit time the solution would still take values in $R$, strictly inside $\widehat R$. Thus no exit occurs on $[t_0,T]$. Since $T<T_{\max}$ was arbitrary, \eqref{eq:rectangle} holds throughout the strong lifespan. Applying this comparison at each later time with the current spatial extrema proves their monotonicity. The inverse \eqref{eq:inverse-root} is uniformly positive on $R$.
\end{proof}

\begin{corollary}[Uniform species separation]
\label{cor:species-separation}
If every species is strictly positive at a starting time $t_0$, there is $\kappa(t_0)>0$ determined by the root rectangle at $t_0$ such that
\[
 x_i(t,x)\ge\kappa(t_0)
\]
for every species and all $t\ge t_0$ in the strong-existence interval.  In particular, strictly positive initial data give a time-uniform lower bound from $t=0$.
\end{corollary}
\begin{proof}
For $R=\prod_j[a_j,b_j]\Subset\cZ$ put
\[
 \kappa_R:=\min_{1\le i\le N}
 \frac{\prod_{j=1}^{N-1}\operatorname{dist}(\delta_i,[a_j,b_j])}
      {\prod_{k\ne i}|\delta_i-\delta_k|}>0.
\]
The inverse formula \eqref{eq:inverse-root} gives $x_i(z)\ge\kappa_R$ for
$z\in R$.  This also excludes a loss of the root chart through a simplex
face before the end of the strong-existence interval.  If an initial root
has equal minimum and maximum, it is constant by
\Cref{thm:root-rectangle}; for the coefficient estimates one may enclose the
image in a slightly larger nondegenerate compact root rectangle.
\end{proof}

The image $x(R)$ also has a direct description in composition space. By strict monotonicity of $m_x$ in each pole gap, an interior composition belongs to $x(R)$ exactly when
\[
 m_x(a_j)\le0\le m_x(b_j),\qquad j=1,\ldots,N-1.
\]
For fixed levels these are affine inequalities in $x$. Their intersection with the closed simplex is already contained in its interior. Indeed, if it contained a boundary point, joining that point to an element of $x(R)$ by a segment would give interior points in $x(R)$ arbitrarily close to the boundary, contradicting the positive lower bound above. Thus $x(R)$ is a compact convex polytope, possibly of lower dimension if some root intervals are singletons.

\subsection{Affine moment coordinates and entropy symmetrization}
\label{sec:moments}
Following \cite{BotheAdditive2026}, we use roots for scalar localization and fixed affine moments for system regularity.  The concentration variables also satisfy a closed entropy-symmetrized system.

For $\ell=1,\ldots,N-1$ define
\begin{equation}\label{eq:core-moments}
 U_\ell:=\sum_{i=1}^N\delta_i^\ell x_i,
 \qquad U_0:=1.
\end{equation}
Let
\[
 L_i(\xi):=\prod_{m\ne i}\frac{\xi-\delta_m}{\delta_i-\delta_m}
 =\sum_{\ell=0}^{N-1}\ell_{i\ell}\xi^\ell
\]
be the Lagrange basis polynomial associated with $\delta_i$.

\begin{proposition}[Affine moment bijection]
\label{prop:core-moment-bijection}
The map $x\mapsto U=(U_1,\ldots,U_{N-1})$ is an affine bijection from the
closed concentration simplex onto an affine simplex $\mathcal T$. The image of $\cS^\circ$ is the interior $\mathcal T^\circ$. Its inverse is
\begin{equation}\label{eq:core-moment-inverse}
 x_i(U)=\sum_{\ell=0}^{N-1}\ell_{i\ell}U_\ell.
\end{equation}
In reduced concentration coordinates $u=(x_1,\ldots,x_{N-1})$, its constant
derivative is the invertible matrix
\begin{equation}\label{eq:core-moment-L}
 \mathsf L_{\ell i}=\delta_i^\ell-\delta_N^\ell,
 \qquad 1\le\ell,i\le N-1.
\end{equation}
\end{proposition}
\begin{proof}
For every polynomial $p(\xi)=\sum_{\ell=0}^{N-1}p_\ell\xi^\ell$ of degree at
most $N-1$,
\[
 \sum_i x_ip(\delta_i)=\sum_{\ell=0}^{N-1}p_\ell U_\ell.
\]
Taking $p=L_i$ gives \eqref{eq:core-moment-inverse}.  The distinctness of the
$\delta_i$ makes the Vandermonde matrix invertible, hence the affine map is a
bijection.  Since
$U_\ell=\delta_N^\ell+\sum_{i<N}(\delta_i^\ell-\delta_N^\ell)u_i$,
its derivative is \eqref{eq:core-moment-L}.
\end{proof}

\begin{proposition}[Closed moment system and entropy symmetrizer]
\label{prop:core-moment-system}
Let $U$ be the moment vector of a sufficiently regular solution of
\eqref{eq:reduced-system}.  Then
\begin{equation}\label{eq:core-moment-system}
 U_t-\Div\bigl(A_M(U)\nabla U\bigr)=0,
 \qquad
 A_M(U):=\mathsf L C(u(U))\mathsf L^{-1},
\end{equation}
with homogeneous Neumann data.  The matrix $A_M$ extends $C^3$-smoothly to a
neighborhood of the moment simplex $\mathcal T$, is similar to $C$, and therefore
has the same physical spectrum.  If
\begin{equation}\label{eq:core-moment-entropy}
 \begin{aligned}
  h_M(U)&:=h_u(u(U)),\\
  H_M(U)&:=D_U^2h_M(U)
       =\mathsf L^{-{\sf T}}H(u(U))\mathsf L^{-1},
 \end{aligned}
\end{equation}
then
\begin{equation}\label{eq:core-moment-symmetrizer}
 H_M(U)A_M(U)
 =\mathsf L^{-{\sf T}}H(u(U))C(u(U))\mathsf L^{-1}
\end{equation}
is symmetric positive definite on the open moment simplex.
\end{proposition}
\begin{proof}
Because $\mathsf L$ is constant, applying it to
\eqref{eq:reduced-system} and using $\nabla U=\mathsf L\nabla u$ gives
\eqref{eq:core-moment-system}.  The affine inverse map and the $C^3$ regularity of $C$
give the closed-simplex extension, while similarity gives the spectral
statement.  Formula \eqref{eq:core-moment-entropy} is the ordinary Hessian
transformation under the affine map $u=u(U)$.  Multiplying by $A_M$ gives
\eqref{eq:core-moment-symmetrizer}, and Proposition~\ref{prop:HC} gives
symmetry and positivity.
\end{proof}

\begin{remark}[Root and affine moment representations]
\label{rem:core-root-moment-duality}
The root and moment coordinates represent the same fixed compressed diagonal
operator in different forms.  The roots are the characteristic values of the
compression and yield a diagonal scalar principal part; the moments are fixed
affine coordinates and preserve divergence structure without introducing a
state-dependent coordinate connection.  This is the same separation of scalar and system variables used
in \cite[Sections~4--8]{BotheAdditive2026}.
\end{remark}

\section{Entropy-stabilized EPD localization and De Giorgi regularity}
\label{sec:epd-dg-block}
\subsection{EPD production and one-sided truncations}
\label{sec:epd}
This section applies the localization method of \cite[Sections~11--12]{BotheAdditive2026}.  We verify the coefficient-dependent entropy identities and derive the quadratic, logarithmic, and oscillation estimates.  By \Cref{thm:root-system}, the additive-friction and affine Darken problems have the same fixed-Stieltjes Cauchy form
\[
 z_{j,t}-\Div(\mathfrak a_j(z)\nabla z_j)
 +\sum_{k\ne j}\frac{\mathfrak a_j+\mathfrak a_k}{z_j-z_k}\nabla z_j\cdot\nabla z_k=0.
\]
The EPD cancellation uses this divergence system and the Stieltjes connection identities.  The one-sided coercivity also requires uniform positive bounds for $\mathfrak a_j$, separated root coordinates, and bounded entropy coefficients.  The logarithmic estimate uses the same bounds and the weighted transverse dissipation supplied by stabilization.  The chain rule introduces no selected-root derivative of $\mathfrak a_j$; the condition $\partial_{z_j}\mathfrak a_j=0$ is therefore unnecessary.

The production identity and coercive stabilization yield the Caccioppoli
and logarithmic estimates.  These imply propagation of a level-set
fraction, a critical-mass estimate, shrinking of superlevel sets, and
oscillation decay.  All constants depend on the structural data specified below.

\paragraph{Structural data for the scalar estimates.}
For a compact root rectangle $R\Subset\cZ$, write
\[
 \mathcal D_R:=\bigl(d,N,\Omega,R,\delta,d_*,d^*,
 \|\sigma\circ x\|_{C^2(R)},\|\tau\circ x\|_{C^2(R)}\bigr),
\]
Here $x=x(z)$ is the inverse Stieltjes map and the domain entry denotes
the uniform local geometric bounds used below. Constants denoted by
$c_R,C_R,\varepsilon_R$ and constants called \emph{structural} may depend on
$\mathcal D_R$. Additional parameters, such as a waiting-time gap or a
prescribed level-set fraction, are specified where they occur. The
$C^2$ bounds record the smooth coefficient law in the strong-root setting.
The local energy and level-set estimates use the bounds on the values of
the modal coefficients and the mixed coefficients, together with the root,
entropy and domain geometry. In particular, uniformity over constitutive
laws requires a common ellipticity lower bound $d_*>0$.

The scalar iteration applies to a strong solution of the displayed
fixed-Stieltjes divergence system with range in $R$, homogeneous Neumann
trace, uniformly positive bounded modal coefficients, and bounded Cauchy
coefficients.  The constitutive formula verifies these bounds.  The
quadratic and logarithmic estimates below retain the dependence on
$\mathcal D_R$.

\begin{lemma}[Coefficient bounds on a root rectangle]\label{lem:coefficient-ledger}
Let $R=\prod_{j=1}^{N-1}[a_j,b_j]\Subset\cZ$.  Then there are constants $0<c_R<C_R<\infty$ such that, for every $z\in R$ and all admissible indices,
\begin{align*}
 c_R&\le \alpha_j(z)\le C_R, &
 |D^\ell\alpha_j(z)|&\le C_R\quad(|\ell|\le2),\\
 |z_j-z_k|^{-1}&\le C_R\quad(j\ne k), &
 \left|D^\ell\frac{\alpha_j+\alpha_k}{z_j-z_k}\right|&\le C_R\quad(|\ell|\le1),\\
 c_R&\le m_x'(z_j)\le C_R.
\end{align*}
The constants depend only on the fixed root rectangle, the species parameters $\delta_i$, the positive lower and upper bounds $d_*,d^*$ in \eqref{eq:darken-uniform-bounds}, and finite $C^2$ bounds for $\sigma$ and $\tau$ on the corresponding compact concentration range.  Equivalently, they are controlled by $\mathcal D_R$.
\end{lemma}
\begin{proof}
The root gaps are uniformly separated on $R$.  The inverse Stieltjes map $x(z)$ is smooth on $R$, while $\sigma$ and $\tau$ are $C^3$; hence \eqref{eq:alpha} gives the stated derivative bounds.  The positivity bounds follow from \eqref{eq:alpha-bounds}.  The Cauchy-coefficient bounds are then immediate.  Finally
\[
 m_x'(z_j)=\sum_i\frac{x_i(z)}{(\delta_i-z_j)^2}
\]
is positive and smooth on $R$, so compactness gives its two-sided bound.
\end{proof}

For a scalar state function $\Psi(z)$ define
\begin{align}
 \mathfrak Q_\Psi(z;\Xi)
 :={}&\sum_j\alpha_j\Psi_{jj}|\Xi_j|^2\notag\\
 &+\sum_{j<k}(\alpha_j+\alpha_k)
 \left(\Psi_{jk}+\frac{\Psi_j-\Psi_k}{z_j-z_k}\right)
 \Xi_j\cdot\Xi_k.
 \label{eq:Qpsi}
\end{align}

\begin{lemma}[Strong root chain rule]
\label{lem:chain}
Let $z$ be a maximal-$L^p$ strong root solution with compact root range and let $\Psi$ be $C^2$ on a neighborhood of that range. Then
\begin{equation}
 \partial_t\Psi(z)
 -\Div\left(\sum_j\alpha_j(z)\Psi_j(z)\nabla z_j\right)
 +\mathfrak Q_\Psi(z;\nabla z)=0
 \label{eq:chain}
\end{equation}
in distributions.  The same formula holds for $\Psi=\mathbf1_{\{z_i>k\}}F$ or $\mathbf1_{\{z_i<k\}}F$ when $F$ has bounded classical second derivatives on a neighborhood of the closed active branch and $F=DF=0$ on $\{z_i=k\}$.  Its zero extension is then $C^{1,1}$ and no hypersurface measure occurs on the target interface.
\end{lemma}
\begin{proof}
For smooth $\Psi$, multiply the $j$th equation in \eqref{eq:root-system} by $\Psi_j(z)$ and sum over $j$.  The temporal terms give $\partial_t\Psi(z)$.  For the principal part,
\[
 -\sum_j\Psi_j\Div(\alpha_j\nabla z_j)
 =-\Div\!\left(\sum_j\alpha_j\Psi_j\nabla z_j\right)
 +\sum_{j,k}\alpha_j\Psi_{jk}\nabla z_j\cdot\nabla z_k.
\]
Pairing $(j,k)$ and $(k,j)$ in the last sum and adding the Cauchy couplings from \eqref{eq:root-system} yields exactly
\[
 \sum_j\alpha_j\Psi_{jj}|\nabla z_j|^2
 +\sum_{j<k}(\alpha_j+\alpha_k)
 \left(\Psi_{jk}+\frac{\Psi_j-\Psi_k}{z_j-z_k}\right)
 \nabla z_j\cdot\nabla z_k.
\]
This is \eqref{eq:Qpsi}.  The calculation remains in divergence form; no derivative of $\alpha_j$ is generated.

The regularity in \Cref{lem:root-strong-transfer} justifies these operations
by the Sobolev chain and product rules. For a one-sided branch the
conditions $F=DF=0$ give a $C^{1,1}$ zero extension with a weak Hessian
equal to the classical Hessian on the active side. The time derivative of
$\Psi(z)$ is $D\Psi(z)z_t$. In space each $\Psi_j(z)$ belongs to
$W^{1,p}$ on the strong-solution slices, and the usual chain rule
applies on the active set. On the inactive set, including the physical
level set $\{(t,y):z_i(t,y)=k\}$, one has $\Psi_j(z)=0$. The Sobolev
level-set rule gives $\nabla[\Psi_j(z)]=0$ almost everywhere there, even
when the physical level set has positive measure. Thus one may take the
Hessian representative to be zero on the interface in
\eqref{eq:Qpsi}, and \eqref{eq:chain} holds throughout. See also
\cite[Lemma~11.1]{BotheAdditive2026}.
\end{proof}

The EPD compatibility equation is
\begin{equation}
 (z_k-z_j)\Psi_{jk}=\Psi_j-\Psi_k,
 \qquad j\ne k.
 \label{eq:EPD}
\end{equation}
It cancels every mixed term in \eqref{eq:Qpsi}.

\subsubsection{Mixing entropy as an exact EPD entropy}
Let $h_z(z):=h(x(z))$.  From \eqref{eq:qj},
\begin{equation}
 \partial_{z_j}^2h_z
 =\sum_i\frac{x_i}{(\delta_i-z_j)^2}
 =m_x'(z_j)>0.
 \label{eq:hjj}
\end{equation}
For $j\ne k$,
\[
 \sum_i\frac{x_i}{(\delta_i-z_j)(\delta_i-z_k)}=0,
\]
which is the divided difference of $m_x$ at two of its zeros.  Together with \eqref{eq:mixed-connection}, this gives \eqref{eq:EPD} for $h_z$.  Hence
\begin{equation}
 \mathfrak Q_{h_z}(z;\Xi)
 =\sum_j\alpha_j(z)m_x'(z_j)|\Xi_j|^2.
 \label{eq:root-entropy-coercive}
\end{equation}
On a compact root rectangle this is uniformly equivalent to $\sum_j|\Xi_j|^2$.

\subsubsection{One-sided EPD potentials and entropy stabilization}
Fix $i\in\{1,\ldots,m\}$ and a level $k\in[a_i,b_i]$.  Define
\begin{equation}
 \Phi_{i,k}(z)
 :=(-1)^{i-1}\int_k^{z_i}(z_i-\sigma)
 \prod_{j\ne i}(z_j-\sigma)\,\dd\sigma.
 \label{eq:Phi}
\end{equation}
For the upper and lower branches set
\begin{equation}
 \Phi_{i,k}^+(z):=\mathbf1_{\{z_i>k\}}\Phi_{i,k}(z),
 \qquad
 \Phi_{i,k}^-(z):=\mathbf1_{\{z_i<k\}}\Phi_{i,k}(z).
 \label{eq:Phi-branches}
\end{equation}
Because both $\Phi_{i,k}$ and its first derivatives vanish on $z_i=k$, the
one-sided branches belong to $C^{1,1}(R)$, their weak Hessians are the
classical Hessians on the active side and zero on the inactive side, and no
hypersurface measure is created at the target interface.  Thus the direct
one-sided chain rule in Lemma~\ref{lem:chain} applies without smoothing the
potential and preserves the EPD identity.

\begin{lemma}[Exact EPD identities and truncation bounds]
\label{lem:multi-EPD}
Let $R=\prod_j[a_j,b_j]\Subset\cZ$ be a compact root rectangle.  There are
constants $0<c_R\le C_R<\infty$ such that, on the active side of either
branch, with
\[
 r=(z_i-k)_+\quad\hbox{or}\quad r=(k-z_i)_+,
\]
one has
\begin{align}
 c_Rr^2&\le \Phi_{i,k}^{\pm}(z)\le C_Rr^2,
 \label{eq:Phi-size}\\
 |\partial_i\Phi_{i,k}|&\le C_Rr,
 \qquad |\partial_j\Phi_{i,k}|\le C_Rr^2\quad(j\ne i),
 \label{eq:Phi-der}\\
 \partial_{ii}\Phi_{i,k}
 &=\prod_{j\ne i}|z_j-z_i|\ge c_R,
 \qquad \partial_{jj}\Phi_{i,k}=0\quad(j\ne i),
 \label{eq:Phi-hess}
\end{align}
and, for every $a\ne b$,
\begin{equation}
 (z_b-z_a)\partial_{ab}\Phi_{i,k}
 =\partial_a\Phi_{i,k}-\partial_b\Phi_{i,k}.
 \label{eq:Phi-EPD}
\end{equation}
Consequently
\begin{equation}
 \mathfrak Q_{\Phi_{i,k}^{\pm}}(z;\nabla z)
 =\alpha_i(z)\partial_{ii}\Phi_{i,k}(z)|\nabla r|^2
 \ge c_R|\nabla r|^2
 \label{eq:pure-coercive}
\end{equation}
almost everywhere.
\end{lemma}

\begin{proof}
Write $P_i(s)=(-1)^{i-1}\prod_{j\ne i}(z_j-s)$.  For $s$ between $k$ and
$z_i$, interlacing gives $0<c_R\le P_i(s)\le C_R$ and uniformly separates
$s$ from every transverse root.  Differentiation under the integral gives
\[
 \partial_i\Phi=\int_k^{z_i}P_i(s)\,\dd s,
 \qquad \partial_{ii}\Phi=P_i(z_i),
 \qquad
 \partial_j\Phi=\int_k^{z_i}
       (z_i-s)\frac{P_i(s)}{z_j-s}\,\dd s\quad(j\ne i).
\]
The potential is affine in each transverse root, so $\partial_{jj}\Phi=0$.
These formulas give \eqref{eq:Phi-size}--\eqref{eq:Phi-hess}; for $z_i<k$
the reversed integration orientation and the factor $z_i-s$ have opposite
signs, leaving the same positive quadratic bound.

For $a,b\ne i$, subtract the two transverse derivative formulas and use
\[
 \frac1{z_a-s}-\frac1{z_b-s}
 =\frac{z_b-z_a}{(z_a-s)(z_b-s)}.
\]
For the pair $i,j$, use instead
\[
 1-\frac{z_i-s}{z_j-s}=\frac{z_j-z_i}{z_j-s}.
\]
These are exactly the two cases of \eqref{eq:Phi-EPD}.  Thus all mixed
production vanishes, and the selected term is positive by $\alpha_i\ge d_*$.
The zero extension is justified by \Cref{lem:chain}.  This supplies the
fixed-pole calculation underlying \cite[Lemma~11.2]{BotheAdditive2026}
without any assumption on selected-root derivatives of the coefficients.
\end{proof}

If $m\ge2$, the pure EPD branch controls the selected gradient but supplies
no transverse-gradient coercivity for the fluxes generated by spatial
localization.  The physical mixing entropy provides the transverse coercivity
required to control these terms.
Since $-\log N\le h_z\le0$, choose $C_h=1+\log N$, so that
\[
 \widehat h(z):=h_z(z)+C_h\ge1\qquad(z\in R),
\]
and for either sign set
\[
 \chi:=\frac12r^2.
\]
On the active side, $\chi_i=\varsigma r$ with
$\varsigma\in\{1,-1\}$, $\chi_{ii}=1$, and all other first derivatives of
$\chi$ vanish.

\begin{lemma}[Exact product identity for the root production]
\label{lem:Q-product}
On the active side,
\begin{align}
 \mathfrak Q_{\chi\widehat h}
 ={}&\chi\mathfrak Q_{h_z}+\widehat h\,\mathfrak Q_\chi
 +2\alpha_i\chi_i\,\partial_{z_i}h_z\,|\nabla z_i|^2\notag\\
 &+\chi_i\sum_{j\ne i}(\alpha_i+\alpha_j)\,\partial_{z_j}h_z
 \nabla z_i\cdot\nabla z_j,
 \label{eq:Q-product}
\end{align}
where
\begin{equation}
 \mathfrak Q_\chi
 =\alpha_i|\nabla r|^2
 +\chi_i\sum_{j\ne i}
 \frac{\alpha_i+\alpha_j}{z_i-z_j}\,
 \nabla z_i\cdot\nabla z_j.
 \label{eq:Q-chi}
\end{equation}
\end{lemma}

\begin{proof}
This is \cite[Lemma~11.4]{BotheAdditive2026} with
$\mathfrak a_j$ replaced by $\alpha_j$.  The identity is purely algebraic:
expand $D^2(\chi\widehat h)$ in \eqref{eq:Qpsi}, use that $h_z$ satisfies
the pairwise EPD system, and use $\chi_j=0$ for $j\ne i$.  No constitutive formula
for the modal diffusivities is used.
\end{proof}

\begin{proposition}[Coercive entropy-stabilized truncation]
\label{prop:stabilized}
For every compact root rectangle $R\Subset\cZ$ there is
$\eps_R>0$ such that, for $0<\eps\le\eps_R$, every selected root, every
level $k\in[a_i,b_i]$, and either sign,
\begin{equation}
 \Psi_{i,k}^{\pm}(z)
 :=\Phi_{i,k}^{\pm}(z)+\eps\chi(z)\widehat h(z)
 \label{eq:stabilized}
\end{equation}
is a one-sided $C^{1,1}$ potential admissible in the chain rule and satisfies
\begin{equation}
 \mathfrak Q_{\Psi_{i,k}^{\pm}}(z;\nabla z)
 \ge c_R|\nabla r|^2
 +c_R\eps r^2\sum_{j\ne i}|\nabla z_j|^2.
 \label{eq:stabilized-coercive}
\end{equation}
Moreover
\begin{equation}
 c_Rr^2\le\Psi_{i,k}^{\pm}\le C_Rr^2,
 \qquad |\Psi_i|\le C_Rr,
 \qquad |\Psi_j|\le C_Rr^2\quad(j\ne i).
 \label{eq:stabilized-flux-bounds}
\end{equation}
The constants $c_R,C_R$ can be chosen uniformly for
$0<\eps\le\eps_R$.
\end{proposition}

\begin{proof}
The entropy correction is introduced to absorb the transverse error terms.  The pure EPD branch gives
\[
 \mathfrak Q_{\Phi_{i,k}^{\pm}}\ge c_0|\nabla r|^2.
\]
By \eqref{eq:root-entropy-coercive}, compactness of $R$, and $\chi=r^2/2$,
\[
 \chi\mathfrak Q_{h_z}
 \ge c_1r^2\sum_j|\nabla z_j|^2.
\]
On $R$, the functions $\widehat h$, $D_z h_z$, the reciprocal root gaps,
and the coefficients $\alpha_j$ are bounded.  Since $\chi_i=\pm r$,
the negative part of the remaining terms in \eqref{eq:Q-product}
is bounded pointwise by
\[
 C r|\nabla r|^2
 +C r|\nabla r|\sum_{j\ne i}|\nabla z_j|.
\]
After multiplication by $\varepsilon$, Young's inequality gives, for arbitrary $\eta>0$,
\[
 \varepsilon C r|\nabla r|\,|\nabla z_j|
 \le \eta|\nabla r|^2
 +C_{\eta}\varepsilon^2r^2|\nabla z_j|^2.
\]
The root rectangle gives a uniform upper bound for $r$, so the term
$\varepsilon C r|\nabla r|^2$ is absorbed by the pure EPD coercivity once $\varepsilon$ is sufficiently small.  Choosing first $\eta>0$ so that $(m-1)\eta\le c_0/4$, and then $\varepsilon_R>0$ so small that the sum of the $O(\varepsilon^2)$ transverse terms is dominated by the positive $\varepsilon\chi\mathfrak Q_{h_z}$ term yields
\[
 \mathfrak Q_{\Phi+\varepsilon\chi\widehat h}
 \ge c_R|\nabla r|^2
 +c_R\varepsilon r^2\sum_{j\ne i}|\nabla z_j|^2.
\]
The flux bounds follow by combining \eqref{eq:Phi-der} with boundedness of $\widehat h$ and $D_z h_z$ on $R$.  This is the affine Darken version of the stabilized coercive estimate in \cite[Proposition~11.5]{BotheAdditive2026}.
\end{proof}

\subsection{Scalar De Giorgi regularity of the Stieltjes roots}
\label{sec:DG}
For the remainder of this section, fix $\eps=\eps_R$ in
\eqref{eq:stabilized}.  Constants in localized estimates may depend on
$\eps_R^{-1}$; since $\eps_R$ is fixed by $\mathcal D_R$, this is
structural dependence and is independent of the truncation level and
amplitude.
The quadratic estimate alone does not provide the required two-time
propagation estimate, because the temporal energy of the stabilized truncation
is only uniformly equivalent to $(z_i-k)_\pm^2$.  We therefore derive the
time-propagation estimate directly from the scalar root equation.  The
additional positive term
\[
 (z_i-k)_\pm^2\sum_{j\ne i}|\nabla z_j|^2
\]
created by the entropy stabilization controls the cross-root drift in the
logarithmic test of the selected scalar root equation.

For $x_0\in\overline\Omega$ and $\rho>0$ put
\[
 D_\rho(x_0):=B_\rho(x_0)\cap\Omega,
 \qquad
 Q_\rho^\Omega(t_0,x_0):=(t_0-\rho^2,t_0)\times D_\rho(x_0).
\]
We suppress the centers when no confusion can arise.

\begin{lemma}[Uniform local geometry up to the physical boundary]
\label{lem:uniform-local-geometry-DG}
There are $\rho_\Omega>0$ and constants depending only on $d$ and the
$C^2$-geometry of $\Omega$ such that, for every
$x_0\in\overline\Omega$ and $0<\rho\le\rho_\Omega$:
\begin{enumerate}
\item $|D_\rho(x_0)|\asymp_\Omega \rho^d$, and there is a uniform thin-annulus modulus: for every $\varepsilon>0$ there is $\lambda_\varepsilon\in(0,1)$, depending only on $\varepsilon$ and the local geometry, such that
\[
 |D_\rho(x_0)\setminus D_{\lambda_\varepsilon\rho}(x_0)|
 \le \varepsilon |D_\rho(x_0)|
 \qquad(0<\rho\le\rho_\Omega);
\]
\item the Sobolev and parabolic Sobolev inequalities on $D_\rho(x_0)$ hold
with constants independent of $x_0$ and $\rho$, for functions which vanish
near the artificial boundary $\partial B_\rho(x_0)\cap\Omega$ but need not
vanish on $\partial\Omega$.  In particular, if $I$ is an interval and
{$\phi$} has this spatial support property, then
\begin{equation}\label{eq:boundary-parabolic-sobolev-DG}
 \iint_{I\times D_\rho}|{\phi}|^{2(d+2)/d}
 \le C_\Omega
 \left(\operatorname*{ess\,sup}_{t\in I}
       \int_{D_\rho}|{\phi(t)}|^2\right)^{2/d}
 \iint_{I\times D_\rho}|\nabla {\phi}|^2,
\end{equation}
\item for {$\varphi\in W^{1,1}(D_\rho)$} and $h<\ell$,
\begin{equation}\label{eq:relative-DeGiorgi-isoperimetric}
 (\ell-h)
 |\{{\varphi}\le h\}\cap D_\rho|
 |\{{\varphi}\ge \ell\}\cap D_\rho|
 \le C_\Omega\rho^{d+1}
 \int_{D_\rho\cap\{h<{\varphi}<\ell\}}|\nabla {\varphi}|\,dx.
\end{equation}
\end{enumerate}
\end{lemma}

\begin{proof}
Choose $\rho_\Omega$ below the tubular radius and sufficiently small
relative to the uniform $C^2$ graph-chart size. If $B_{2\rho}(x_0)$ meets
the wall, orient the last coordinate along the inward normal at the
nearest boundary point, center at $x_0$, and rescale by $\rho$. In these
coordinates the rescaled intersection is
\[
 D=B_1\cap\{y_d>\psi(y')\},\qquad \psi(0)\le0,\qquad
 \operatorname{Lip}\psi\le \tfrac18 .
\]
The bound on the slope follows from the $C^2$ geometry by reducing
$\rho_\Omega$, with the chart defined on a fixed dilation of the patch.
Let $B=B_{1/8}(e_d/2)$. For every $q\in B$ and $y\in D$,
\[
 q_d-\psi(y')-\tfrac18|q'-y'|
 \ge \tfrac38-\tfrac18-\tfrac{9}{64}>0 .
\]
Hence the segment from $y$ to $q$ stays above the graph, by the Lipschitz
bound for $\psi$, and stays inside $B_1$ by convexity. In particular $D$
contains $B$ and is star-shaped with respect to every point of $B$.
If $B_{2\rho}(x_0)$ misses the wall, $D=B_1$ and the same conclusions hold.
Consequently
\begin{equation}\label{eq:local-volume-density-DG}
 c_\Omega\rho^d\le |D_\rho(x_0)|\le\omega_d\rho^d .
\end{equation}

For a function $\phi$ vanishing near the artificial side, first extend
it by zero across that side within the epigraph of $\psi$. Flatten the
graph with $(y',y_d)\mapsto(y',y_d-\psi(y'))$ and reflect the extended
function evenly across the plane. The flattening and its inverse are
uniformly bi-Lipschitz, and the reflected support stays in a fixed
dilation of the patch. Undoing the flattening, or using ordinary zero
extension for the interior ball, gives $E\phi\in W^{1,2}(\mathbb R^d)$
satisfying
\begin{equation}\label{eq:local-reflection-extension-DG}
 E\phi=\phi\ \text{ on }D_\rho(x_0),\qquad
 \operatorname{supp}E\phi\subset B_{C_\Omega\rho}(x_0),
\end{equation}
and
\begin{equation}\label{eq:local-reflection-extension-norm-DG}
 \|E\phi\|_{L^2(\mathbb R^d)}
 \le C_\Omega\|\phi\|_{L^2(D_\rho)},\qquad
 \|\nabla E\phi\|_{L^2(\mathbb R^d)}
 \le C_\Omega\|\nabla\phi\|_{L^2(D_\rho)}.
\end{equation}
All constants here depend only on the fixed $C^2$ geometry of $\Omega$.
These are the standard uniform graph-chart consequences recorded in
\cite[Lemma~12.1]{BotheAdditive2026}.

For $0<\lambda<1$, the volume estimate gives
\[
 |D_\rho\setminus D_{\lambda\rho}|
 \le |B_\rho\setminus B_{\lambda\rho}|
 =\omega_d(1-\lambda^d)\rho^d.
\]
Given $\varepsilon>0$, choose $\lambda_\varepsilon\in(0,1)$ so close to one
that
$\omega_d(1-\lambda_\varepsilon^d)\le
 \varepsilon c_\Omega$.
Together with \eqref{eq:local-volume-density-DG} this proves the
thin-annulus assertion in part~(1).

For a function $\phi$ as in part~(2), apply the Euclidean Sobolev and
Gagliardo--Nirenberg inequalities to the extension in
\eqref{eq:local-reflection-extension-DG}--
\eqref{eq:local-reflection-extension-norm-DG}.  In particular, uniformly in
$x_0$ and $\rho$,
\begin{equation}\label{eq:local-GN-DG}
 \|\phi\|_{L^{2(d+2)/d}(D_\rho)}^{2(d+2)/d}
 \le C_\Omega
 \|\phi\|_{L^2(D_\rho)}^{4/d}
 \|\nabla\phi\|_{L^2(D_\rho)}^2.
\end{equation}
For $d>2$ the corresponding spatial Sobolev estimate follows in the same
way from $W^{1,2}(\mathbb R^d)\hookrightarrow
L^{2d/(d-2)}(\mathbb R^d)$; for $d=2$ the corresponding scale-invariant
finite-exponent estimates follow after rescaling.  Applying \eqref{eq:local-GN-DG} for almost every time and then
integrating in time gives \eqref{eq:boundary-parabolic-sobolev-DG}.

For part~(3), first work in the rescaled set $D$. Write $w_B$ for the
average over the ball of star centers. For smooth $w$, the segment
formula gives
\[
 \int_D |w(y)-w_B|\,dy
 \le\frac{2}{|B|}\int_D\int_B\int_0^1
       |\nabla w((1-s)y+sq)|\,ds\,dq\,dy .
\]
For $s\le1/2$, keep $q$ fixed and change variables in $y$, with
Jacobian $(1-s)^d\ge2^{-d}$. For $s\ge1/2$, keep $y$ fixed and change
variables in $q$, with Jacobian $s^d\ge2^{-d}$. In both cases the
image stays in $D$. Since $|D|/|B|$ is uniformly bounded, the right-hand
side is at most $C_\Omega\int_D|\nabla w|$. The estimate passes to
$W^{1,1}(D)$ by smooth approximation, and replacing $w_B$ by $w_D$ costs
at most a factor of two. Rescaling gives the uniform $L^1$ Poincar\'e
inequality
\begin{equation}\label{eq:local-L1-poincare-DG}
 \int_{D_\rho}|w-w_{D_\rho}|\,dx
 \le C_\Omega\rho\int_{D_\rho}|\nabla w|\,dx,
 \qquad
 w_{D_\rho}:=|D_\rho|^{-1}\int_{D_\rho}w\,dx.
\end{equation}
Consequently,
\begin{align*}
 \int_{D_\rho}\!\int_{D_\rho}|w(x)-w(y)|\,dx\,dy
 &\le 2|D_\rho|\int_{D_\rho}|w-w_{D_\rho}|\,dx\\
 &\le C_\Omega\rho^{d+1}
       \int_{D_\rho}|\nabla w|\,dx.
\end{align*}
For $\varphi\in W^{1,1}(D_\rho)$ and $h<\ell$, take
\[
 w=\min\{(\varphi-h)_+,\ell-h\}.
\]
If $A=\{\varphi\le h\}\cap D_\rho$ and
$B=\{\varphi\ge\ell\}\cap D_\rho$, then
$|w(x)-w(y)|=\ell-h$ on $A\times B$, while
$|\nabla w|=|\nabla\varphi|\mathbf1_{\{h<\varphi<\ell\}}$ almost everywhere.
The preceding double-integral estimate therefore yields
\eqref{eq:relative-DeGiorgi-isoperimetric}.
\end{proof}

\subsubsection{Quadratic and two-time stabilized estimates}

The one-sided Caccioppoli estimate from the preceding section provides
level-set control. We also derive its two-time form, retaining the
stabilized energy at the initial time slice.

\begin{proposition}[One-sided root Caccioppoli inequality]
\label{prop:root-Caccioppoli-dossier}
Let $z$ satisfy the strong root hypotheses of \Cref{lem:chain} on a
neighborhood of $Q_\rho^\Omega$, with range in the compact root rectangle
$R$ and homogeneous Neumann trace on the physical face of the cylinder.
Fix a root $i$, a level $k\in[a_i,b_i]$, and either sign. Put
$r=(z_i-k)_+$ or $r=(k-z_i)_+$.  There are constants $c_R,C_R>0$ such that,
for every $C^1$ cutoff $0\le\eta\le1$ which vanishes near the lower time face
and the artificial lateral boundary but may be nonzero on the physical
boundary,
\begin{align}
 &\operatorname*{ess\,sup}_{s}
   \int_{D_\rho}\eta^2r^2(s)
 +c_R\iint_{Q_\rho^\Omega}\eta^2|\nabla r|^2 \notag\\
 &\quad
 +c_R\varepsilon_R\iint_{Q_\rho^\Omega}
   \eta^2r^2\sum_{j\ne i}|\nabla z_j|^2
 \le C_R\iint_{Q_\rho^\Omega}
 r^2\bigl(|\nabla\eta|^2+\eta|\partial_t\eta|\bigr).
 \label{eq:root-Caccioppoli-cutoff-dossier-v11}
\end{align}
Consequently, for $0<\varrho<\rho$,
\begin{align}
 &\operatorname*{ess\,sup}_{t\in(t_0-\varrho^2,t_0)}
   \int_{D_\varrho}r^2(t)
 +\iint_{Q_\varrho^\Omega}|\nabla r|^2 \notag\\
 &\quad
 +c_R\varepsilon_R\iint_{Q_\varrho^\Omega}
 r^2\sum_{j\ne i}|\nabla z_j|^2
 \le \frac{C_R}{(\rho-\varrho)^2}
 \iint_{Q_\rho^\Omega}r^2.
 \label{eq:root-Caccioppoli-dossier-v11}
\end{align}
The constants are uniform for interior and boundary cylinders below
$\rho_\Omega$.
\end{proposition}

\begin{proof}
Put $\Psi=\Psi_{i,k}^{\pm}$ and
\[
 F_\Psi:=\sum_j\alpha_j\Psi_j\nabla z_j.
\]
Multiply the chain identity \eqref{eq:chain} by $\eta^2$ and integrate over a time interval beginning where the cutoff vanishes.  The physical-boundary term vanishes because $\partial_\nu z_j=0$.  Integration by parts gives
\[
 \int\eta^2\Psi(t)
 +\iint\eta^2\mathfrak Q_\Psi
 \le \iint 2\eta |F_\Psi||\nabla\eta|
 +\iint 2\eta|\partial_t\eta|\Psi.
\]
By \eqref{eq:stabilized-flux-bounds},
\[
 |F_\Psi|
 \le C_R\bigl(r|\nabla r|+r^2\sum_{j\ne i}|\nabla z_j|\bigr),
 \qquad \Psi\asymp_R r^2.
\]
The first flux term is absorbed by the selected-gradient part of
\eqref{eq:stabilized-coercive}; for the transverse terms we write
\[
 r^2|\nabla z_j||\nabla\eta|
 =(r|\nabla z_j|)(r|\nabla\eta|)
\]
and absorb the first factor into the stabilized transverse dissipation.  Young's inequality therefore yields
\[
 \sup_s\int\eta^2r^2(s)
 +c_R\iint\eta^2|\nabla r|^2
 +c_R\varepsilon_R\iint\eta^2r^2\sum_{j\ne i}|\nabla z_j|^2
 \le C_R\iint r^2(|\nabla\eta|^2+\eta|\partial_t\eta|).
\]
Standard spatial and temporal cutoffs between $Q_\rho^\Omega$ and $Q_\varrho^\Omega$ give \eqref{eq:root-Caccioppoli-dossier-v11}.  Compare \cite[Proposition~12.2]{BotheAdditive2026}.
\end{proof}

\begin{lemma}[Two-time stabilized level energy]
\label{lem:two-time-stabilized-energy}
Assume that the strong-solution, compact-range, and Neumann-boundary
hypotheses of \Cref{prop:root-Caccioppoli-dossier} hold on a neighborhood
of the entire cylinder $[t_1,t_2]\times D_\rho(x_0)$, where $t_1<t_2$.
Let $\zeta=\zeta(x)$ be a time-independent spatial cutoff which
vanishes near the artificial boundary and may be nonzero on
$\partial\Omega$.  Then, with $r=(z_i-k)_\pm$ and
$\Psi=\Psi_{i,k}^{\pm}$,
\begin{align}
 &\int \zeta^2\Psi(z({t_2}))
 +c_R\int_{{t_1}}^{{t_2}}\!\!\int \zeta^2|\nabla r|^2
 +c_R\varepsilon_R\int_{{t_1}}^{{t_2}}\!\!\int
     \zeta^2r^2\sum_{j\ne i}|\nabla z_j|^2 \notag\\
 &\qquad\le
 \int \zeta^2\Psi(z({t_1}))
 +C_R\int_{{t_1}}^{{t_2}}\!\!\int r^2|\nabla\zeta|^2.
 \label{eq:two-time-stabilized-energy}
\end{align}
All spatial integrals are taken over the corresponding physical domain
slice.
\end{lemma}

\begin{proof}
Set
\[
 F_\Psi:=\sum_j\alpha_j(z)\Psi_j(z)\nabla z_j.
\]
Multiply \eqref{eq:chain} by $\zeta^2$ and integrate over
$(t_1,t_2)\times D_\rho$.  The contribution from the physical boundary
vanishes because $\partial_\nu z_j=0$ for every $j$, while $\zeta$ vanishes
near the artificial boundary.  Hence
\begin{align}
 &\int_{D_\rho}\zeta^2\Psi(z(t_2))
 +\int_{t_1}^{t_2}\!\!\int_{D_\rho}\zeta^2\mathfrak Q_\Psi
 \notag\\
 &\qquad=
 \int_{D_\rho}\zeta^2\Psi(z(t_1))
 -2\int_{t_1}^{t_2}\!\!\int_{D_\rho}
       \zeta F_\Psi\cdot\nabla\zeta .
 \label{eq:two-time-chain-identity-expanded}
\end{align}
By \eqref{eq:stabilized-flux-bounds} and the boundedness of the modal
coefficients on $R$,
\[
 |F_\Psi|
 \le C_R\left(r|\nabla r|+
          r^2\sum_{j\ne i}|\nabla z_j|\right).
\]
Using \eqref{eq:stabilized-coercive}, Young's inequality gives, with
constants depending only on the structural data and on the fixed
$\varepsilon_R$,
\begin{align*}
 2\zeta|F_\Psi||\nabla\zeta|
 &\le \frac{c_R}{2}\zeta^2|\nabla r|^2
 +\frac{c_R\varepsilon_R}{2}\zeta^2r^2
       \sum_{j\ne i}|\nabla z_j|^2
 +C_Rr^2|\nabla\zeta|^2.
\end{align*}
Indeed, the selected term is estimated as
\[
 C_R\zeta r|\nabla r||\nabla\zeta|
 \le \frac{c_R}{4}\zeta^2|\nabla r|^2
      +C_Rr^2|\nabla\zeta|^2,
\]
and, for each $j\ne i$,
\[
 C_R\zeta r^2|\nabla z_j||\nabla\zeta|
 \le \frac{c_R\varepsilon_R}{2(m-1)}
      \zeta^2r^2|\nabla z_j|^2
      +C_Rr^2|\nabla\zeta|^2.
\]
Insert these bounds into \eqref{eq:two-time-chain-identity-expanded} and absorb the
two dissipation terms into the left-hand side.  This gives
\eqref{eq:two-time-stabilized-energy}, with the coefficient of the initial
energy unchanged.  This is the localized two-time estimate used in
\cite[Lemma~12.3]{BotheAdditive2026}.
\end{proof}

\subsubsection{A scalar logarithmic estimate from the root equation}

Fix one root $u:=z_i$ and abbreviate
\[
 \Gamma_{ij}(z):=\frac{\alpha_i(z)+\alpha_j(z)}{z_i-z_j}.
\]
In divergence form its equation is
\begin{equation}\label{eq:selected-root-scalar-drift}
 u_t-\Div({\alpha_i(z)}\nabla u)
 +\sum_{j\ne i}\Gamma_{ij}(z)\nabla u\cdot\nabla z_j=0.
\end{equation}
All ${\alpha_i}$, ${\alpha_i^{-1}}$ and $\Gamma_{ij}$ are uniformly bounded on the compact
root rectangle.  The following change of orientation treats the upper
and lower phases within the same scalar equation.
If $\varepsilon\in\{-1,1\}$ and $c\in\mathbb R$, then along the given
solution
\begin{equation}\label{eq:oriented-selected-root-equation}
 {\widetilde u}:=\varepsilon z_i+c
 \quad\Longrightarrow\quad
 {\widetilde u_t}-\Div({\alpha_i(z)}\nabla {\widetilde u})
 +\sum_{j\ne i}\Gamma_{ij}(z)\nabla {\widetilde u}\cdot\nabla z_j=0,
 \qquad \partial_\nu {\widetilde u}=0.
\end{equation}
Indeed, this is $\varepsilon$ times the $i$-th root equation; the coefficient
field is evaluated along the original state $z$ and is unchanged.  Thus every
upper-phase argument for an oriented variable is also a lower-phase argument
for the original root.  The next lemma controls the drift in
\eqref{eq:selected-root-scalar-drift}--\eqref{eq:oriented-selected-root-equation}
by the additional weighted dissipation in
\eqref{eq:two-time-stabilized-energy}.

Fix once and for all a support parameter $\sigma_0\in(0,1/2)$; below we
may take $\sigma_0=1/4$.  For $H>0$, $0<\delta<\sigma_0/2$, and
$0\le v\le H$, define the scaled logarithmic square
\begin{align}
 \ell_{H,\delta}(v)
 &:={\left[\log\frac{(1-\sigma_0+\delta)H}
 {H-v+\delta H}\right]}_+,
 \label{eq:scaled-log-ell}\\
 {\Lambda_{H,\delta}}(v)&:=H^2\ell_{H,\delta}(v)^2.
 \label{eq:scaled-log-F}
\end{align}
Thus ${\Lambda_{H,\delta}}=0$ on $[0,\sigma_0H]$.  On the active interval
$(\sigma_0H,H]$, putting
$D(v):=H-v+\delta H$, direct differentiation gives
\begin{align}
 {\Lambda_{H,\delta}}'(v)
 &=\frac{2H^2\ell_{H,\delta}(v)}{D(v)},\notag\\
 {\Lambda_{H,\delta}}''(v)
 &=\frac{2H^2(1+\ell_{H,\delta}(v))}{D(v)^2},
 \label{eq:scaled-log-derivatives}
\end{align}
and hence
\begin{equation}\label{eq:scaled-log-ratio}
 \frac{({\Lambda_{H,\delta}}'(v))^2}{{\Lambda_{H,\delta}}''(v)}
 =2H^2\frac{\ell_{H,\delta}(v)^2}{1+\ell_{H,\delta}(v)}
 \le 2\sigma_0^{-2}L_\delta\,v^2,
 \qquad
 L_\delta:=\log\frac{1-\sigma_0+\delta}{\delta},
\end{equation}
where the quotient is defined as zero on the inactive interval.  By
\eqref{eq:scaled-log-ratio}, the error generated by the root drift is a factor
$L_\delta$ times the same weighted square $v^2|\nabla z_j|^2$ controlled by
the stabilized truncation energy.
The function ${\Lambda_{H,\delta}}$ is $C^{1,1}$ on $[0,H]$; the one-sided second
derivative in \eqref{eq:scaled-log-derivatives} is its weak second derivative
on the active side.  The standard scalar Lipschitz chain rule applies, and
there is no interface measure at $v=\sigma_0H$ because
${\Lambda_{H,\delta}}'(\sigma_0H)=0$.

\begin{lemma}[Scaled logarithmic-square root estimate]
\label{lem:root-logarithmic-estimate}
Assume the strong-root, compact-range, and homogeneous Neumann
hypotheses of \Cref{lem:two-time-stabilized-energy} on the cylinder
$[t_1,t_2]\times D_\rho$, where $t_1<t_2$. Let $k\in[a_i,b_i]$ and
$H>0$. Suppose that $v=(u-k)_+$ or $v=(k-u)_+$ satisfies $0\le v\le H$
on this cylinder. Let $\eta,\zeta$ be time-independent spatial cutoffs with
$0\le\eta,\zeta\le1$, $\zeta\equiv1$ on a neighborhood of
$\operatorname{supp}\eta$, and both vanishing near the artificial spatial
boundary.  Then, for $0<\delta<\sigma_0/2$,
\begin{align}
 &\int \eta^2{\Lambda_{H,\delta}}(v({t_2}))
 +c_R\int_{{t_1}}^{{t_2}}\!\!\int
   \eta^2{\Lambda_{H,\delta}}''(v)|\nabla v|^2
 \notag\\
 &\quad\le
 \int \eta^2{\Lambda_{H,\delta}}(v({t_1}))
 +C_RL_\delta\int \zeta^2v^2({t_1}) \notag\\
 &\qquad
 +C_RL_\delta\int_{{t_1}}^{{t_2}}\!\!\int
 v^2\bigl(|\nabla\eta|^2+|\nabla\zeta|^2\bigr).
 \label{eq:root-logarithmic-estimate}
\end{align}
The constants are independent of $H$, $\delta$, {$t_1$}, and {$t_2$}.  The same
estimate holds on a boundary patch, with no additional boundary term.
\end{lemma}

\begin{proof}
We treat $v=(u-k)_+$; the lower truncation follows from the oriented equation \eqref{eq:oriented-selected-root-equation}.  On the active set test \eqref{eq:selected-root-scalar-drift} by
$\eta^2\Lambda_{H,\delta}'(v)$ and integrate between $t_1$ and $t_2$.  The time derivative gives the difference of the logarithmic energies.  After integration by parts the principal term contains
\[
 \iint \alpha_i\eta^2\Lambda_{H,\delta}''(v)|\nabla v|^2,
\]
which is positive.  The cutoff term and the cross-root drift are bounded by
\[
 C_R\eta|\Lambda'|\,|\nabla v||\nabla\eta|,
 \qquad
 C_R\eta^2|\Lambda'|\,|\nabla v|\sum_{j\ne i}|\nabla z_j|.
\]
Using Young's inequality with the positive $\Lambda''|\nabla v|^2$ term and the elementary ratio estimate \eqref{eq:scaled-log-ratio} gives
\[
 \int\eta^2\Lambda(v(t_2))
 +c_R\iint\eta^2\Lambda''(v)|\nabla v|^2
 \le \int\eta^2\Lambda(v(t_1))
 +C_RL_\delta\iint v^2\Bigl(|\nabla\eta|^2+
 \eta^2\sum_{j\ne i}|\nabla z_j|^2\Bigr).
\]
The last term is controlled by the weighted transverse dissipation of the stabilized entropy.  Apply \Cref{lem:two-time-stabilized-energy} with a cutoff $\zeta\equiv1$ near $\operatorname{supp}\eta$ to control it by the initial $L^2$ level energy and the cutoff error.  Since $\Psi\asymp_R v^2$, this yields \eqref{eq:root-logarithmic-estimate}.  Compare \cite[Lemma~12.4]{BotheAdditive2026}.
\end{proof}

\begin{lemma}[Short-time propagation of a level-set fraction]
\label{lem:root-measure-propagation}
There exist numbers
\[
 \delta_0,\;\lambda_0,\;\epsilon_0,\;\gamma_0\in(0,1),
\]
depending only on the structural data, the compact root rectangle and the
local geometry of $\Omega$, with the following property.  Let
$0<\rho\le\rho_\Omega$, let $t_*$ belong to the positive-time
strong-solution interval, and choose $k,H$ with $H\ge0$ and
\[
 [k,k+H]\subset[a_i,b_i].
\]
Suppose that $[t_*,t_*+\delta_0\rho^2)\times D_\rho$ lies in the
positive-time strong interval, with the same compact-range and Neumann
hypotheses, and that throughout this slab
\[
 u\le k+H.
\]
If
\begin{equation}\label{eq:initial-low-fraction}
 |\{u(t_*)\le k\}\cap D_\rho|\ge\frac12|D_\rho|,
\end{equation}
then, for every $t\in[t_*,t_*+\delta_0\rho^2)$,
\begin{equation}\label{eq:propagated-low-fraction}
 |\{u(t)\le k+(1-\epsilon_0)H\}\cap D_{\lambda_0\rho}|
 \ge\gamma_0|D_{\lambda_0\rho}|.
\end{equation}
The symmetric statement holds when the initial half-measure information is
at the upper side of an interval.
\end{lemma}

\begin{proof}
If $H=0$, the assumption $u\le k$ gives the conclusion immediately.
Hence assume $H>0$.
Take $\sigma_0=1/4$.  By the thin-annulus estimate choose
$\lambda_0\in(0,1)$ so that
$|D_\rho\setminus D_{\lambda_0\rho}|\le |D_\rho|/32$, uniformly in the
center and radius.  Choose buffered cutoffs $\eta,\zeta$ as in
\Cref{lem:root-logarithmic-estimate}, with $\eta=1$ on
$D_{\lambda_0\rho}$ and
$|\nabla\eta|+|\nabla\zeta|\le C/((1-\lambda_0)\rho)$.
Fix
\[
 0<\delta_0\le\min\{(1-\lambda_0)^2,\lambda_0^2/4,1/4\}.
\]
Set $v=(u-k)_+$ and $V=|D_\rho|$.  The initial half-measure assumption gives
$\int\eta^2\Lambda_{H,\delta}(v(t_*))\le H^2L_\delta^2V/2$.
On $\{v>(1-\delta)H\}$ one has
$\ell_{H,\delta}(v)\ge L_\delta-\log2$. Applying the logarithmic
estimate on $[t_*,t]$ for each $t_*<t<t_*+\delta_0\rho^2$ gives
\[
 \frac{|\{v(t)>(1-\delta)H\}\cap D_{\lambda_0\rho}|}{V}
 \le \frac{L_\delta^2/2+C_RL_\delta}
              {(L_\delta-\log2)^2}.
\]
Here $C_R$ is independent of $H,\rho,t,\delta$; the chosen time length has
absorbed the cutoff factor.  Choose $\delta\in(0,1/8)$ so small that the
last quotient is at most $5/8$, which is possible since it tends to $1/2$.
Since $|D_{\lambda_0\rho}|\ge31V/32$, the complementary low set has
measure at least $11V/32$, and hence at least
$|D_{\lambda_0\rho}|/4$. Take $\epsilon_0=\delta$ and $\gamma_0=1/4$.
At $t=t_*$ the initial half-measure condition and the thin-annulus bound
give the same conclusion. The lower-phase assertion follows from
\eqref{eq:oriented-selected-root-equation}.  The choices have the order
$\lambda_0,\delta_0,\delta$ and are independent of the oscillation amplitude.
This proves the propagation estimate of
\cite[Lemma~12.5]{BotheAdditive2026}.
\end{proof}

\subsubsection{Critical mass, shrinking, and oscillation decay}

\begin{lemma}[Critical mass from the quadratic Caccioppoli estimate]
\label{lem:root-critical-mass}
There is a constant $\nu_*=\nu_*(\mathcal D_R)>0$, depending in particular on the Caccioppoli and geometric constants, such that the following holds.  Suppose
$Q_{2\rho}^\Omega(t_0,x_0)$ lies in a positive-time strip and
$0<2\rho\le\rho_\Omega$, let
$[k,k+H]\subset[a_i,b_i]$, let $0\le (u-k)_+\le H$ there, and assume
\[
 \frac{|\{u>k\}\cap Q_{2\rho}^\Omega|}
 {|Q_{2\rho}^\Omega|}\le\nu_*.
\]
Then
\begin{equation}\label{eq:critical-mass-conclusion}
 \operatorname*{ess\,sup}_{Q_\rho^\Omega}(u-k)_+\le\frac H2.
\end{equation}
The analogous lower statement also holds.
\end{lemma}

\begin{proof}
If $H=0$, the conclusion is immediate from $(u-k)_+=0$.
Hence assume $H>0$.
Put $r_n=\rho(1+2^{-n})$ and
$k_n=k+\tfrac H2(1-2^{-n})$.  Let
$A_n=Q_{r_n}^\Omega\cap\{u>k_n\}$, and choose cutoffs $\eta_n$ which are
one on $Q_{r_{n+1}}^\Omega$ and supported in $Q_{r_n}^\Omega$, with the
usual bounds $|\nabla\eta_n|^2+|\partial_t\eta_n|\le C4^n/\rho^2$.
The Caccioppoli estimate applied to $(u-k_n)_+$ gives
\[
 \operatorname*{ess\,sup}_t\int|\eta_n(u-k_n)_+|^2
 +\iint|\nabla(\eta_n(u-k_n)_+)|^2
 \le C_R4^nH^2\rho^{-2}|A_n|.
\]
Use \eqref{eq:boundary-parabolic-sobolev-DG} and
$k_{n+1}-k_n=H2^{-n-2}$.  With
$Y_n=|A_n|/|Q_{2\rho}^\Omega|$ and $q=1+2/d$, this yields
\[
 Y_{n+1}\le A B^nY_n^q
\]
for structural $A\ge1$ and $B>1$, enlarging $B$ if necessary.
The normalization removes $H$ and $\rho$:
$dq=d+2$ and $|Q_{2\rho}^\Omega|\asymp\rho^{d+2}$.
Writing $\varepsilon=q-1$, induction gives
\[
 Y_n\le A^{-1/\varepsilon}
              B^{-n/\varepsilon-1/\varepsilon^2}
 \quad\hbox{if}\quad
 Y_0\le A^{-1/\varepsilon}B^{-1/\varepsilon^2}.
\]
This choice of $\nu_*$ forces $Y_n\to0$.  The limiting cylinder is
$Q_\rho^\Omega$ and the limiting level is $k+H/2$, proving the assertion.
This also displays the coefficient dependence of the threshold in
\cite[Lemma~12.6]{BotheAdditive2026}.
\end{proof}

\begin{lemma}[De Giorgi shrinking under a persistent opposite phase]
\label{lem:root-shrinking}
Fix $\gamma\in(0,1)$, $\delta\in(0,1)$, and a spatial buffer factor
$\Lambda>1$.  Let
$I=(s,s+\delta\rho^2)$ and suppose
\begin{equation}\label{eq:persistent-low-phase}
 |\{u(t)\le k_0\}\cap D_\rho|\ge\gamma|D_\rho|
 \qquad\text{for every }t\in I.
\end{equation}
Assume also that $0<\Lambda\rho\le\rho_\Omega$ and that
$D_{\Lambda\rho}$ lies in the same buffered local chart,
\[
 M\ge0,\qquad [k_0,k_0+M]\subset[a_i,b_i],
\]
and
\[
 u\le k_0+M
 \qquad\text{on }I\times D_{\Lambda\rho}.
\]
Define
\[
 k_j:=k_0+M\bigl(1-2^{-j}\bigr),\qquad j=0,1,2,\ldots.
\]
Then
\begin{equation}\label{eq:shrinking-measure-estimate}
 \frac{|\{u>k_J\}\cap(I\times D_\rho)|}
 {|I\times D_\rho|}
 \le \frac{C_R(\Lambda)}{\gamma\sqrt{\delta J}},
 \qquad J\ge1.
\end{equation}
The same statement holds with the signs reversed.
\end{lemma}

\begin{proof}
If $M=0$, then $u\le k_0=k_J$, so the high-level set in the
conclusion is empty.  Hence assume $M>0$.
Write $Q=I\times D_\rho$, $H_j=M2^{-j}$,
$E_j=Q\cap\{k_j<u<k_{j+1}\}$, and $A_j=Q\cap\{u>k_j\}$.
Choose a spatial cutoff which is one on $D_\rho$, supported in
$D_{\Lambda\rho}$, and has gradient bounded by $C_\Lambda/\rho$.
Apply the two-time estimate on $(s_n,t)$, where $s_n\downarrow s$ and
$s_n<t<s+\delta\rho^2$.  The amplitude bound gives an initial energy
bounded by $C_R(\Lambda)H_j^2\rho^d$, independently of $n$, while the cutoff
term is bounded by $C_R(\Lambda)H_j^2\delta\rho^d$.
Letting $s_n\downarrow s$ and then $t\uparrow s+\delta\rho^2$ yields
\[
 \iint_Q|\nabla(u-k_j)_+|^2
 \le C_R(\Lambda)H_j^2\rho^d.
\]
This argument does not require an initial trace at $s$ or a solution on an
earlier time interval.

Apply \eqref{eq:relative-DeGiorgi-isoperimetric} at each time to the levels
$k_j,k_{j+1}$.  The persistent low set is contained in $\{u\le k_j\}$,
so, after integrating in time and using Cauchy--Schwarz,
\[
 |A_{j+1}|
 \le \frac{C_R\rho}{\gamma H_j}
       \left(\iint_Q|\nabla(u-k_j)_+|^2\right)^{1/2}|E_j|^{1/2},
 \qquad
 |A_{j+1}|^2\le
       \frac{C_R(\Lambda)}{\gamma^2}\rho^{d+2}|E_j|.
\]
The sets $E_j$ are disjoint and $A_J\subset A_{j+1}$ for $j<J$.
Summation and $|Q|\asymp\delta\rho^{d+2}$ give
\[
 J|A_J|^2\le
  \frac{C_R(\Lambda)}{\gamma^2}\rho^{d+2}|Q|,
\]
which is \eqref{eq:shrinking-measure-estimate}.  The argument is the
scalar shrinking step of \cite[Lemma~12.7]{BotheAdditive2026}; the lower
phase uses the already established lower truncations.
\end{proof}

\begin{proposition}[One-step oscillation reduction]
\label{prop:root-oscillation-reduction}
There are structural constants $\vartheta_*,\eta_*\in(0,1)$ and
$\rho_*>0$ such that the following holds.  Let $u=z_i$ be any selected root
and let $Q_{4\rho}^\Omega(t_0,x_0)$ be any interior or boundary cylinder
contained in a positive-time strip, with $0<\rho\le\rho_*$.  The radius
$\rho_*$ is chosen so that $4\rho_*\le\rho_\Omega$, hence every spatial set
used below belongs to one of the uniform local-geometry charts.  Then
\begin{equation}\label{eq:one-step-oscillation-reduction}
 \operatorname*{ess\,osc}_{Q_{\vartheta_*\rho}^\Omega(t_0,x_0)}u
 \le(1-\eta_*)
 \operatorname*{ess\,osc}_{Q_{4\rho}^\Omega(t_0,x_0)}u.
\end{equation}
The constants are independent of the distance of $t_0$ to a finite upper
time face.
\end{proposition}

\begin{proof}
Let $\mu^+,\mu^-$ be the essential extrema on $Q_{4\rho}^\Omega$ and set
$\omega=\mu^+-\mu^-$.  The case $\omega=0$ is immediate.  At
$t_*=t_0-\delta_0\rho^2$, one side of the midpoint
$(\mu^++\mu^-)/2$ occupies at least half of $D_\rho$.  Suppose first it
is the lower side.  Apply \Cref{lem:root-measure-propagation} with
$k=(\mu^++\mu^-)/2$ and $H=\omega/2$.  Throughout
$I=(t_*,t_0)$ the level
\[
 k_0=\mu^+-M,\qquad M=\epsilon_0\omega/2,
\]
has a low set of fraction at least $\gamma_0$ in
$D_{\lambda_0\rho}$.  The shrinking lemma applies there with spatial
buffer $D_\rho$, aspect ratio $\delta_0/\lambda_0^2\in(0,1)$, and upper
bound $\mu^+$.

Put
\[
 \rho_c=\tfrac18\min\{\lambda_0,\sqrt{\delta_0}\}\rho.
\]
Then $Q_{2\rho_c}^\Omega\subset I\times D_{\lambda_0\rho}$, and the
ratio of its volume to that of the terminal slab is bounded below by a
structural $\chi_*>0$.  Shrinking gives a structural $C_*$ such that
\[
 \frac{|\{u>\mu^+-M2^{-J}\}\cap(I\times D_{\lambda_0\rho})|}
      {|I\times D_{\lambda_0\rho}|}\le C_*J^{-1/2}.
\]
Choose a fixed integer $J$ with $C_*J^{-1/2}\le\nu_*\chi_*$.  The
high-set fraction in $Q_{2\rho_c}^\Omega$ is then at most $\nu_*$, so
\Cref{lem:root-critical-mass} implies
\[
 \operatorname*{ess\,sup}_{Q_{\rho_c}^\Omega}u
 \le\mu^+-M2^{-J-1}.
\]
Thus \eqref{eq:one-step-oscillation-reduction} holds with
$\vartheta_*=\rho_c/\rho$ and $\eta_*=\epsilon_0 2^{-J-2}$.
The other midpoint alternative uses the lower versions of the same three
lemmas.  All levels stay in $[\mu^-,\mu^+]\subset[a_i,b_i]$.
The propagation step uses estimates on $[t_*,t]$ with $t<t_0$.
The volume ratio is structural and independent of the oscillation size.
This proves the oscillation estimate corresponding to
\cite[Proposition~12.8]{BotheAdditive2026}.
\end{proof}

{
\begin{lemma}[H\"older regularity up to the physical boundary and terminal time]
\label{lem:oscillation-to-terminal-holder}
Let $u$ be continuous on $[\tau,T)\times\overline\Omega$ and satisfy
$|u|\le M_R$ there.  Assume that there are $q,\sigma\in(0,1)$ and $r_*>0$
such that, whenever $x_0\in\overline\Omega$, $0<r\le r_*$ and
$(t_0-r^2,t_0)\subset(\tau,T)$,
\begin{equation}\label{eq:abstract-oscillation-contraction-blue}
 \operatorname*{osc}_{Q_{qr}^\Omega(t_0,x_0)}u
 \le \sigma\,
 \operatorname*{osc}_{Q_r^\Omega(t_0,x_0)}u .
\end{equation}
Assume also the uniform boundary-volume and local-geometry bounds of
\Cref{lem:uniform-local-geometry-DG}.  Then, for every $\tau_1>\tau$, there
are $\theta\in(0,1)$ and $C<\infty$, depending only on
$q,\sigma,r_*,M_R,\tau_1-\tau$ and the uniform geometry, such that
\begin{equation}\label{eq:abstract-terminal-holder-blue}
 |u(t,x)-u(s,y)|
 \le C\bigl(|x-y|+|t-s|^{1/2}\bigr)^\theta
\end{equation}
for all $(t,x),(s,y)\in[\tau_1,T)\times\overline\Omega$.  If $T<\infty$,
the constant is independent of the distances of $t$ and $s$ from the upper
time face $T$.
\end{lemma}

\begin{proof}
Choose
\[
 0<r_0\le\min\{r_*,\sqrt{(\tau_1-\tau)/2}\},\qquad
 \theta=\min\left\{\tfrac12,
             \frac{\log(1/\sigma)}{\log(1/q)}\right\}>0.
\]
Iteration and comparison with the nearest radius $q^nr_0$ give
\[
 \operatorname*{osc}_{Q_r^\Omega(t,x)}u
 \le C M_R(r/r_0)^\theta,
 \qquad 0<r\le r_0,\quad t\ge\tau_1.
\]
For two points, suppose $t\ge s$ and put
$\ell=|x-y|+\sqrt{t-s}$.  If $0<2\ell<r_0$, both points belong to the
closure of $Q_{2\ell}^\Omega(t,x)$.  Continuity extends the oscillation
bound to its physical boundary and upper time face, and gives the desired
estimate.  For $2\ell\ge r_0$, boundedness suffices.  For coincident points the estimate is immediate.  The comparison cylinder ends at the later time and never
uses a value at or after $T$; this is the terminal-cylinder argument of
\cite[Lemma~12.9]{BotheAdditive2026}.
\end{proof}
}

\begin{theorem}[De Giorgi regularity of the roots]
\label{thm:DG}
Let $\Omega\subset\mathbb R^d$ be a bounded domain of class
$C^2$, and fix $p>d+2$.  Let $z$ be a positive-time strong root solution: on
every compact strip $I\Subset(\tau,T)$ it belongs to
\[
 W^{1,p}(I;L^p(\Omega;\mathbb R^{N-1}))
 \cap L^p(I;W^{2,p}(\Omega;\mathbb R^{N-1}))
 \cap C(\overline I;B_{p,p}^{2-2/p}(\Omega;\mathbb R^{N-1})),
\]
satisfies the distributional Stieltjes root system, and has homogeneous
Neumann trace.  Assume that its range on
$[\tau,T)\times\overline\Omega$ lies in a compact root rectangle $R$.
Then there are
$\theta=\theta(\mathcal D_R)\in(0,1)$ and, for every $\tau_1>\tau$, a
constant $C=C(\mathcal D_R,\tau_1-\tau)<\infty$ such that
\begin{equation}\label{eq:root-holder}
 \|z\|_{C_{t,x}^{\theta/2,\theta}
 ([\tau_1,T)\times\overline\Omega)}\le C.
\end{equation}
If $T<\infty$, the constant is independent of the distance to the upper time
face $T$.
\end{theorem}

{
\begin{proof}
The coefficient-dependent chain identities, stabilized energies, and scalar
iteration have been established above by the argument of
\cite[Sections~11--12]{BotheAdditive2026}.  For the cylinder convention of the terminal lemma, set
\[
 q=\vartheta_*/4,\qquad \sigma=1-\eta_*,\qquad r_*=4\rho_*.
\]
Apply \Cref{prop:root-oscillation-reduction} with $\rho=r/4$.
For every root this gives exactly
\eqref{eq:abstract-oscillation-contraction-blue} on $Q_r^\Omega$.
The strong-solution trace embedding makes the roots continuous on every
compact positive-time strip, so essential and ordinary oscillations agree.
Given $\tau_1>\tau$, apply \Cref{lem:oscillation-to-terminal-holder}
with initial tail time $\tau'=(\tau+\tau_1)/2$.  Continuity at $\tau'$
is available from the strong class; no trace at $\tau$ or at $T$ is needed.
This yields \eqref{eq:root-holder}, with the asserted structural dependence
and waiting-time gap.  All comparison cylinders end at the later time and
use only earlier values, proving uniformity at a finite upper time face.
\end{proof}
}

\begin{remark}[Dependence of the oscillation constants]
\label{rem:DG-constant-ledger}
Apart from the auxiliary parameters explicitly specified in each statement, the constants in the Caccioppoli, logarithmic, propagation, shrinking, critical-mass, and oscillation estimates depend only on $\mathcal D_R$.  In particular, this dependence includes the positive ellipticity lower bound $d_*$, the upper coefficient bounds, and the root-gap and entropy bounds supplied by \Cref{lem:coefficient-ledger}.  The H\"older exponent in \Cref{thm:DG} inherits the coefficient dependence of the oscillation contraction, while the H\"older norm also depends on the positive waiting-time gap.  These constants do \emph{not} depend on $H$, on the oscillation size, on the chosen level, on the time location of the cylinder, or on any a-priori $C^1$ norm of the strong solution.  The latter regularity is used only to
justify the strong chain and testing identities on compact positive-time
strips.
The time-propagation estimate used in the oscillation argument is
\Cref{lem:root-logarithmic-estimate}.  Its logarithmic temporal energy is obtained
directly from the scalar root equation, while the cross-root drift is
controlled by the additional stabilized dissipation
$r^2\sum_{j\ne i}|\nabla z_j|^2$.
\end{remark}

\section{Continuation, positive-time regularity, and relaxation}
\label{sec:continuation-block}

The invariant root rectangle provides a compact interior composition range.
The H\"older bound in \Cref{thm:DG} supplies the modulus used in two
different localization steps.
\begin{enumerate}[label=(\roman*),leftmargin=*,itemsep=2pt,topsep=3pt]
\item \emph{Fix the entropy metric at one state on each cylinder.}
In affine moment coordinates, the constant similarity defined by this
metric makes the variable-coefficient divergence-form system uniformly
coercive on that cylinder. The first-order conormal estimate gives a
spatial Lipschitz bound on positive-time strips.
\item \emph{Freeze the principal operator at one time on each short interval.}
The spatial Lipschitz bound controls the quadratic forcing. Time H\"older
control makes the change of the principal operator small in
\(\mathcal L(X_1,X_0)\). Maximal regularity, followed by a time-trace
estimate, gives a compatible endpoint datum at a hypothetical finite
maximal time.
\end{enumerate}
These two steps yield global continuation. The second-order conormal
estimate then gives uniform classical regularity on positive-time tails
and supplies the derivative bound used in the \(C^1\) relaxation estimate.

\subsection{Affine-moment conormal regularity}
\label{sec:conormal}

By Theorem~\ref{thm:root-rectangle}, the concentration range lies in a compact
$\cK\Subset\cS^\circ$.  The root map, the inverse root map, and the affine
moment map are smooth there.  For any $0<\tau<T$, apply Theorem~\ref{thm:DG} on the strip $[\tau/2,T)$ to obtain the estimate from time $\tau$.  For the corresponding compact moment range $\mathcal K_M\Subset\mathcal T^\circ$ this yields
\begin{equation}
 \|U\|_{C_{t,x}^{\theta/2,\theta}([\tau,T)\times\overline\Omega)}\le C.
 \label{eq:u-holder}
\end{equation}
By \Cref{rem:DG-constant-ledger}, one may take
$C=C(\mathcal D_R,\tau)$ in \eqref{eq:u-holder}, uniformly in $T$; in
particular, this constant is independent of
$\sup_{t\in[\tau,T)}\|\nabla u(t)\|_{L^\infty(\Omega)}$.
This is the same transition from spectral roots to affine system coordinates
as in the additive-friction proof; see
\cite[Section~13]{BotheAdditive2026}.

\begin{lemma}[Uniform local frozen symmetrizer]\label{lem:frozen-symmetrizer}
There exist $\varepsilon_*>0$ and $\nu_*>0$, depending only on the compact moment range $\mathcal K_M$ and the structural data $\mathcal D_R$ of the fixed root rectangle, such that
\[
 \sym\bigl(H_M(U_*)A_M(U)\bigr)\ge \nu_* I
\]
whenever $U,U_*\in\mathcal K_M$ and $|U-U_*|\le\varepsilon_*$.  The constants are uniform in the frozen state $U_*$.
\end{lemma}
\begin{proof}
At $U=U_*$, \eqref{eq:core-moment-symmetrizer} is symmetric positive definite.  Its smallest eigenvalue has a positive minimum on the compact set $\mathcal K_M$.  The matrices $H_M(U_*)$ are uniformly bounded there and $A_M$ is uniformly continuous, so the symmetric part remains positive on one common state-space ball.  This is the frozen-entropy argument used in \cite[Lemma~13.1]{BotheAdditive2026}.
\end{proof}

Fix $U_*\in\mathcal K_M$ and set
\[
 H_*=H_M(U_*),\qquad S_*=H_*^{1/2}.
\]
By \Cref{lem:frozen-symmetrizer}, after restricting to a cylinder on which
$|U-U_*|\le\varepsilon_*$, one has
\begin{equation}
 \sym(H_*A_M(U))\ge\nu_* I.
 \label{eq:frozen-sym}
\end{equation}
The H\"older modulus \eqref{eq:u-holder} determines a cylinder radius with this property uniformly along the trajectory.

Set
\[
 v=S_*(U-U_*).
\]
Then
\begin{equation}
 v_t-\Div\bigl(S_*A_M(U)S_*^{-1}\nabla v\bigr)=0.
 \label{eq:v-system}
\end{equation}
Furthermore
\[
 \sym(S_*A_M(U)S_*^{-1})
 =S_*^{-1}\sym(H_*A_M(U))S_*^{-1},
\]
so \eqref{eq:v-system} is uniformly strongly parabolic in Euclidean component
space.

\paragraph{Quantitative conormal estimates.}
We use \Cref{lem:DJ-quantitative-first,prop:DJ-quantitative-second}, the
quantitative forms of the unweighted H\"older estimates of
\cite[Theorems~4.1 and~4.2]{DongJeon2026} needed for zero divergence datum.
\Cref{app:conormal-quantitative} retains the coefficient-modulus
dependence in the first-order estimate by starting from
\cite[Lemma~3.4, equation~(3.15)]{DongJeon2026}.
For a bounded weak solution of a uniformly strongly parabolic
divergence-form system with homogeneous conormal data on a flat face, the
two estimates give, respectively,
\[
 \begin{array}{ccl}
 \mathbb A\in C^{\bar\beta/2,\bar\beta}
 &\Longrightarrow&v\in C^{(1+\beta)/2,1+\beta},\\
 \mathbb A\in C^{(1+\bar\beta)/2,1+\bar\beta}
 &\Longrightarrow&v\in C^{1+\beta/2,2+\beta},
 \end{array}
 \qquad 0<\beta<\bar\beta<1,
\]
on a smaller half-cylinder.  More precisely, on normalized nested
half-cylinders $Q_{1/2}^+\subset Q_1^+$,
\Cref{lem:DJ-quantitative-first,prop:DJ-quantitative-second} give
\[
 \|v\|_{C^{(k+\beta)/2,k+\beta}(Q_{1/2}^+)}
 \le C\|v\|_{L^\infty(Q_1^+)},
 \qquad k=1,2,
\]
provided
\[
 \|\mathbb A\|_{C^{(k-1+\bar\beta)/2,k-1+\bar\beta}(Q_1^+)}
 \le M_A.
\]
Here $C$ depends only on $k$, the spatial and component dimensions, the
parabolicity and coefficient bounds, $M_A$, $\beta$, $\bar\beta$, and the
fixed cylinder ratio. In particular, it is independent of any prior
gradient bound for $v$. The $k=2$ estimate uses normalized time
differences and requires no time derivative of the coefficient tensor.

For the interior estimate, take a solution $v(t,y)$ on a full cylinder
and introduce an additional spatial coordinate $r>0$. Set
$V(t,y,r)=v(t,y)$ and extend the tensor by the identity in the $r$
direction and zero mixed blocks. The extended system is strongly
parabolic in one higher space dimension and has zero conormal derivative
at $r=0$. Apply the half-cylinder estimate and restrict to obtain the
interior estimate.

\begin{lemma}[First-order conormal Schauder estimate]
\label{lem:lipschitz}
Let $0<\beta<\min\{\theta,\alpha\}$, where $\theta$ is furnished by
Theorem~\ref{thm:DG}.  For every $0<\tau<\tau_1<T$ there is $C_{\tau_1}$, depending on the structural data, the bound \eqref{eq:u-holder}, and the gap $\tau_1-\tau$, such that
\begin{equation}
 \|U\|_{C_{t,x}^{(1+\beta)/2,1+\beta}
 ([\tau_1,T)\times\overline\Omega)}\le C_{\tau_1}.
 \label{eq:first-conormal-gain}
\end{equation}
In particular, because $u$ and $U$ are related by a fixed affine bijection,
\begin{equation}
 \sup_{(t,x)\in[\tau_1,T)\times\Omega}
 \bigl(|\nabla U(t,x)|+|\nabla u(t,x)|\bigr)\le C_{\tau_1}.
 \label{eq:lipschitz}
\end{equation}
If $T<\infty$, the constant is independent of the distance to the terminal
face.
\end{lemma}
\begin{proof}
We use the following local reduction, as in \cite[Lemma~13.3]{BotheAdditive2026}.  Fix a point $(t_0,x_0)$ on a positive-time tail and let $U_*=U(t_0,x_0)$.  By \eqref{eq:u-holder}, for a radius $r_*>0$ depending only on the uniform H\"older modulus and the compact moment range,
\[
 |U(t,x)-U_*|\le\varepsilon_*
\]
throughout every interior or boundary parabolic cylinder of radius at most $r_*$.  Hence \Cref{lem:frozen-symmetrizer} gives the uniform coercivity \eqref{eq:frozen-sym} there.  Reduce $r_*$ further according to $\tau_1-\tau$, if necessary, so that every enlarged backward cylinder used for centers in $[\tau_1,T)$ stays above the lower time face $\tau$.

With $S_*=H_M(U_*)^{1/2}$ and $v=S_*(U-U_*)$, the equation becomes \eqref{eq:v-system}.  Its coefficient
\[
 \widetilde A(t,x)=S_*A_M(U(t,x))S_*^{-1}
\]
has a uniformly positive symmetric part and belongs to $C^{\theta/2,\theta}$ with a norm independent of the center of the cylinder.  This gives the required tensor on interior cylinders.  We include the boundary transformation explicitly to verify that the present moment system satisfies the hypotheses of the quantitative conormal estimate.

Fix a boundary point $x_*$ and, after an orientation-preserving rotation
$Q_*\in SO(d)$, represent the boundary locally as a $C^{2+\alpha}$ graph
$y_d=\psi(y')$.  Use the graph chart
\begin{equation}\label{eq:core-boundary-chart}
 \Theta(\xi)=x_*+Q_*\bigl(\xi',\xi_d+\psi(\xi')\bigr),
 \qquad \det D\Theta=1,
\end{equation}
and put
\[
 G(\xi)=(D\Theta(\xi))^{-1}(D\Theta(\xi))^{-\sf T},
 \qquad \widehat v(t,\xi)=v(t,\Theta(\xi)).
\]
The determinant in \eqref{eq:core-boundary-chart} is one because the graph
map has triangular derivative with unit diagonal and $Q_*$ has determinant
one.  Hence the Piola identity introduces no Jacobian weight in either the
time term or the divergence term.  The transformed equation is exactly
\begin{equation}\label{eq:core-flattened-system}
 \widehat v_t-
 \partial_{\xi_\gamma}\!\left(
 \mathbb A^{\gamma\delta}(t,\xi)
 \partial_{\xi_\delta}\widehat v\right)=0,
 \qquad
 \mathbb A^{\gamma\delta}(t,\xi)
 =G^{\gamma\delta}(\xi)
   \widetilde A(t,\Theta(\xi)).
\end{equation}
No Jacobian weight or term outside the displayed divergence form is introduced.

The original boundary condition is $\partial_\nu U=0$.  Since $S_*$ is
constant, it gives $\partial_\nu v=0$, and therefore the physical flux
$\widetilde A\nabla v$ has zero normal component.  The same Piola transform
then yields on the flat face $\{\xi_d=0\}$ the exact homogeneous conormal
condition associated with \eqref{eq:core-flattened-system},
\begin{equation}\label{eq:core-flat-conormal}
 \sum_{\gamma,\delta=1}^d
 (e_d)_\gamma\,\mathbb A^{\gamma\delta}
 \partial_{\xi_\delta}\widehat v=0.
\end{equation}
Thus the boundary condition required by the conormal Schauder theorem is not
an additional assumption; it is the transformed Maxwell--Stefan no-flux
condition.

To verify strong parabolicity, decompose
$\widetilde A=\widetilde A_{\rm s}+\widetilde A_{\rm a}$ into its symmetric
and skew-symmetric component-space parts.  For every spatial-component array
$\zeta=(\zeta_\gamma)_{\gamma=1}^d$,
\[
 \sum_{\gamma,\delta}
 G^{\gamma\delta}\zeta_\gamma^{\sf T}
 \widetilde A_{\rm a}\zeta_\delta=0
\]
because $G^{\gamma\delta}=G^{\delta\gamma}$.  Consequently
\[
 \sum_{\gamma,\delta}
 \zeta_\gamma^{\sf T}\mathbb A^{\gamma\delta}\zeta_\delta
 =\sum_{\gamma,\delta}G^{\gamma\delta}
 \zeta_\gamma^{\sf T}\sym\widetilde A\,\zeta_\delta,
\]
and \eqref{eq:frozen-sym}, together with the uniform positive definiteness of
$G$ on the finite chart family, gives one uniform strong-parabolicity
constant.  Moreover $G\in C^{1+\alpha}$ and
$\widetilde A\in C^{\theta/2,\theta}$.  Since the exponent in the statement
satisfies $\beta<\min\{\theta,\alpha\}$, choose an intermediate exponent
\[
 \beta<\bar\beta<\min\{\theta,\alpha\}.
\]
Then $\mathbb A\in C^{\bar\beta/2,\bar\beta}$, with a norm uniform in the
cylinder center.

On every finite cylinder compactly contained in the strong time interval,
$p>d+2>2$ and maximal-$L^p$ regularity give
$\widehat v,\nabla_{\xi}\widehat v,\partial_t\widehat v\in L^2$.
The similarity is constant and the spatial chart is independent of time,
so both preserve this local regularity. Equations
\eqref{eq:core-flattened-system} and \eqref{eq:core-flat-conormal} give the
weak conormal identity for test functions that can meet the flat face.
Compactness of the moment range gives the uniform $L^\infty$ bound after
the constant similarity.  Thus all hypotheses of \Cref{lem:DJ-quantitative-first} are satisfied.  Apply that quantitative estimate with coefficient exponent $\bar\beta$ and solution-regularity exponent $\beta<\bar\beta$.  Parabolic rescaling gives a uniform $C^{(1+\beta)/2,1+\beta}$ estimate on a concentric smaller boundary half-cylinder; the additional-coordinate reduction gives the corresponding interior estimate.

Applying the inverse constant similarity and boundary chart gives the same estimate
for $U$.  A finite spatial covering and overlapping backward cylinders
propagate the local estimates along $[\tau_1,T)$; only backward cylinders are
used, so the constants remain uniform when $T<\infty$ and the upper time face
is approached.  Finally the moment map is a fixed affine bijection, so the
same first-order bound holds for $u$.
\end{proof}

\subsection{Terminal trace bound and global continuation}
\label{sec:terminal}

\begin{lemma}[Terminal trace estimate under time H\"older and spatial Lipschitz bounds]
\label{lem:terminal-trace}
Let $b<\infty$ and let $u$ be a maximal-$L^p$ strong solution of
\eqref{eq:nondiv-u} on $[a,b)$ such that
\[
 u([a,b)\times\overline\Omega)\subset\cK\Subset\mathcal U,
\]
\[
 [u]_{C_t^{\delta_t}([a,b);C(\overline\Omega))}\le M_0,
 \qquad
 \sup_{a\le t<b}\|\nabla u(t)\|_{L^\infty}\le M_1
\]
for some $\delta_t>0$.  Then
\begin{equation}
 \sup_{a\le t<b}\|u(t)\|_{X_{\gamma,p}}<\infty.
 \label{eq:terminal-trace-abstract}
\end{equation}
The bound is independent of the distance of $t$ to $b$.
Moreover $u\in W^{1,p}(a,b;X_0)\cap L^p(a,b;X_1)$, and
$u(t)\to u^b$ in $X_{\gamma,p}$ as $t\uparrow b$ for some
$u^b$ with $u^b(\overline\Omega)\subset\cK$.
\end{lemma}
\begin{proof}
We adapt the short-interval freezing and trace recursion of \cite[Lemma~13.4]{BotheAdditive2026} to the reduced core-diagonal concentration operator.

For a frozen time $s$ put
\[
 \mathcal A(s)v=-C(u(s,\cdot))\Delta v,
 \qquad D(\mathcal A(s))=X_1,
 \qquad F(t):=F(u(t),\nabla u(t)).
\]
The compact state range and the spatial Lipschitz bound imply
\[
 \sup_{a\le t<b}\|F(t)\|_{X_0}\le C_F
\]
and give one common spatial Lipschitz constant for the coefficient fields
$C(u(t,\cdot))$.  Hence, by Arzel\`a--Ascoli, their closure $\mathfrak F$ is compact in
$C(\overline\Omega;\R^{m\times m})$.  Each $K\in\mathfrak F$ retains the
common spatial Lipschitz bound and takes values in the same compact
normally elliptic matrix family.  Its realization
$\mathcal A_Kv=-K\Delta v$ has the \emph{same} domain $X_1$ and the
Neumann Lopatinskii--Shapiro condition of Lemma~\ref{lem:LS}.
Only continuity of the principal coefficients is required in
\cite[Assumption~5.1 and Theorem~5.2]{Augner2024}; no convergence of their
spatial derivatives is needed.  That theorem gives maximal regularity
for a shifted realization associated with each $K$.  On a fixed finite
horizon $(0,T_0)$ the exponential conjugation used in
Lemma~\ref{lem:uniform-MR} removes the shift.

For completeness, write
\[
 \mathbb E_0=\{v\in W^{1,p}(0,T_0;X_0)\cap L^p(0,T_0;X_1):v(0)=0\},
 \qquad \mathcal L_Kv=v_t+\mathcal A_Kv.
\]
These are isomorphisms from the common space $\mathbb E_0$ onto
$L^p(0,T_0;X_0)$, and
\[
 \|\mathcal L_K-\mathcal L_{K_0}\|_{\mathcal L(\mathbb E_0,L^p(0,T_0;X_0))}
 \le c\|K-K_0\|_\infty.
\]
Thus $c\|K-K_0\|_\infty\|\mathcal L_{K_0}^{-1}\|\le1/2$ implies
$\|\mathcal L_K^{-1}\|\le2\|\mathcal L_{K_0}^{-1}\|$ by the Neumann
series.  A finite covering of $\mathfrak F$ gives one inverse bound.
Extending a forcing by zero to $(0,T_0)$ and restricting its solution
then gives the same zero-trace estimate on every shorter interval;
uniqueness identifies the restriction.  No terminal coefficient field
or terminal solution trace is assumed in this argument.

For $I=(s,s+h)$ write
\[
 \mathbb E_I:=W^{1,p}(I;X_0)\cap L^p(I;X_1),
\]
with its ordinary intersection norm.  Choose a bounded coretraction of the trace on the fixed horizon $(0,T_0)$.  Subtracting the restriction of this coretraction from a solution, and applying the uniform zero-trace estimate to the remainder, gives
\begin{equation}
 \|v\|_{\mathbb E_I}
 \le C_{MR}\bigl(\|v(s)\|_{X_{\gamma,p}}+\|f\|_{L^p(I;X_0)}\bigr)
 \label{eq:uniform-nonzero-MR-core}
\end{equation}
whenever $v_t+\mathcal A(s)v=f$, $I=(s,s+h)$, and $0<h\le T_0$.
Here $C_{MR}$ is independent of $s$ and $h$: the coretraction is chosen on the fixed horizon, its restriction has no larger maximal-regularity norm, and the frozen operators are uniformly bounded in $\mathcal L(X_1,X_0)$.

The corresponding uniform trace estimate must retain the initial trace:
\begin{equation}
 \|v\|_{C([s,s+h];X_{\gamma,p})}
 \le C_{\rm tr}\bigl(\|v\|_{\mathbb E_I}
                  +\|v(s)\|_{X_{\gamma,p}}\bigr),
 \qquad 0<h\le T_0,
 \label{eq:uniform-augmented-trace-core}
\end{equation}
where $C_{\rm tr}\ge1$ is independent of $s,h$. Translate $v$ to
$(0,h)$ and apply the fixed-horizon extension of
\cite[Lemma~7.2]{Amann2005} used for \eqref{eq:uniform-short-time-C1}.
The extension to $(0,T_0)$ has maximal-regularity norm bounded by
$C(\|v\|_{\mathbb E_I}+\|v(s)\|_{X_{\gamma,p}})$, with a constant
independent of $h$. The trace embedding on $(0,T_0)$, followed by
restriction, gives \eqref{eq:uniform-augmented-trace-core}. The
initial-trace term is essential for constants uniform as $h\downarrow0$.
For a time-constant function with a nonzero value in $X_1$, the
maximal-regularity norm scales as $h^{1/p}$ and the trace norm is fixed.

Smoothness of $C$ on $\cK$ and the assumed time H\"older control give
\begin{equation}
 \|\mathcal A(t)-\mathcal A(s)\|_{\mathcal L(X_1,X_0)}
 \le LM_0|t-s|^{\delta_t}.
 \label{eq:terminal-operator-holder}
\end{equation}
Choose $h_0\le T_0$ so that
$C_{MR}LM_0h_0^{\delta_t}\le1/2$.  On any
$I=(s,s+h)$ with $a\le s<s+h<b$ and $h\le h_0$, freeze at $s$:
\[
 u_t+\mathcal A(s)u
 =F(t)+(\mathcal A(s)-\mathcal A(t))u.
\]
Using \eqref{eq:uniform-nonzero-MR-core},
\eqref{eq:terminal-operator-holder}, and the uniform forcing bound, and
absorbing the last term, gives
\begin{equation}
 \|u\|_{\mathbb E_I}
 \le 2C_{MR}\bigl(\|u(s)\|_{X_{\gamma,p}}+C_Fh^{1/p}\bigr).
 \label{eq:terminal-MR}
\end{equation}
Combining \eqref{eq:terminal-MR} with the augmented estimate
\eqref{eq:uniform-augmented-trace-core} gives
\[
 \begin{aligned}
  \sup_{t\in[s,s+h]}\|u(t)\|_{X_{\gamma,p}}
  &\le C_{\rm tr}(1+2C_{MR})\|u(s)\|_{X_{\gamma,p}}\\
  &\quad+2C_{\rm tr}C_{MR}C_Fh_0^{1/p}.
 \end{aligned}
\]
Thus, with
\[
 C_{it}:=C_{\rm tr}(1+2C_{MR})\ge1,
 \qquad C_0:=2C_{\rm tr}C_{MR}C_Fh_0^{1/p},
\]
one has the uniform affine recursion
\begin{equation}
 \sup_{t\in[s,s+h]}\|u(t)\|_{X_{\gamma,p}}
 \le C_{it}\|u(s)\|_{X_{\gamma,p}}+C_0.
 \label{eq:trace-recursion}
\end{equation}
Starting at $a$, at most $1+\lceil(b-a)/h_0\rceil$ successive intervals are needed.  On the final interval, whose upper endpoint is the excluded time $b$, apply the estimate first for every shorter interval ending strictly below $b$ and then take the supremum of their left-hand sides.  No trace at $b$ is assumed.  Iterating the affine recursion a finite number of times proves \eqref{eq:terminal-trace-abstract}; no contractivity condition $C_{it}<1$ is needed.

For the terminal limit, fix $s\in[a,b)$ with
$b-s<h_0$.  Estimate \eqref{eq:terminal-MR} on $(s,r)$ is uniform for
$s<r<b$.  Letting $r\uparrow b$ in the defining integrals gives
$u\in\mathbb E_{(s,b)}$; local strong regularity covers $[a,s]$.
The trace embedding on the fixed interval $(a,b)$ gives the asserted
limit in $X_{\gamma,p}$.  Its range stays in $\cK$ by uniform spatial
convergence.  No terminal trace was assumed.
\end{proof}

To prove global continuation, suppose for contradiction that
$T_{\max}<\infty$.  By the strictly positive initial condition
\eqref{eq:initial-data}, the Stieltjes roots are defined at $t=0$.  Their
initial extrema lie in a compact rectangle strictly inside $\cZ$.
Theorem~\ref{thm:root-rectangle} and Corollary~\ref{cor:species-separation}
therefore place the entire trajectory $[0,T_{\max})$ in one compact
$\cK\Subset\cS^\circ$.  Choose any $0<\tau<T_{\max}/2$.
Apply Theorem~\ref{thm:DG} on $[\tau/4,T_{\max})$ to obtain a bound from time $\tau/2$, and transfer the estimate by the smooth concentration and moment maps.  This gives
\[
 [u]_{C_t^{\theta/2}([\tau/2,T_{\max});C(\overline\Omega))}\le M_0.
\]
Apply Lemma~\ref{lem:lipschitz} on $[\tau/2,T_{\max})$ to obtain the estimate on $[\tau,T_{\max})$.  It gives
\[
 \sup_{t\in[\tau,T_{\max})}\|\nabla u(t)\|_{L^\infty}\le M_1.
\]
Apply \Cref{lem:terminal-trace} with $\delta_t=\theta/2$. It gives
\[
 \sup_{\tau\le t<T_{\max}}\|u(t)\|_{X_{\gamma,p}}<\infty,
 \qquad u(t)\longrightarrow u^b\quad\text{in }X_{\gamma,p}
 \quad\text{as }t\uparrow T_{\max}.
\]
The limit $u^b$ has range in the compact reduced composition set and
is Neumann-compatible because it belongs to $X_{\gamma,p}$. By
\Cref{thm:local}, it starts a new local strong solution at $T_{\max}$.
The original solution is in the maximal-regularity class on
$(\tau,T_{\max})$ by \Cref{lem:terminal-trace}. The two pieces have the
same trace at $T_{\max}$, so their concatenation is in the
maximal-regularity class across that time and solves \eqref{eq:reduced-system}
almost everywhere. This contradicts maximality. Hence
\[
 T_{\max}=\infty.
\]
This proves the global-existence part of Theorem~\ref{thm:main}.

\subsection{Uniform positive-time smoothing and exponential relaxation}
\label{sec:longtime}

\begin{proposition}[Uniform second-order conormal Schauder estimate on positive-time tails]
\label{prop:uniform-tail-smoothing}
Let $x$ be the global solution furnished above.  There exists $\beta\in(0,\alpha)$, independent of $\tau$, such that for every $\tau>0$ there is
$C_\tau<\infty$ with
\begin{equation}
 \sup_{s\ge\tau}
 \|u\|_{C_{t,x}^{1+\beta/2,2+\beta}
 ([s,s+1]\times\overline\Omega)}\le C_\tau.
 \label{eq:uniform-tail-C2}
\end{equation}
The same estimate, after smooth composition on the fixed compact root range,
holds for the concentrations $x$ and the Stieltjes roots $z$.
\end{proposition}
\begin{proof}
The invariant-root theorem applied to the strictly positive initial data places the entire trajectory in one compact root rectangle and hence in one compact $\cK\Subset\cS^\circ$.  Consequently the De Giorgi exponent $\theta$ is fixed by the structural data and this single root rectangle, independently of the waiting time.  Choose once and for all
$0<\beta_1<\min\{\theta,\alpha\}$ and then $0<\beta<\beta_1$.  For each $\tau>0$, apply the De Giorgi theorem on $[\tau/6,\infty)$ to obtain a uniform
$C^{\theta/2,\theta}$ bound on $[\tau/3,\infty)$.  Applying
Lemma~\ref{lem:lipschitz} on backward cylinders of one fixed small radius gives
one uniform
$C^{(1+\beta_1)/2,1+\beta_1}$ bound on $[\tau/2,\infty)$.

Repeat the frozen-symmetrizer and boundary-flattening construction from
\Cref{sec:conormal}.  The first conormal estimate gives $U$ uniformly
$C^{(1+\beta_1)/2,1+\beta_1}$ on $[\tau/2,\infty)$.  Fix the exponent
$0<\beta<\beta_1$ from the statement and choose
\[
 \beta<\bar\beta<\beta_1.
\]
Since $A_M$ is $C^3$ on the compact moment range and the boundary metric of a
$C^{2+\alpha}$ chart is $C^{1+\alpha}$ (with $\beta_1<\alpha$), the
transformed coefficient tensor belongs to
$C^{(1+\bar\beta)/2,1+\bar\beta}$ with a uniform norm.  It is uniformly
strongly parabolic by the same frozen-symmetrizer calculation, the transformed
solution is a bounded weak conormal solution, and the divergence datum is
zero.  Thus \Cref{prop:DJ-quantitative-second}, applied with coefficient exponent
$\bar\beta$ and solution-regularity exponent $\beta<\bar\beta$, gives
\[
 u\in C_{t,x}^{1+\beta/2,2+\beta}
\]
on each concentric smaller boundary half-cylinder.  The additional-coordinate reduction in \Cref{sec:conormal} gives the corresponding estimate on interior cylinders.  A fixed overlapping family of backward cylinders
covers every slab $[s,s+1]$, $s\ge\tau$, and all constants depend only on the
single compact root rectangle, the coefficient bounds, $\tau$, and the domain
charts.  This proves \eqref{eq:uniform-tail-C2}.  The root and concentration
maps have uniformly bounded derivatives of all required orders on the compact
state range, so the estimate transfers by composition.
\end{proof}

Let
\[
 \overline x_i=\frac1{|\Omega|}\int_\Omega x_i^0\,\dd x,
 \qquad i=1,\ldots,N.
\]
Every mean is positive.  Integrating the component balance against the constant test function and using the no-flux trace gives
\[
 \frac{\dd}{\dd t}\int_\Omega x_i(t,y)\,\dd y=0
 \quad(i=1,\ldots,N),
\]
first almost everywhere in time and then in the integrated sense by strong time continuity.  Thus the species means are conserved.  With the
convention $0\log0=0$, define the relative mixing entropy
\begin{equation}
 \mathcal H(x\mid\overline x)
 :=\int_\Omega\sum_{i=1}^N
 x_i\log\frac{x_i}{\overline x_i}\,\dd x.
 \label{eq:relative-entropy}
\end{equation}

\begin{lemma}[Relative entropy identity]
\label{lem:entropy-from-zero}
For every $t\ge0$,
\begin{equation}
 \mathcal H(x(t)\mid\overline x)
 +4\int_0^t\sum_{i=1}^N\int_\Omega
 d_i(x)|\nabla\sqrt{x_i}|^2\,\dd x\,\dd s
 =\mathcal H(x^0\mid\overline x).
 \label{eq:relative-entropy-identity}
\end{equation}
\end{lemma}
\begin{proof}
The invariant root rectangle gives a uniform positive lower bound for every species, so the entropy test is admissible on every finite interval.  Apply \Cref{cor:entropy} and use conservation of each species integral $\int_\Omega x_i$.  The linear terms produced when passing from the mixing entropy to the relative entropy cancel because the means $\overline x_i$ are constant in time.
\end{proof}

We use the bounded-domain logarithmic Sobolev inequality in the normalization
\begin{equation}
 \int_\Omega\rho\log\frac{\rho}{\overline\rho}\,\dd x
 \le C_{LS}(\Omega)\int_\Omega|\nabla\sqrt\rho|^2\,\dd x,
 \qquad
 \overline\rho=|\Omega|^{-1}\int_\Omega\rho.
 \label{eq:LSI}
\end{equation}
It holds on every bounded connected $C^2$ domain for nonnegative $\rho$ with $\sqrt\rho\in H^1(\Omega)$ and positive mean; compare the bounded-domain logarithmic Sobolev step in \cite{JungelStelzer2013}.  The constant here is defined by the displayed normalization, so the factor $4$ below results from the identity $|\nabla\rho|^2/\rho=4|\nabla\sqrt\rho|^2$.  Applying \eqref{eq:LSI} specieswise and using $d_i\ge d_*$ gives
for almost every positive time
\[
 4\sum_i\int d_i|\nabla\sqrt{x_i}|^2
 \ge \frac{4d_*}{C_{LS}(\Omega)}
 \mathcal H(x\mid\overline x).
\]
Therefore, with
\begin{equation}
 \kappa_{LS}:=\frac{4d_*}{C_{LS}(\Omega)},
 \label{eq:kappaLS}
\end{equation}
Gronwall's inequality on $[\varepsilon,t]$ followed by
$\varepsilon\downarrow0$ yields
\begin{equation}
 \mathcal H(x(t)\mid\overline x)
 \le e^{-\kappa_{LS}t}\mathcal H(x^0\mid\overline x).
 \label{eq:entropy-decay}
\end{equation}

For completeness, the scalar Bregman density
\[
 \psi(a\mid b)=a\log(a/b)-a+b
\]
satisfies $\psi''(a)=1/a\ge1$ on $(0,1]$, hence
$\psi(a\mid b)\ge\tfrac12(a-b)^2$, including $a=0$ by continuity.  Since the
linear terms cancel after summing the species,
\begin{equation}
 \|x(t)-\overline x\|_{L^2}
 \le \sqrt{2\mathcal H(x^0\mid\overline x)}
 e^{-\kappa_{LS}t/2}.
 \label{eq:L2-decay}
\end{equation}
Finally Proposition~\ref{prop:uniform-tail-smoothing} gives, for example,
\[
 \sup_{t\ge1}\|x(t)-\overline x\|_{W^{2,\infty}}\le M_2.
\]
The bounded-domain Gagliardo--Nirenberg estimate
\[
 \|v\|_{W^{1,\infty}}
 \le C\|v\|_2^{2/(d+4)}
       \|v\|_{W^{2,\infty}}^{(d+2)/(d+4)}+C\|v\|_2
\]
therefore yields
\begin{equation}
 \|x(t)-\overline x\|_{C^1(\overline\Omega)}
 \le C e^{-\kappa_{LS}t/(d+4)},
 \qquad t\ge0,
 \label{eq:C1-decay}
\end{equation}
after enlarging $C$ to cover the finite initial interval.  This completes the
regularity and long-time assertions of Theorem~\ref{thm:main}.
The equilibrium is uniquely determined by the conserved species means:
for a positive stationary strong solution the entropy is time-independent,
so its dissipation vanishes and $d_i>0$ forces every $\nabla x_i$ to vanish.
Connectedness of $\Omega$ makes every component constant, and its conserved
mean determines its value.

\section{Scope and fixed-frame characterization of the affine Darken class}
\label{sec:outlook}
For smooth uniformly positive tracer diffusivities, the entropy
symmetrization and normal ellipticity in
\Cref{sec:structure-local} yield entropy dissipation and local strong
solvability. These conclusions concern the molar-frame normalization
of \Cref{sec:inverse}. The constant-density barycentric system with
unequal molar masses is a separate problem.

The global argument uses the Stieltjes frame generated by the prescribed
pole list. At each interior composition, preservation of even one
Stieltjes eigendirection by the tangent mobility forces
$d(x)=\sigma(x)\one+\tau(x)\delta$, as shown in
\Cref{prop:fixed-frame-characterization}. In this frame the multi-affine
inverse map gives the mixed coefficient
$(\alpha_j+\alpha_k)/(z_j-z_k)$ in the divergence-form root equations.
The EPD identities cancel the mixed entropy production, and mixing entropy
supplies the transverse coercivity for localization. The modal coefficients are uniformly positive and bounded on the
invariant root rectangle. The scalar estimates permit $\alpha_j$ to
depend on $z_j$.

The theorem permits full composition dependence of $\sigma$ and $\tau$
subject to its regularity and positivity assumptions. The constant
multiplicative subclass follows by taking $\sigma=0$ and
$\tau=B_\delta^{-1}$, with the modal formula in
\Cref{rem:trace-determinant}. \Cref{app:repeated-factors} treats
coincident factors in that subclass by grouping species.
The prescription of Liu, Vlugt, and Bardow discussed in
\cite[Section~8]{BotheDruet2023}, where $1/d_i$ is a species-dependent
mole-fraction average of reciprocal dilute-limit diffusivities, can
fall outside the fixed-pole affine class.

For $N\ge4$, constant positive friction arrays can lie outside both the
additive class of \cite{BotheAdditive2026} and the constant multiplicative
class characterized in \Cref{cor:darken-constant-corner}. For four
distinct indices, the multiplicative form gives $f_{ij}f_{k\ell}=f_{ik}f_{j\ell}$,
whereas the additive form gives
\[
 f_{ij}f_{k\ell}-f_{ik}f_{j\ell}
   =(g_i-g_\ell)(g_k-g_j).
\]
This difference is nonzero for distinct additive generators.
Nonideal driving forces can change the principal diffusion frame even when
the mobility preserves the prescribed Stieltjes directions. These cases
leave open the use of other state variables or regularity arguments.

\appendix
% Keep appendix cross-references stable across cleveref versions.
\crefalias{section}{appendix}
\section{Additional identities for the Stieltjes map}
\label{app:audit}
Differentiating the inverse Stieltjes map \eqref{eq:inverse-root} gives
\[
 \partial_{z_j}x_i=-\frac{x_i}{\delta_i-z_j}.
\]
Therefore, for $j\ne k$,
\[
 \partial_{z_k}\partial_{z_j}x_i
 =\frac{x_i}{(\delta_i-z_j)(\delta_i-z_k)}
 =\frac{\partial_{z_j}x_i-\partial_{z_k}x_i}{z_k-z_j},
\]
whereas
\[
 \partial_{z_j}^2x_i=0.
\]
At a Stieltjes root $z_j$,
\[
 \sum_i\frac{x_i}{\delta_i-z_j}=0,
 \qquad
 \sum_i\frac{x_i}{(\delta_i-z_j)^2}=m_x'(z_j)>0.
\]
These identities give the eigenvector relation $A_0q_j=z_jq_j$ and
the EPD cancellation in the chain-rule production.

\paragraph{EPD Hessian symmetrizers.}
For $F\in C^2(\cT^\circ)$ write
$F(U(z))=\Psi(z)$, $P_j=\partial_{z_j}U$, and
$P=(P_1,\ldots,P_m)$.  Since $U$ is affine in $x$, the chain rule gives
\[
 P_j^{\sf T}D_U^2F\,P_k
 =\Psi_{jk}-\frac{\Psi_j-\Psi_k}{z_k-z_j}\quad(j\ne k),
 \qquad P_j^{\sf T}D_U^2F\,P_j=\Psi_{jj}.
\]
Thus \eqref{eq:EPD} makes $P^{\sf T}D_U^2F\,P$ diagonal.
Since the columns of $P$ form an eigenbasis of $A_M$, this gives
$D_U^2F\,A_M=A_M^{\sf T}D_U^2F$.
The Hessian is positive definite whenever every $\Psi_{jj}>0$.

\section{Constant multiplicative frictions with repeated factors}
\label{app:repeated-factors}
The grouping reduction closes because the two common tracer laws
factor through the group totals. For the constant multiplicative
frictions in \Cref{cor:constant-friction}, these laws are $\sigma=0$ and
$\tau=(\sum_a X_a/\eta_a)^{-1}$, where $X_a$ is the total mole fraction
of the species with factor $\eta_a$. For a general composition-dependent
law, dependence of the grouped fluxes on the group totals requires a
separate check.

\begin{proposition}[Grouping equal constant factors]
\label{prop:repeated-factor-extension}
Under the domain, exponent, and strictly positive compatible initial-data
hypotheses of \Cref{thm:main}, let
$\widehat\delta_i>0$ and
$d_i=\widehat\delta_i/B_{\widehat\delta}$, allowing repeated factors.
Then the full $N$-species system has a unique global strong solution,
uniform species separation, uniform positive-time classical regularity,
and the entropy, $L^2$, and $C^1$ relaxation conclusions of that theorem.
\end{proposition}
\begin{proof}
\emph{Closed group equations.}
Let $\eta_1<\cdots<\eta_K$ be the distinct values of the factors, and set
\[
 G_a=\{i:\widehat\delta_i=\eta_a\},\qquad
 X_a=\sum_{i\in G_a}x_i,\qquad
 B=\sum_{a=1}^K\frac{X_a}{\eta_a},\qquad
 a_a=\frac{\eta_a}{B},\qquad
 b=\frac1B\sum_{a=1}^K\eta_a\nabla X_a.
\]
Here $b$ is a spatial vector field. The direct flux formula gives
\begin{equation}\label{eq:grouped-species-equations}
 \partial_t x_i=\Div(a_a\nabla x_i-x_i b)\quad(i\in G_a),
 \qquad
 \partial_tX_a=\Div(a_a\nabla X_a-X_a b),
\end{equation}
with homogeneous Neumann conditions. Both $a_a$ and $b$ are determined by
the group totals and their gradients. If $K\ge3$, the group system is
affine Darken with distinct poles $\eta_a$, $\sigma=0$, and
$\tau=B^{-1}$, so \Cref{thm:main} gives its global solution.  If $K=2$, using $X_2=1-X_1$ gives
\[
 a_1\nabla X_1-X_1b
 =\frac{\eta_1X_2+\eta_2X_1}{B}\nabla X_1
 =\eta_1\eta_2\nabla X_1;
\]
thus $X_1$ solves the Neumann heat equation.  If $K=1$, $X_1\equiv1$.
In every case the group solution has a time-uniform positive lower bound
for each $X_a$ and uniform $C^{1+\beta/2,2+\beta}$ bounds on positive-time
unit slabs, for some $\beta>0$. For the heat equation these follow from
the scalar maximum principle and linear regularity. The initial group
data inherit the trace regularity and Neumann compatibility by summation.

\emph{Separation of the individual species.}
Let $I_+$ be the maximal initial interval of strict species positivity
for the solution supplied by \Cref{thm:local}. Uniqueness identifies its
group sums with the global group solution on $I_+$. On this interval put
$r_i=x_i/X_a$ for $i\in G_a$ and
$J_a=a_a\nabla X_a-X_ab$.  Subtracting $r_i$ times the group equation
from the species equation yields
\begin{align}
 X_a\partial_t r_i
 &=\Div(a_aX_a\nabla r_i)+J_a\cdot\nabla r_i,\notag\\
 \partial_t r_i-a_a\Delta r_i
 &=\bigl(\nabla a_a+2a_a\nabla\log X_a-b\bigr)\cdot\nabla r_i,
 \qquad \partial_\nu r_i=0.
 \label{eq:within-group-ratio}
\end{align}
On every finite time interval the coefficients are bounded and the diffusion
is uniformly positive, including at time zero by the continuous
$C^1$ trace of the group solution.  Equivalently,
\[
 \partial_t r_i-\Div(a_a\nabla r_i)
 =(2a_a\nabla\log X_a-b)\cdot\nabla r_i.
\]
Testing the upper and lower excursions gives the scalar maximum
principle. Young's inequality controls the bounded drift, and Gronwall
propagates the zero initial level energy. Consequently
\begin{equation}\label{eq:within-group-ratio-bounds}
 0<\min_{\overline\Omega}r_i(0)
 \le r_i(t,y)\le\max_{\overline\Omega}r_i(0)\le1.
\end{equation}
The reciprocal of $X_a$ and the quotient have the strong regularity needed
for these tests on compact intervals.  The lower bound in
\eqref{eq:within-group-ratio-bounds}, multiplied by the group lower bound,
excludes a first loss of positivity by continuity. Consequently the group
identification and ratio bounds hold throughout the full maximal interval,
and the full state stays in a compact interior simplex set.

\emph{Continuation from the group coefficient bounds.}
Suppose the full maximal time $T$ were finite.  On any fixed positive tail
before $T$, the global group solution gives uniform bounds for the coefficients of
\begin{equation}\label{eq:group-linear-nondivergence}
 \partial_t x_i-a_a\Delta x_i=b_a\cdot\nabla x_i+c_a x_i,
 \qquad b_a=\nabla a_a-b,\quad c_a=-\Div b.
\end{equation}
In particular $b_a,c_a$ are bounded, $a_a$ is uniformly spatially Lipschitz
and uniformly H\"older continuous in time, and $a_a$ is bounded away from
zero.  The frozen scalar Neumann realizations $-a_a(s,\cdot)\Delta$ have
uniform maximal regularity: the parameter-elliptic and boundary-symbol
checks of \Cref{lem:LS} apply in one component, and the compact-coefficient
argument of \Cref{lem:uniform-MR} gives uniform constants.

Write $Y_\gamma=B^{2-2/p}_{p,p,N}(\Omega)$ and
$\mathbb E_I^{\rm sc}=W^{1,p}(I;L^p)\cap L^p(I;W_N^{2,p})$.
On $I=(s,s+h)$, freezing \eqref{eq:group-linear-nondivergence} at $s$ and
using its uniform estimate gives
\[
 \|x_i\|_{\mathbb E_I^{\rm sc}}
 \le C\left(\|x_i(s)\|_{Y_\gamma}
      +\omega(h)\|x_i\|_{L^p(I;W^{2,p})}
      +\|\nabla x_i\|_{L^p(I\times\Omega)}
      +\|x_i\|_{L^p(I\times\Omega)}\right),
\]
where $\omega(h)\to0$ is the time modulus inherited from the group solution.
The interpolation inequality
\[
 \|\nabla w\|_{L^p}\le\varepsilon\|w\|_{W^{2,p}}
                         +C_\varepsilon\|w\|_{L^p}
\]
allows one first to choose $\varepsilon$ and then a common $h_0>0$ to
absorb both derivative terms.  Since $0\le x_i\le1$, the result is
\[
 \|x_i\|_{\mathbb E_I^{\rm sc}}
 \le C\bigl(\|x_i(s)\|_{Y_\gamma}+h^{1/p}\bigr),
 \qquad 0<h\le h_0.
\]
The augmented trace estimate \eqref{eq:uniform-augmented-trace-core} now gives
the finite affine recursion of \Cref{lem:terminal-trace}, for all components.
Its final interval is treated by a supremum over endpoints strictly below
$T$.  Thus the terminal trace norm is bounded.  The general Darken restart
in \Cref{thm:local}, together with the compact interior range already proved,
contradicts finiteness of $T$.

\emph{Uniform tail regularity and relaxation.}
Fix a waiting time $\tau_0>0$ and use the uniform group coefficient bounds
on $[\tau_0/4,\infty)$. Choose $h=\min\{h_0,\tau_0/4\}$, with $h_0$ from
the preceding absorption estimate for these bounds. On an interval
$I=(s,s+h)$ with $s\ge\tau_0/4$, set $v=\chi(t)x_i$, using a translated
smooth cutoff that vanishes near $s$ and equals one on the latter half.
Its equation has the same principal and lower-order terms as
\eqref{eq:group-linear-nondivergence}, with additional forcing
$\chi' x_i$ and zero initial trace. Freezing and interpolation yield
\[
 \|v\|_{\mathbb E_I^{\rm sc}}
 \le C\bigl(\|v\|_{L^p(I\times\Omega)}
             +\|\chi' x_i\|_{L^p(I\times\Omega)}\bigr)\le C,
\]
with a constant independent of $s$. The trace and mixed-derivative
embeddings in \eqref{eq:uniform-short-time-C1} give uniform space--time
H\"older bounds on the latter halves. Choosing consecutive starting
times spaced by $h/4$ covers $[\tau_0/2,\infty)$ with overlapping
latter halves.

The conormal argument of \Cref{sec:conormal} can now be applied in the
reduced concentration variables themselves.  Indeed
$u_t=\Div(C(u)\nabla u)$, and the exact positive symmetrizer $H(u)$ in
\eqref{eq:HC} is available for every positive Darken law.  On the compact
interior range, freezing $H(u_*)^{1/2}$ and choosing a fixed cylinder radius
using the H\"older modulus established above gives the same strongly parabolic
tensor and zero conormal boundary condition as in \Cref{lem:lipschitz}.
The quantitative first- and second-order estimates of
\Cref{lem:DJ-quantitative-first,prop:DJ-quantitative-second}, applied on
nested tails between $\tau_0/2$ and $\tau_0$ as in
\Cref{prop:uniform-tail-smoothing}, give uniform $C^{1+\beta/2,2+\beta}$
regularity on the tail from $\tau_0$.
Finally, the entropy and Fisher identities in \Cref{cor:entropy}, conservation
of the species integrals, and the logarithmic Sobolev and interpolation estimates
of \Cref{sec:longtime} apply unchanged to this uniformly positive
Darken core.  They give exponential entropy, $L^2$, and $C^1$ relaxation
to the conserved species means.  Uniqueness and local continuous dependence
are inherited from \Cref{thm:local}.
\end{proof}

\section{Quantitative conormal estimates in the unweighted H\"older case}
\label{app:conormal-quantitative}
\numberwithin{equation}{section}
\setcounter{equation}{0}

This appendix proves the quantitative conormal estimates used in
\Cref{sec:conormal} and Proposition~\ref{prop:uniform-tail-smoothing}.
In the journal version of \cite[Lemma~3.5, equation~(3.17)]{DongJeon2026},
the gradient estimate lists only the dimensions, ellipticity, weight
exponent, and an auxiliary integrability exponent in the dependence of its
constant. The example below shows that the constant must also depend on
coefficient oscillation, even for smooth scalar coefficients in the
unweighted case. We start from the oscillation bound before absorption in
\cite[Lemma~3.4, equation~(3.15) in the journal version]{DongJeon2026}
and track this dependence in both regularity estimates. The two estimates
below are for the unweighted H\"older case.

For this appendix only, let $n\ge2$ be the spatial dimension and $m\ge1$
the number of components.  Write
\[
 D_r^+=B_r'(0)\times(0,r),\qquad
 Q_r^+=(-r^2,0)\times D_r^+.
\]
The corresponding cylinder at $X=(t,x)$ is
\[
 Q_r^+(X)=(t-r^2,t)\times
 \bigl(B_r'(x')\times((x_n-r)_+,x_n+r)\bigr).
\]
When $r<x_n$ this is a full interior cylinder. For a boundary center
$(t_0,x_0)$ on the flat face, set
$\mathcal S_r(s,y)=(t_0+r^2s,x_0+ry)$. The rescaled unknown, coefficients,
and divergence datum are $v\circ\mathcal S_r$, $A\circ\mathcal S_r$,
and $r g\circ\mathcal S_r$, respectively. This scaling also multiplies
the conormal flux by $r$. On a fixed pair of nested radii, local estimates
transfer by this scaling and a finite covering. Let
$A=(A^{\gamma\delta})_{\gamma,\delta=1}^n$ be a real tensor of
$m\times m$ blocks with
\begin{equation}\label{eq:DJ-coercivity}
 \sum_{\gamma,\delta=1}^n
 \zeta_\gamma^{\sf T}A^{\gamma\delta}\zeta_\delta
 \ge\lambda\sum_{\gamma=1}^n|\zeta_\gamma|^2,
 \qquad |A|\le\Lambda,
 \qquad \lambda>0.
\end{equation}
No symmetry is assumed.  The equation and its conormal condition are
\begin{equation}\label{eq:DJ-linear-problem}
 \partial_\gamma(A^{\gamma\delta}\partial_\delta v)-v_t
   =\partial_\gamma g^\gamma\quad\hbox{in }Q_1^+,
 \qquad
 A^{n\delta}\partial_\delta v-g^n=0
   \quad\hbox{on }\{x_n=0\}.
\end{equation}
Repeated spatial indices are summed. We call
$v\in L^2((-1,0);H^1(D_1^+;\R^m))$ a weak energy solution if it satisfies
\[
 \int_{Q_1^+}\bigl[-v\cdot\varphi_t+
       (A^{\gamma\delta}D_\delta v-g^\gamma)
                 \cdot D_\gamma\varphi\bigr]=0
\]
for smooth tests supported away from the artificial spatial faces and
compactly supported in time. Tests may meet the flat face. This identity
includes the conormal condition. On each buffered half-cylinder $I\times D$,
write
\[
 V_D=\{\varphi\in H^1(D;\R^m):
     \varphi\text{ has zero trace on the artificial spatial faces}\}.
\]
By restriction and density the equation gives $v_t\in L^2(I;V_D')$
there. A spatial localization supported away from the artificial faces
therefore has an $L^2$ time trace. This permits integrated energy tests
at time slices. We use the usual parabolic H\"older norms.  In particular,
$C^{(1+\beta)/2,1+\beta}$ controls $v,Dv$, the
$C^{\beta/2,\beta}$ seminorm of $Dv$, and the
$C_t^{(1+\beta)/2}$ seminorm of $v$.

\begin{remark}[Necessity of coefficient-modulus dependence]
\label{rem:DJ-counterexample}
\normalfont
Take $n=2$, $m=1$, write $y=(y_1,y_2)$, $r=|y|$, and let
$0<\varepsilon<1$.  Define the stationary function and symmetric matrix
\[
 w_\varepsilon(y)=y_1(r^2+\varepsilon^2)^{-1/4},\qquad
 A_\varepsilon(y)=I+c_\varepsilon(r)
 \begin{pmatrix}y_2^2&-y_1y_2\\-y_1y_2&y_1^2\end{pmatrix},
\]
\[
 c_\varepsilon(r)=-\frac{2}{r^2+\varepsilon^2}
       +\frac{5r^2}{4(r^2+\varepsilon^2)^2}.
\]
Both are smooth, including at the origin.  With
$q=r^2/(r^2+\varepsilon^2)$, the radial and tangential eigenvalues of
$A_\varepsilon$ are $1$ and
\[
 b(q)=1-2q+\tfrac54q^2,
 \qquad b(q)-\tfrac15=\tfrac1{20}(5q-4)^2,
 \qquad 1-b(q)=q(2-\tfrac54q).
\]
Thus $\tfrac15I\le A_\varepsilon\le I$ for every $\varepsilon$.
In polar coordinates put
$R(r)=r(r^2+\varepsilon^2)^{-1/4}$, so that
$w_\varepsilon=R(r)\cos\vartheta$.  Direct differentiation gives
\[
 r^2R''+rR'-b(q)R=0.
\]
Consequently $\Div(A_\varepsilon Dw_\varepsilon)=0$.
On $y_2=0$, its normal flux is zero: there $A_\varepsilon$ is diagonal
and $\partial_{y_2}w_\varepsilon=0$.
On every fixed half-cylinder, $|w_\varepsilon|\le r^{1/2}$ gives a
uniform $L^\infty$, hence $L^1$, bound, whereas
\[
 \partial_{y_1}w_\varepsilon(0)=\varepsilon^{-1/2}.
\]
Continuity gives the same lower bound for the supremum of the gradient
on a half-cylinder meeting the origin.  Thus a gradient estimate for this family requires more than a common
ellipticity bound and fixed dimensions. The coefficients have finite
Dini and H\"older moduli for each $\varepsilon>0$. The spatial
$C^{\bar\beta}$ seminorm is at least
$(11/16)\varepsilon^{-\bar\beta}$, since $A_\varepsilon(0)=I$ and the
tangential eigenvalue at $(\varepsilon,0)$ equals $5/16$.
For the partial mean oscillation in \cite[Definition~1.1]{DongJeon2026},
take the cylinder of radius $s$ with spatial center $(0,s/2)$ and
top time $-1/4$, for $2\varepsilon\le s\le 1/4$. The tangential mean
of the off-diagonal entry vanishes at every normal height, because this
entry is odd in $y_1$. On the subrectangle
\[
 s/2\le y_1\le 3s/4,\qquad s/4\le y_2\le s/2
\]
one has $|(A_\varepsilon)_{12}(y)|\ge c>0$ uniformly in $s$ and
$\varepsilon$. Indeed,
$-c_\varepsilon(r)\ge 3/[4(r^2+\varepsilon^2)]$ on this set.
The subrectangle occupies a fixed positive fraction of the spatial
section, so $\eta_{A_\varepsilon}(s)\ge c_0>0$ on these scales. The partial
Dini integral therefore grows at least as
$c_0\log(1/(8\varepsilon))$ as $\varepsilon\downarrow0$.
The qualitative regularity assertion in
\cite[Theorem~4.1]{DongJeon2026} applies to each coefficient of the family.
\end{remark}

\begin{lemma}[First-order conormal estimate with coefficient-modulus dependence]
\label{lem:DJ-quantitative-first}
Let $0<\beta<\bar\beta<1$.  Suppose that \eqref{eq:DJ-coercivity}
holds and
\[
 \|A\|_{C^{\bar\beta/2,\bar\beta}(Q_1^+)}\le M,
 \qquad
 G:=\|g\|_{C^{\bar\beta/2,\bar\beta}(Q_1^+)}<\infty.
\]
Every weak energy solution of \eqref{eq:DJ-linear-problem} satisfies
\begin{equation}\label{eq:DJ-first-bound}
 \|v\|_{C^{(1+\beta)/2,1+\beta}(Q_{1/2}^+)}
 \le C\bigl(\|v\|_{L^2(Q_1^+)}+G\bigr),
\end{equation}
where $C$ depends only on $n,m,\lambda,\Lambda,M,\beta,\bar\beta$
and the fixed cylinder ratio.  In particular, $C$ does not depend on a
prior bound for $Dv$.
\end{lemma}
\begin{proof}
We first prove an a priori estimate when the gradient is bounded on
subcylinders that may meet the flat face but end strictly before time
zero.  This holds on the compact strong-solution strips used in the
main text.  To avoid assuming a bound at the excluded upper time face,
fix $T\in(-1/8,0)$ and, during the absorption step, use
\[
 Q_{a,T}^+:=Q_a^+\cap\{t<T\}.
\]
All comparison cylinders end at their centers' times and hence stay
below $T$.  The constants will be independent of $T$; only after
absorption do we let $T\uparrow0$.  The extension to weak energy
solutions is given at the end of the proof.

\emph{The pre-absorption estimate.}
Fix the auxiliary oscillation exponent $p_0=1/2$, and put
$\sigma=(1+\beta)/2$.  The partial mean oscillations of $A$ and $g$
are bounded, respectively, by constant multiples of
$Mr^{\bar\beta}$ and $Gr^{\bar\beta}$.
For either majorant $\omega$, the envelopes used in the oscillation
argument can be dominated by
\[
 \widetilde\omega(r)=\sum_{j=0}^\infty
      \kappa^{j\sigma}\omega(\min\{\kappa^{-j}r,1\}),
 \qquad
 \widehat\omega(r)=\sup_{r\le s\le1}
      (r/s)^\sigma\widetilde\omega(s),
\]
where $\kappa\in(0,1)$ is fixed by the structural constants and $\beta$.
Since $\beta<\bar\beta$ and $\beta<\sigma$, summing the geometric
series with ratio $\kappa^{\sigma-\beta}$ gives
\begin{equation}\label{eq:DJ-envelope-bound}
 \widehat\omega_A(r)+\int_0^r\frac{\widehat\omega_A(s)}s\,ds
       \le CMr^\beta,
 \qquad
 \widehat\omega_g(r)+\int_0^r\frac{\widehat\omega_g(s)}s\,ds
       \le CGr^\beta.
\end{equation}
Here and in the pre-absorption inequalities below, the constants do not
depend on $M$, $G$, or the solution.

Set $F^\gamma=A^{\gamma\delta}D_\delta v-g^\gamma$ and
$W=(D_1v,\ldots,D_{n-1}v,F^n)$.
In the unweighted case this is precisely the tangential-gradient/normal-flux
vector in the oscillation estimate of
\cite[Lemma~3.4]{DongJeon2026}.
Define its normalized $L^{p_0}$ mean oscillation by
\[
 \psi(X,s):=\inf_{c\in\mathcal V_{X,s}}
 \left(\frac{1}{|Q_s^+(X)|}
       \int_{Q_s^+(X)}|W-c|^{p_0}\right)^{1/p_0},
\]
where $\mathcal V_{X,s}=\mathbb R^{mn}$ when $s\le x_n/4$, and
$\mathcal V_{X,s}=\mathbb R^{m(n-1)}\times\{0\}$ when $s>x_n/4$.  That estimate, with the envelopes above as majorants,
states that for $0<\rho\le r$ on an admissible local cylinder,
\begin{align}
 \psi(X,\rho)
 &\le C(\rho/r)^\sigma
   \left(\frac{1}{|Q_{5r}^+(X)|}\int_{Q_{5r}^+(X)}|Dv|+\|g\|_\infty\right)\notag\\*
 &\quad+C\|Dv\|_{L^\infty(Q_{8r}^+(X))}
                \widehat\omega_A(\rho)
       +C\widehat\omega_g(\rho).
 \label{eq:DJ-preabs-oscillation}
\end{align}
All cylinders on the right are required to lie in the domain of the
solution.  The usual localization of this estimate follows by applying
it on the enclosing rescaled full or half-cylinder.
Choose minimizing constants $c_j$ for the oscillations at radii
$2^{-j}r$.  The inequality $|a+b|^{p_0}\le |a|^{p_0}+|b|^{p_0}$ and
the bounded volume ratios give
\[
 |c_{j+1}-c_j|\le C\bigl(\psi(X,2^{-j}r)
                         +\psi(X,2^{-j-1}r)\bigr).
\]
This remains valid when the admissible constant space changes at
$2^{-j}r=x_n/4$: both constants are compared with the same vector
field $W$ on the smaller cylinder.  Zero is admissible in both spaces,
so $|c_0|$ is controlled by the initial $L^{p_0}$ average of $W$.
Summing \eqref{eq:DJ-preabs-oscillation} therefore bounds the limit by
that average and the two Dini integrals in
\eqref{eq:DJ-envelope-bound}.  At Lebesgue points the limit is $W(X)$;
the subsequent oscillation estimate supplies its continuous
representative, including on the flat face.
The normal derivative is recovered from
\[
 D_nv=(A^{nn})^{-1}
       \left(F^n+g^n-\sum_{\delta<n}A^{n\delta}D_\delta v\right).
\]
Coercivity gives $\|(A^{nn})^{-1}\|\le\lambda^{-1}$.
Using Cauchy--Schwarz for the initial average yields, with
$q=(n+2)/2$,
\begin{equation}\label{eq:DJ-preabs-gradient}
 |Dv(X)|\le C_0r^{-q}\|Dv\|_{L^2(Q_{5r}^+(X))}
       +C_0G+C_0Mr^\beta
                     \|Dv\|_{L^\infty(Q_{8r}^+(X))}.
\end{equation}
This estimate retains the dependence on the coefficient modulus.

\emph{Absorption and its radius dependence.}
Put $E=\|v\|_{L^2(Q_1^+)}+G$ and
$\mathcal F_T(a)=\|Dv\|_{L^\infty(Q_{a,T}^+)}$.
The energy inequality for \eqref{eq:DJ-linear-problem}, with a cutoff
allowed to meet the flat face, gives
\[
 \|Dv\|_{L^2(Q_{15/16,T}^+)}\le CE,
\]
with a constant independent of $T\in(-1/8,0)$.
Fix $\epsilon=2^{-q-2}$ and choose
\begin{equation}\label{eq:DJ-absorption-radius}
 s_0=\min\left\{1,
       \left(\frac{\epsilon}{C_0(1+M)}\right)^{1/\beta}\right\}.
\end{equation}
For $3/4\le a<b\le7/8$, take
$r=c_\#s_0(b-a)$ in \eqref{eq:DJ-preabs-gradient}; the fixed geometric
constant $c_\#>0$ is small enough that all its enlarged cylinders lie
in $Q_{b,T}^+$.  Taking the essential supremum over $X\in Q_{a,T}^+$ gives
\begin{equation}\label{eq:DJ-gradient-recursion}
 \mathcal F_T(a)\le\epsilon\mathcal F_T(b)
       +Cs_0^{-q}(b-a)^{-q}E.
\end{equation}
Use $a_j=3/4+(1/8)(1-2^{-j})$ and iterate with
$(a,b)=(a_j,a_{j+1})$.  Here
$\epsilon2^q=1/4$ and $\mathcal F_T(7/8)<\infty$ for each fixed
$T<0$, so the terminal term tends to zero and the geometric series
converges.  This gives $\mathcal F_T(3/4)\le C(M)E$, independently
of $T$.  Letting $T\uparrow0$ by monotonicity yields
\begin{equation}\label{eq:DJ-gradient-bound}
 \|Dv\|_{L^\infty(Q_{3/4}^+)}\le C(M)E.
\end{equation}
In particular, the right-hand side contains no unknown gradient bound.
The constant therefore depends on the coefficient modulus through $s_0$.

\emph{Gradient and time H\"older regularity.}
Insert \eqref{eq:DJ-gradient-bound} into
\eqref{eq:DJ-preabs-oscillation}, using a fixed admissible outer radius
for centers in $Q_{5/8}^+$.  Since $\sigma>\beta$, this gives
$\psi(X,\rho)\le C(M)E\rho^\beta$ at all sufficiently small radii.
The same summation of successive minimizing constants shows that the
error of approximation to $W(X)$ is at most $C(M)E\rho^\beta$.
For two nearby centers, compare both constants with an enclosing
backward cylinder of comparable radius; its volume ratio is bounded.
This proves the $C^{\beta/2,\beta}$ estimate for $W$, up to the flat
face and the upper time face.  The displayed recovery formula for
$D_nv$, the H\"older bounds for $A,g$, and \eqref{eq:DJ-gradient-bound}
then give
\begin{equation}\label{eq:DJ-gradient-holder}
 \|Dv\|_{C^{\beta/2,\beta}(Q_{5/8}^+)}
       +\|F\|_{C^{\beta/2,\beta}(Q_{5/8}^+)}\le C(M)E.
\end{equation}

The improved time regularity of $v$ follows directly from the divergence
equation.  Choose an even smooth function
$\rho$ supported in $B_{1/4}(0)$ with integral one, and put
$\rho_r(y)=r^{-n}\rho(y/r)$.  For $x\in\overline{D_{1/2}^+}$ define
the signed averaging kernel
\[
 \phi_{r,x}(y)=\rho_r(y-x-re_n)
                  +r\partial_{y_n}\rho_r(y-x-re_n).
\]
Its support lies strictly inside the half-space, its integral is one,
all its first moments about $x$ vanish, and
$\|\phi_{r,x}\|_1\le C$, $\|D\phi_{r,x}\|_1\le C/r$.
For small fixed upper bounds on $r$, the support lies in $D_{5/8}^+$.
Let $m_r(t)=\int v(t,y)\phi_{r,x}(y)\,dy$.
Taylor's formula with \eqref{eq:DJ-gradient-holder} and the vanishing
moments gives
\[
 |m_r(t)-v(t,x)|\le C(M)Er^{1+\beta}.
\]
The weak equation and $\int D\phi_{r,x}=0$ give
\[
 |m_r'(t)|=
 \left|\int (F(t,y)-F(t,x))D\phi_{r,x}(y)\,dy\right|
 \le C(M)Er^{\beta-1}.
\]
Taking $r=|t-s|^{1/2}$ for nearby times proves
$|v(t,x)-v(s,x)|\le C(M)E|t-s|^{(1+\beta)/2}$.
Longer time differences are covered by a fixed finite chain of such
intervals.  Together with the spatial estimate and the $L^2$ control
of an average, this also bounds $\|v\|_\infty$ and proves
\eqref{eq:DJ-first-bound}.

To reach weak energy solutions, work on a fixed buffered half-cylinder
whose closure, apart from the flat and upper faces, lies in $Q_1^+$.
Choose smooth $A_j,g_j$ converging uniformly to $A,g$, with uniformly
bounded H\"older norms and uniform coercivity. Projection onto the closed
cylinder is Lipschitz separately in the space and time variables.
Extend each coefficient by this projection and mollify with a
nonnegative approximate identity. This preserves the coercivity
because \eqref{eq:DJ-coercivity} is convex in the tensor.

To preserve the weak equation and its boundary condition under approximation,
construct the correction $e_j$.  Set
\[
 f_j^\gamma=g_j^\gamma-g^\gamma
           +(A^{\gamma\delta}-A_j^{\gamma\delta})D_\delta v
\]
and solve the coercive energy problem
\[
 \partial_\gamma(A_j^{\gamma\delta}D_\delta e_j)-(e_j)_t
       =\partial_\gamma f_j^\gamma,
 \qquad A_j^{n\delta}D_\delta e_j-f_j^n=0
       \quad\hbox{on the flat face},
\]
with zero initial trace and zero Dirichlet trace on the artificial
lateral boundary.  Its Galerkin construction uses the space of $H^1$
functions vanishing on that artificial boundary.  Testing by $e_j$
gives
\[
 \|e_j\|_{L^\infty_tL^2_x}+\|De_j\|_{L^2}
 \le C\|f_j\|_{L^2}
 \le C\bigl(\|A_j-A\|_\infty\|Dv\|_{L^2}
                         +\|g_j-g\|_{L^2}\bigr)\longrightarrow0.
\]
The identity defining $f_j$ shows directly that $v_j:=v+e_j$ solves
\eqref{eq:DJ-linear-problem} with $A_j,g_j$, including its conormal
condition.  Its $L^2$ norm is uniformly bounded. To reach the preliminary bounded
gradient class in \cite[discussion following Theorem~3.1]{DongJeon2026},
we first establish local differentiability for each smooth approximation.
Use backward time quotients $q_h=(v_j(t)-v_j(t-h))/h$ on smaller
half-cylinders with fixed buffers at the artificial spatial faces and
lower time face. They satisfy the system with principal coefficient
$A_j(t)$ and divergence datum
\[
 f_{j,h}=\frac{g_j(t)-g_j(t-h)}h
       -\frac{A_j(t)-A_j(t-h)}h\,Dv_j(t-h),
\]
with the same datum in the conormal condition. Both $f_{j,h}$ in $L^2$
and $q_h$ in $L^2_t V_D'$ are bounded for fixed $j$ uniformly for small
$h$. The first bound uses the smooth time dependence of $A_j,g_j$ and the
second uses $(v_j)_t\in L^2_t V_D'$ from the weak equation. To obtain a
local $L^2$ bound for $q_h$, take nested radii $a<b<c$ with
$b=(a+c)/2$ on one such buffered patch. The cutoff energy inequality
bounds the $L^2$ gradient of $q_h$ on the $a$-cylinder by its $L^2$ norm
on the $b$-cylinder and the $L^2$ norm of $f_{j,h}$, with a cutoff cost
bounded by $C(c-a)^{-2}$ in the squared estimate. For a spatial cutoff
$\chi$ equal to one on $D_b$ and vanishing near the artificial faces of
$D_c$, use
\[
 \|\chi q_h\|_{L^2(D_c)}^2
 \le \eta\|\chi q_h\|_{V_{D_c}}^2
       +C\eta^{-1}\|\chi q_h\|_{V_{D_c}'}^2
 \qquad(0<\eta<1).
\]
The $V_{D_c}$ norm here is the $H^1$ norm, and multiplication by $\chi$
has norm $O((c-a)^{-1})$ on $V_{D_c}$ and on its dual. Combining these
inequalities gives, for $Y_h(a)=\|q_h\|_{L^2_t H^1_x(Q_a^+)}^2$ on this
patch,
\[
 Y_h(a)\le \zeta Y_h(c)+C_{j,\zeta}(c-a)^{-8}
 \qquad(0<\zeta<1),
\]
by choosing $\eta$ proportional to $\zeta(c-a)^4$. Iteration with
geometrically decreasing radius gaps and $\zeta 2^8<1$ gives the bound
on a smaller cylinder. The outer energy norm of $q_h$ is finite for
each fixed $h$ and disappears in the iteration. Passing to the limit
gives $(v_j)_t\in L^2_tH^1_x$ locally, up to the flat face. A tangential
quotient already has an $L^2$ bound from $Dv_j\in L^2$, so its local
$L^2$ gradient bound follows from the same cutoff energy test.
Iterate these two arguments for $w=\partial_t^b D_{x'}^\mu v_j$.
The divergence datum is the corresponding derivative of $g_j$ plus
products of derivatives of $A_j$ with gradients of derivatives of $v_j$
of total time and tangential order strictly less than $b+|\mu|$.
Its time or tangential difference quotients are controlled in $L^2$
by derivatives already obtained at order $b+|\mu|$ or below. The same
datum enters the conormal condition, since all differentiated tests
remain admissible on the flat face. This yields all time and tangential
derivatives in $L^2_tH^1_x$ on smaller truncated cylinders. The smooth
equation solves for $D_{nn}v_j$ using the invertible block $A_j^{nn}$.
Successively differentiating this normal identity controls higher
normal derivatives in $L^2$, with time and tangential derivatives
available at each step. Sobolev embedding on a further buffered
half-cylinder gives a finite bound for $Dv_j$ up to the flat face.
The preliminary bounds may depend on $j$ and on the fixed buffers.
The absorption estimate \eqref{eq:DJ-gradient-recursion} and the
subsequent H\"older estimate are uniform in $j$ and $T$.
Compactness on smaller cylinders, the energy convergence just proved,
and then $T\uparrow0$ give \eqref{eq:DJ-first-bound}.  The buffers are
fixed in advance and can be chosen to contain the cylinder in the stated estimate.
\end{proof}

\begin{proposition}[Second-order homogeneous conormal estimate]
\label{prop:DJ-quantitative-second}
Let $0<\beta<\bar\beta<1$, assume \eqref{eq:DJ-coercivity}, and suppose
\[
 \|A\|_{C^{(1+\bar\beta)/2,1+\bar\beta}(Q_1^+)}\le M.
\]
A weak energy solution of \eqref{eq:DJ-linear-problem} with $g=0$
satisfies
\begin{equation}\label{eq:DJ-second-bound}
 \|v\|_{C^{1+\beta/2,2+\beta}(Q_{1/2}^+)}
       \le C\|v\|_{L^2(Q_1^+)},
\end{equation}
with $C$ depending only on $n,m,\lambda,\Lambda,M,\beta,\bar\beta$
and the fixed cylinder ratios.  No time derivative of $A$ and no
previous bound for $D^2v$ or $v_t$ is required.
\end{proposition}
\begin{proof}
We use the quantitative first-order estimate in the following
finite-difference argument, which is the unweighted homogeneous
specialization of the higher-order argument in \cite[Theorem~4.2]{DongJeon2026}.
All estimates below are on successively smaller fixed half-cylinders.
Set $E=\|v\|_{L^2(Q_1^+)}$ and choose
$\beta<\theta<\bar\beta$.
Lemma~\ref{lem:DJ-quantitative-first}, applied on a fixed larger cylinder,
gives
\[
 \|v\|_{C^{(1+\theta)/2,1+\theta}(Q_{15/16}^+)}\le CE.
\]
The constants below may depend on $M$ and the indicated exponents, but
never on a difference step.

\emph{Tangential derivatives.}
For $\ell<n$, let
$v_{\ell,h}(t,x)=(v(t,x+he_\ell)-v(t,x))/h$ on buffered cylinders.
Its divergence datum is
$-((A(t,x+he_\ell)-A(t,x))/h)Dv(t,x+he_\ell)$ when the
principal coefficient is kept at $A(t,x)$.  It has a uniform $L^2$
bound because $D_\ell A$ is bounded, while
$\|v_{\ell,h}\|_{L^2}\le C\|Dv\|_{L^2}$.
The local energy estimate therefore bounds $Dv_{\ell,h}$ in $L^2$
uniformly on $Q_{29/32}^+$, including the flat face.  Weak passage to
the limit gives the energy regularity of $D_\ell v$ and the equation
\[
 \partial_\gamma(A^{\gamma\delta}D_\delta D_\ell v)
      -\partial_tD_\ell v=\partial_\gamma\bar g_\ell^\gamma,
 \qquad
 \bar g_\ell^\gamma=-(D_\ell A^{\gamma\delta})D_\delta v.
\]
Its conormal condition is
$A^{n\delta}D_\delta D_\ell v-\bar g_\ell^n=0$.
The datum $\bar g_\ell$ has a uniform $C^{\theta/2,\theta}$ bound
by $CE$.  Applying Lemma~\ref{lem:DJ-quantitative-first} with regularity
exponent $\theta$ for the coefficients and datum and solution exponent $\beta$ gives
\begin{equation}\label{eq:DJ-tangential-gain}
 \|D_\ell v\|_{C^{(1+\beta)/2,1+\beta}(Q_{7/8}^+)}\le CE
       \qquad(\ell<n).
\end{equation}
This estimate controls every second spatial derivative except
$D_{nn}v$.

\emph{Normalized time differences.}
For $0<h<h_0$, with a fixed sufficiently small $h_0>0$, define
\[
 \delta_h f(t,x)=f(t,x)-f(t-h,x),\qquad
 w_h=h^{-1/2}\delta_hv,\qquad A_h=h^{-1/2}\delta_hA.
\]
They are used only where both times lie in the preceding larger
cylinder.  Direct subtraction of the two weak equations gives
\begin{equation}\label{eq:DJ-time-difference-equation}
 \partial_\gamma(A^{\gamma\delta}(t)D_\delta w_h)
       -\partial_tw_h=\partial_\gamma\widehat g_h^\gamma,
 \qquad
 \widehat g_h^\gamma=-A_h^{\gamma\delta}D_\delta v(t-h),
\end{equation}
with the \emph{same} modified datum in the boundary condition,
$A^{n\delta}(t)D_\delta w_h-\widehat g_h^n=0$.

The elementary difference estimate
\[
 \|h^{-1/2}\delta_h f\|_{C^{\bar\beta/2,\bar\beta}}
       \le C\|f\|_{C^{(1+\bar\beta)/2,1+\bar\beta}}
\]
holds uniformly in $h$ on these buffered cylinders.  For time
increments, split into $|t-s|\le h$ and $|t-s|>h$.
For spatial increments, split into $|x-y|\le\sqrt h$ and
$|x-y|>\sqrt h$: in the first case use
$\|Df(t)-Df(t-h)\|_\infty\le C h^{\bar\beta/2}$ and integrate along
the segment; in the second use
$\|f(t)-f(t-h)\|_\infty\le C h^{(1+\bar\beta)/2}$.
These four estimates give the displayed bound without differentiating
$f$ in time.
It follows that $\|\widehat g_h\|_{C^{\theta/2,\theta}}\le CE$.
The first-order time regularity of $v$ also gives
$\|w_h\|_\infty\le CEh^{\theta/2}$.
Applying Lemma~\ref{lem:DJ-quantitative-first} to
\eqref{eq:DJ-time-difference-equation} therefore yields
\begin{equation}\label{eq:DJ-time-difference-bound}
 \|w_h\|_{C^{(1+\beta)/2,1+\beta}(Q_{3/4}^+)}\le CE
       \qquad(0<h<h_0).
\end{equation}

\emph{Existence and spatial regularity of the time derivative.}
For backward second differences $\delta_h^2 f=\delta_h(\delta_h f)$,
\eqref{eq:DJ-time-difference-bound} implies
\begin{equation}\label{eq:DJ-time-second-differences}
 \|\delta_h^2v\|_\infty\le CEh^{1+\beta/2},
 \qquad
 \|\delta_h^2Dv\|_\infty\le CEh^{(1+\beta)/2}.
\end{equation}
Put $q_h=\delta_hv/h$.  The exact identity
\[
 q_h-q_{h/2}=-\frac{\delta_{h/2}^2v}{h}
\]
and \eqref{eq:DJ-time-second-differences} show that the dyadic
quotients converge uniformly.  Their distributional limit is $v_t$,
so the continuous limit is independent of the initial difference step.  Summing the
geometric tail gives, for every sufficiently small $h>0$,
\[
 \|v_t-q_h\|_\infty\le CEh^{\beta/2}.
\]
The quotient at one fixed step bounds $\|v_t\|_\infty$.
For $t>s$ put $h=t-s$.  On a further fixed smaller cylinder,
\[
 \begin{split}
 |v_t(t,x)-v_t(s,x)|
 &\le CEh^{\beta/2}
       +h^{-1/2}|w_h(t,x)-w_h(s,x)|\\
 &\le CEh^{\beta/2}
 \end{split}
\]
by \eqref{eq:DJ-time-difference-bound}; larger time gaps are covered
by the supremum bound.  Thus
$\|v_t\|_\infty+[v_t]_{C_t^{\beta/2}}\le CE$.
The second inequality in \eqref{eq:DJ-time-second-differences} gives
$[Dv]_{C_t^{(1+\beta)/2}}\le CE$.
Indeed, for $0<a<1$ and a fixed buffered time step $H>0$, the
identity
\[
 \delta_hf=2^{-J}\delta_{2^Jh}f
       +\sum_{i=0}^{J-1}2^{-i-1}\delta_{2^ih}^2f,
 \qquad H/2<2^Jh\le H,
\]
shows that $|\delta_h^2f|\le Kh^a$ implies
\[
 |\delta_hf|\le C_a\bigl(K+H^{-a}\|f\|_\infty\bigr)h^a
\]
on a smaller interval.  The supremum term controls the coarse-step
remainder and cannot in general be omitted.  Apply this with $f=Dv$
and $a=(1+\beta)/2$; both terms on the right are bounded by $CE$.

Spatial H\"older continuity of $v_t$ follows directly from first
differences; no conversion of second spatial differences at the boundary
is necessary.  Equation~\eqref{eq:DJ-time-difference-bound} gives
\[
 \|\delta_s w_s\|_\infty\le CEs^{(1+\beta)/2},\qquad
 \|D\delta_s w_s\|_\infty\le CEs^{\beta/2}.
\]
The second estimate uses the time H\"older seminorm of $Dw_s$, not a
second spatial derivative of $v$.
Since $\delta_s^2v=\sqrt{s}\,\delta_sw_s$, for
$r=|x-y|$ the segment between $x$ and $y$ in the convex spatial
half-cylinder gives
\[
 \left|\frac{\delta_s^2v(t,x)-\delta_s^2v(t,y)}{s}\right|
 \le CE\min\{s^{\beta/2},\,r s^{(\beta-1)/2}\}.
\]
Fix $H\in(0,h_0)$ small enough for the buffered cylinders and put
$h=2^{-j}H$.  The exact dyadic identity yields
\[
 q_h=q_H+\frac12\sum_{i=0}^{j-1}
                    \frac{\delta_{2^ih}^2v}{2^ih}.
\]
The fixed-step function $q_H$ is spatially Lipschitz, with bound $CE$
(the dependence on $H$ is included in $C$).
For $r^2<H$, split the dyadic steps $s=2^ih$ at $r^2$ to obtain
\[
 \begin{split}
 |q_h(t,x)-q_h(t,y)|
 &\le CE\left(r+\sum_{s\le r^2}s^{\beta/2}
                 +r\sum_{s>r^2}s^{(\beta-1)/2}\right)\\
 &\le CE r^\beta.
 \end{split}
\]
Both sums are geometric: $\beta>0$ controls the small steps and
$\beta<1$ controls the large ones.  The estimate is independent of
$j$.  Letting $h\downarrow0$ proves
$|v_t(t,x)-v_t(t,y)|\le CE|x-y|^\beta$; for $r^2\ge H$ use the
supremum bound.  The same straight-segment argument applies to points
on the flat face by continuity.  Together these estimates prove
\begin{equation}\label{eq:DJ-time-holder}
 \|v_t\|_{C^{\beta/2,\beta}}+[Dv]_{C_t^{(1+\beta)/2}}\le CE
\end{equation}
on the required smaller cylinder.  All time differences were backward;
no extension beyond the upper time face was used.

\emph{The remaining normal derivative.}
The homogeneous equation first gives, in distributions,
\[
 D_n(A^{nn}D_nv)=v_t-
   \sum_{(\gamma,\delta)\ne(n,n)}
        D_\gamma(A^{\gamma\delta}D_\delta v).
\]
All terms on the right are in $C^{\beta/2,\beta}$ with norm at most
$CE$, by \eqref{eq:DJ-tangential-gain}, \eqref{eq:DJ-time-holder}, and
the coefficient bounds. When $\gamma=n$ and $\delta<n$, the mixed
derivative is $D_nD_\delta v$ from
\eqref{eq:DJ-tangential-gain}. The continuous normal term
$A^{nn}D_nv$ therefore has a continuous derivative in the normal
variable, as follows also by one-dimensional integration on normal
lines. The spatial product rule yields
\[
 D_{nn}v=(A^{nn})^{-1}
      \bigl[D_n(A^{nn}D_nv)-(D_nA^{nn})D_nv\bigr].
\]
Coercivity gives $\|(A^{nn})^{-1}\|\le\lambda^{-1}$, and the inverse
has the required H\"older regularity. Hence $D_{nn}v$ has a
$C^{\beta/2,\beta}$ bound of size $CE$, including its extension to the
flat face. Together with the preceding estimates this proves
\eqref{eq:DJ-second-bound}.
The tangential energy difference quotients and the backward time
differences justify every step for the original weak energy solution;
no prior second-order regularity or time derivative of $A$ is assumed.
\end{proof}

\paragraph{Application to the affine Darken system.}
For $g=0$, the $L^2$ norms on the right of
\eqref{eq:DJ-first-bound} and \eqref{eq:DJ-second-bound} are bounded by
the $L^\infty$ norm times the fixed volume factor.
The tensor in \eqref{eq:core-flattened-system} satisfies
\eqref{eq:DJ-coercivity}; its coefficient H\"older norm is uniformly
bounded after the frozen similarity and the fixed boundary charts.
The first such bound follows from \eqref{eq:u-holder}, not from an
a-priori gradient bound.  The second follows from
\eqref{eq:first-conormal-gain} and the $C^{2+\alpha}$ boundary geometry.
Thus the radius in \eqref{eq:DJ-absorption-radius} is uniform along each
positive-time tail.  The additional-coordinate construction of
\Cref{sec:conormal} gives the interior estimates with the same type of
dependence.  The same statements apply to the reduced concentration
tensor in \Cref{app:repeated-factors}, after its independent H\"older
bound has been established.  In all three uses, the uniform coefficient
bound is established before the estimate is applied, and only
backward cylinders are required.

\paragraph{AI assistance.}
OpenAI ChatGPT was used for manuscript revision and proof checking. The author takes full responsibility for the content and correctness of the manuscript.

\end{document}